\documentclass[sn-mathphys]{sn-jnl-mod}

\normalbaroutside

\usepackage{amsmath,amssymb,amsthm}
\usepackage{enumitem}
\usepackage[export]{adjustbox} 
\usepackage{array}

\theoremstyle{thmstyleone}%

\newtheorem{thm}{Theorem}[section]

\newtheorem{cor}[thm]{Corollary}
\newtheorem{prop}[thm]{Proposition}

\theoremstyle{remark}
\newtheorem{rem}[thm]{Remark}

\newcommand{\bo}[1]{{\bf #1}}
\graphicspath{{./pics/}}

\newcommand{\Per}{\operatorname{Per}}

\newcommand{\di}{\operatorname{div}}

\newcommand{\Id}{\operatorname{\bo{Id}}}
\newcommand{\diam}{\operatorname{diam}}

\newcommand{\txtb}{\textcolor{blue}} 
\newcommand{\txtr}{\textcolor{red}} 

\begin{document}

\title[Optimization of the Steklov-Lam\'e eigenvalues with respect to the domain]{Optimization of the Steklov-Lam\'e eigenvalues with respect to the domain}

\author[1]{\fnm{Pedro R.S.} \sur{Antunes}}\email{prantunes@fc.ul.pt}

\author*[2]{\fnm{Beniamin} \sur{Bogosel}}\email{beniamin.bogosel@polytechnique.edu}

\affil[1]{\orgdiv{Departamento de Matem\'{a}tica}, \orgname{Instituto Superior T\'{e}cnico, Universidade de Lisboa}, \orgaddress{\street{Av. Rovisco Pais 1}, \city{Lisboa}, \postcode{P-1049-001}} and \orgdiv{Grupo de F\'{i}sica Matem\'{a}tica}, \orgname{Faculdade de Ci\^{e}ncias, Universidade de Lisboa}, \orgaddress{\street{Campo Grande, Edif\'{i}cio C6}, \city{Lisboa}, \postcode{P-1749-016}, \country{Portugal}}}

\affil[2]{\orgdiv{Centre de Math\'ematiques Appliqu\'ees}, \orgname{Ecole Polytechnique}, \orgaddress{\street{Rue de Saclay}, \city{Palaiseau}, \postcode{91128}, \country{France}}}

\abstract{
	This work deals with theoretical and numerical aspects related to the behavior of the Steklov-Lam\'e eigenvalues on variable domains. After establishing the eigenstructure for the disk, we prove that for a certain class of Lam\'e parameters, the disk maximizes the first non-zero eigenvalue under area or perimeter constraints in dimension two. Upper bounds for these eigenvalues can be found in terms of the scalar Steklov eigenvalues, involving various geometric quantities. We prove that the Steklov-Lam\'e eigenvalues are upper semicontinuous for the complementary Hausdorff convergence of $\varepsilon$-cone domains and, as a consequence, there exist shapes maximizing these eigenvalues under convexity and volume constraints. A numerical method based on fundamental solutions is proposed for computing the Steklov-Lam\'e eigenvalues, allowing to study numerically the shapes maximizing the first ten non-zero eigenvalues. 
}



\maketitle

\section{Introduction}
Given an open, bounded, connected Lipschitz domain consider the Steklov eigenvalue problem
\begin{equation}
\left\{\begin{array}{rcll}
-\Delta u & = & 0 & \text{ in }\Omega \\
\nabla u \cdot \bo n & = & \sigma_n(\Omega) u & \text{ on } \partial \Omega,
\end{array}\right.
\label{eq:steklov-eigs}
\end{equation} 
where $\bo n$ is the outer unit normal vector to $\partial \Omega$. 
It is known that the Steklov spectrum consists of a sequence of eigenvalues of the form 
\[ 0=\sigma_0(\Omega) < \sigma_1(\Omega) \leq ... \to +\infty.\]
The study of optimization problems related to Steklov eigenvalues was initiated by the works of Weinstock \cite{weinstock} and Hersch, Payne and Schiffer \cite{hersch-payne-schiffer}. Recently, there have been many works related to the study of these eigenvalues, as indicated in the survey paper \cite{survey-girouard-polterowich}. The sloshing behavior of a liquid in a cup has been related to problem in  \eqref{eq:steklov-eigs} in \cite{sloshing}. The Steklov-Neumann problem, consisting of adding some boundary parts with Neumann boundary condition in \eqref{eq:steklov-eigs}, has been studied in \cite{ammari-nigam}. It is shown that the corresponding equation models the behavior of a liquid in a container with immovable parts on its surface. 

Weinstock proved in \cite{weinstock} that $\sigma_1(\Omega)$ is maximized by the disk among simply connected two dimensional sets with fixed perimeter. Numerical observations made in \cite{Bogosel2} show that adding a small hole and rescaling to have prescribed perimeter may increase the Steklov eigenvalue. Therefore, simple connectedness is essential for Weinstock's result. Brock proved in \cite{brock} that $\sigma_1(\Omega)$ is maximized by the ball under volume constraint in any dimension. In \cite{hersch-payne-schiffer}  Hersch, Payne and Schiffer provided various upper bounds for functionals depending on the Steklov spectrum, equality being attained for the disk in many of them. One particularity of all these results is that direct proofs are given that the disk is optimal.

More recently the question of existence of solutions for problems depending on the Steklov eigenvalues was investigated. One key ingredient is understanding the semi-continuity properties for the Steklov eigenvalues when the domain changes. In \cite{Bogosel} existence of maximizers was proved for convex shapes and for shapes verifying an $\varepsilon$-cone property. This result was generalized in \cite{bogosel-bucur-giacomini} to general domains under volume constraint using a relaxed formulation. Numerical methods were developed in \cite{Bogosel2}, \cite{osting-steklov} for studying shapes maximizing $\sigma_k(\Omega)$ given some $k \geq 1$. 

Recently in \cite{Sebastian} the Steklov-Lam\'e problem was investigated, which is the analogue of problem \eqref{eq:steklov-eigs} in the setting of linearized elasticity. The precise definition of the Steklov-Lam\'e eigenvalues and the resulting properties are recalled in the next section. The objective of this work is to investigate theoretically and numerically the maximizers of the Steklov-Lam\'e eigenvalues. Although the questions we ask are natural, by analogy to the scalar Steklov problem, the techniques are more involved, reflecting the difficulties raised by the vectorial context.

In this work, we will also address the numerical	shape optimization of Steklov-Lam\'{e} eigenvalues using the Method of Fundamental Solutions (MFS) as forward solver. The MFS approximation is based on shifts of the fundamental solution of the PDE to some points placed at the exterior of the domain. Thus, by construction, the MFS approximation satisfies the PDE of the problem and the approximation is usually justified by density results. The MFS is a mesh and integration free method and typically presents very fast convergence when applied to smooth shapes. For details about the MFS we refer to the following works \cite{Alves,Alves-Antunes_2013,Bogomolny,FK}.

{\bf Structure of the paper.} In Section \ref{sec:properties} we compute the Steklov-Lam\'e eigenstructure of the disk for all ranges of admissible Lam\'e parameters and we establish an analogue of the Weinstock inequality \cite{weinstock} for a certain range of parameters. In Section \ref{sec:existence} we investigate the behavior of the Steklov-Lam\'e eigenvalues on moving domains. In particular, we show that there exist maximizers for the Steklov-Lam\'e eigenvalues in the class of convex shapes with fixed volume. In Section \ref{sec:moler-payne} we prove a result inspired by Moler and Payne \cite{moler-payne} related to changes in the solution of a PDE related to the Steklov-Lam\'e problem when the boundary conditions are verified in an approximate way. This result justifies the use of the MFS to approximate the Steklov-Lam\'e eigenvalues, presented in Section \ref{sec:num-methods}. Numerical results related to the optimization of the eigenvalues are shown in Section \ref{sec:num-results}.

\section{The Steklov-Lam\'e eigenvalues}
\label{sec:properties}
\subsection{Definitions and main properties} In the following, we use regular lower case fonts for scalar functions and bold lower case fonts for vectorial functions. Most of the results presented in this paper are valid in arbitrary dimensions. The eigenvalues of the disk and the numerical simulations are related to dimension $d=2$. For simplicity, denote $\bo H^1(\Omega) = (H^1(\Omega))^d$ and $\bo H_0^1(\Omega) = (H_0^1(\Omega))^d$. We use the same type of notations for $L^2$ spaces: bold case refers to vectorial elements of the proper dimension. The scalar product of two vectors $\bo x, \bo y$ is denoted by $\bo x\cdot   \bo y$. The matrix scalar product of two matrices $\bo S=(s_{ij})_{1\leq i,j \leq n}$ and $\bo T=(t_{ij})_{1\leq i,j\leq n}$ is denoted by $\bo S:\bo T = \sum_{i,j=1}^n s_{ij}t_{ij}$.

Consider a Lipschitz domain $\Omega\subset\Bbb{R}^d$. Note that more general domains for which the Steklov-Lam\'e eigenvalues are defined could be considered, as underlined in \cite{Sebastian}. Consider the solution $\bo u \in \bo H^1(\Omega)$ of the problem
\begin{equation}
\left\{ \begin{array}{rcll}
-\di A(e(\bo u)) & = & 0 & \text{ in } \Omega \\
Ae(\bo u)\bo n & = & \Lambda(\Omega) \bo u & \text{ on } \partial \Omega,
\end{array}\right.
\label{eq:steklov-lame}
\end{equation}
where $e(\bo u) = \frac{1}{2}( \nabla \bo u+ \nabla \bo u^T)$ is the usual symmetrized gradient and the material properties are given by Hooke's law $A\xi = 2\mu \xi +\lambda \text{tr}(\xi) \Id$. The parameters $\mu>0$ and $\lambda$ are called the Lam\'e coefficients and they are assumed to satisfy the condition $\lambda+\frac{2}{d}\mu>0$. The Jacobian of $\bo u$ is denoted by $\nabla \bo u$ and $\Id$ denotes the identity matrix. The spectral problem \eqref{eq:steklov-lame} was studied in \cite{Sebastian} where it is proved that under the hypotheses stated above, the spectrum of this problem consists of an increasing sequence of non-negative eigenvalues. It is straightforward to observe that the problem \eqref{eq:steklov-lame} is equivalent to the variational formulation
\begin{equation}
\int_\Omega Ae(\bo u): e(\bo v) = \Lambda(\Omega) \int_{\partial \Omega} \bo u \cdot \bo v \ \ \ \text{ for every } \bo v \in \bo H^1(\Omega).
\label{eq:var-form}
\end{equation} 

The space of rigid motions $\bo R(\Omega)$ is defined (as in \cite{Sebastian}) as the set of functions $\bo v \in \bo H^1(\Omega)$ such that $e(\bo v)=0$. It is a classical result that for a connected open domain $\Omega$ we have
\begin{equation}
\bo R(\Omega) = \{\bo v \in \bo H^1(\Omega) :  \bo v(x) = a+Bx, a\in \Bbb{R}^d, B \in \Bbb{R}^{d\times d}, B^T=-B\}.
\label{eq:zeri-eigenfunctions}
\end{equation}
One can observe that $\dim \bo R(\Omega) = \frac{d(d+1)}{2}$. All elements in $\bo R(\Omega)$ verify $e(\bo u)=0$. Therefore all rigid motions are eigenfunctions for \eqref{eq:steklov-lame} associated to a zero eigenvalue. Conversely, any eigenfunction $\bo u$ associated to the zero eigenvalue verifies $e(\bo u) = 0$ in $\Omega$. 

In view of the previous considerations, and the results in \cite{Sebastian}, the Steklov-Lam\'e spectrum of a connected Lipschitz domain $\Omega$ is given by
\[ 0 = \Lambda_{0,1}(\Omega) = ... = \Lambda_{0,\frac{d(d+1)}{2}}(\Omega) < \Lambda_1(\Omega) \leq \Lambda_2(\Omega)\leq  ... \to +\infty.\]
In view of the variational formulation \eqref{eq:var-form}, it is classical that the eigenvalues can be characterized using Rayleigh quotients
\begin{equation}
\Lambda_n(\Omega) = \min_{\bo S_{n}\subset \bo H^1(\Omega)}
\max_{\bo u \in \bo S_n\setminus\bo H_0^1(\Omega)} \frac{\int_\Omega Ae(\bo u):e(\bo u)}{\int_{\partial \Omega} |\bo u|^2}
\label{eq:rayleigh}
\end{equation}
where the minimum is taken over all subspaces $\bo S_{n}$ of $\bo H^1(\Omega)$ having dimension $n+\frac{d(d+1)}{2}$. Denote for each $n\geq 1$ by $\bo u_n\in \bo H^1(\Omega)$ an eigenfunction associated to the eigenvalue $\Lambda_n(\Omega)$. It is immediate to observe that if $\bo u_i$ and $\bo u_j$ are associated to the different eigenvalues $\Lambda_i(\Omega) \neq \Lambda_j(\Omega)$ then \eqref{eq:var-form} implies that
\[ \Lambda_i(\Omega) \int_{\partial \Omega} \bo u_i \cdot \bo u_j = \int_\Omega Ae(\bo u_i): e(\bo u_j) =  \int_\Omega Ae(\bo u_j): e(\bo u_i)=\Lambda_j(\Omega) \int_{\partial \Omega} \bo u_i \cdot \bo u_j.\]
As a direct consequence $\int_{\partial \Omega} \bo u_i \cdot \bo u_j = 0$. It is natural to assume that the eigenfunctions $\bo u_n, n \geq 1$ form an orthonormal family when restricted to $\bo L^2(\partial \Omega)$. We make this assumption in the rest of the article. Another direct consequence of \eqref{eq:var-form} is 
\[ \int_{\partial \Omega} \bo u_n \cdot \bo r= 0,\]
for every $n\geq 1$ and $\bo r \in \bo R(\Omega)$, i.e. eigenfunctions associated to $\Lambda_n(\Omega)$ with $n \geq 1$ are orthogonal in $\bo L^2(\partial \Omega)$ to all rigid motions.

\begin{rem} 
	It is possible to express the eigenvalues of \eqref{eq:steklov-lame} using Rayleigh quotients for subspaces of dimension $n$ in $\bo H^1(\Omega)$ which are orthogonal to $\bo R(\Omega)$ in $\bo L^2(\partial \Omega)$. However, the formulation \eqref{eq:rayleigh} is more practical for the theoretical questions that will be answered later in the paper.
\end{rem}



In the following, in order to underline the dependence of the eigenvalue on the shape $\Omega$ and on the parameters $\lambda,\mu$, denote by $\Lambda_n(\Omega,\lambda,\mu)$ an eigenvalue of \eqref{eq:steklov-lame} for a certain pair of Lam\'{e} parameters. Then we have the following result concerning the scaling of the eigenvalues with respect to the parameters.
\begin{prop}
	{\rm (i)}  Scaling with respect to homotheties:
	\begin{equation}\label{eq:scaling-homotheties}
	\Lambda_n(t\Omega,\lambda,\mu) = \frac{1}{t} \Lambda_n(\Omega,\lambda,\mu) \text{ for any } t>0.
	\end{equation}
	
	{\rm (ii)} Scaling of the Lam\'e parameters:
	\begin{equation}
	\label{multpar}
	\Lambda_n(\Omega,\alpha\lambda,\alpha\mu)=\alpha\Lambda_n(\Omega,\lambda,\mu),\ \forall\alpha>0
	\end{equation}
	\label{prop:scaling}
\end{prop}
\begin{proof}(i) is a direct consequence by a change of variables. (ii) is a consequence of the linearity of \eqref{eq:steklov-lame}.
\end{proof} 

In this work we will consider the shape optimization problems
\begin{equation}\label{shoptprob}
\Lambda_n^*(\Omega,\lambda,\mu):=\sup \Big\{\Lambda_n(\Omega,\lambda,\mu), \Omega\subset\mathbb{R}^d:|\Omega|=1\Big\}.
\end{equation}
and
\begin{equation}
\label{shoptprobconv}
\Lambda_n^{\#}(\Omega,\lambda,\mu):=\sup\left\{\Lambda_n(\Omega,\lambda,\mu), \Omega\subset\mathbb{R}^d,\ \Omega\ \text{convex},\ |\Omega|=1\right\}.
\end{equation}

Later on, we will show that problem \eqref{shoptprobconv} has a solution, implying that the supremum could be replaced by the maximum. Numerical simulations will be performed to approximate solutions to problems \eqref{shoptprob} and \eqref{shoptprobconv}, indicating that optimal shapes are likely to exist also for problem \eqref{shoptprob}. This is in accord with theoretical and numerical observations for the maximization of the scalar Steklov eigenvalues \cite{osting-steklov}, \cite{bogosel-bucur-giacomini}, however, the general theory of existence is not completely established not even in the scalar case, when only a volume constraint is present.

\subsection{The disk} In this section we focus on the case of the disk in dimension $d=2$ and we derive the closed form of the eigenvalues and eigenfunctions. This will be useful for having a benchmark for the numerical approximation method and also will allow to answer partially some questions regarding the maximality of the disk for the first non-zero eigenvalue. We introduce polar coordinates
\[\bo u(r,\theta)=u_r(r,\theta)\bo e_r+u_\theta(r,\theta)\bo e_\theta,\] 
where
\[\bo e_r=\cos(\theta)\bo e_1+\sin(\theta)\bo e_2\ \text{and}\ \bo e_\theta=-\sin(\theta)\bo e_1+\cos(\theta)\bo e_2.\]

We consider $\bo u$ defined by a Fourier expansion
\begin{equation}
\label{solu}
\bo u(r,\theta)=\begin{bmatrix}c_0^r(r)\\ c_0^\theta(r)\end{bmatrix}+\sum_{n=1}^\infty\begin{bmatrix}c_n^r(r)\\ c_n^\theta(r)\end{bmatrix}\cos(n\theta)+\sum_{n=1}^\infty\begin{bmatrix}s_n^r(r)\\ s_n^\theta(r)\end{bmatrix}\sin(n\theta)
\end{equation}
and search for solutions of the partial differential equation $\di A(e(\bo u))  =  0$, which implies that we have (cf. ~\cite{VMFG})
\begin{equation}
\begin{array}{c}
c_0^r(r)=A_0 r \\
c_0^\theta(r)=B_0r,
\end{array}
\label{eq:n0}
\end{equation}

\begin{equation}
\begin{array}{c}
c_1^r(r)=-A_1^0+A_1 \left(\frac{-\lambda+\mu}{\lambda+\mu}\right) r^2\\
c_1^\theta(r)=B_1^0+B_1\left(\frac{3\lambda+5\mu}{\lambda+\mu}\right)r^2\\
s_1^r(r)=B_1^0-B_1  \left(\frac{-\lambda+\mu}{\lambda+\mu}\right) r^2\\
s_1^\theta(r)=A_1^0+A_1\left(\frac{3\lambda+5\mu}{\lambda+\mu}\right)r^2\\
\end{array}
\label{eq:n1}
\end{equation}
and
\begin{equation}
\begin{array}{c}
c_n^r(r)=-A_n^0 r^{n-1}+A_n \left(\frac{-n\lambda-(n-2)\mu}{n(\lambda+\mu)}\right) r^{n+1}\\
c_n^\theta(r)=B_n^0r^{n-1}+B_n\left(\frac{(n+2)\lambda+(n+4)\mu}{n(\lambda+\mu)}\right)r^{n+1}\\
s_n^r(r)=B_n^0r^{n-1}-B_n \left(\frac{-n\lambda-(n-2)\mu}{n(\lambda+\mu)}\right)  r^{n+1}\\
s_n^\theta(r)=A_n^0r^{n-1}+A_n\left(\frac{(n+2)\lambda+(n+4)\mu}{n(\lambda+\mu)}\right)r^{n+1}\\
\end{array},\ n=2,3,...
\label{eq:ngen}
\end{equation}
for some constants $A_i, B_i,\ i=0,1,...$ and $A_i^0,B_i^0,\ i=1,2,...$ Moreover, as shown in~\cite{VMFG}, for a solution of type \eqref{solu} in the disk we have
\begin{align*}Ae(\bo u)\bo n(r)&=\begin{bmatrix}(\lambda+2\mu)c_0^{r}\ '(r)+\frac{\lambda}{r}c_0^r(r)\\ \mu\left(c_0^\theta\ '(r)-\frac{1}{r}c_0^\theta(r)\right)\end{bmatrix}\\
&+\sum_{n=1}^\infty\begin{bmatrix}(\lambda+2\mu)c_n^{r}\ '(r)+\frac{\lambda}{r}c_n^r(r)+\frac{n\lambda}{r}s_n^\theta(r)\\ \mu\left(\frac{n}{r}s_n^r(r)+c_n^\theta\ '(r)-\frac{1}{r}c_n^\theta(r)\right)\end{bmatrix}\cos(n\theta)\\
&+\sum_{n=1}^\infty\begin{bmatrix}(\lambda+2\mu)s_n^{r}\ '(r)+\frac{\lambda}{r}s_n^r(r)-\frac{n\lambda}{r}c_n^\theta(r)\\ \mu\left(-\frac{n}{r}c_n^r(r)+s_n^\theta\ '(r)-\frac{1}{r}s_n^\theta(r)\right)\end{bmatrix}\sin(n\theta).
\end{align*}

\begin{thm}
	\label{thm:eigdisk}
	The Steklov-Lam\'{e} spectrum of a disk of radius equal to $R$ is the sorted list of the following real numbers:
	\begin{enumerate}[label=\upshape{(\roman*)}]
		\item $0$ (with multiplicity 3), 
		\item $\frac{2(\lambda+\mu)}{R},$
		\item $\frac{4\mu(\lambda+\mu)}{(\lambda+3\mu)R}$ (counted twice) and
		\item $\frac{2\mu(n-1)}{R}$ (counted twice), for $n=2,3,...$ and 
		\item $\frac{2(n+1)\mu(\lambda+\mu)}{(\lambda+3\mu)R}$ (counted twice), for $n=2,3,...$
	\end{enumerate}
	The eigenfunctions in each of the previous cases are linear combinations of the following sets of functions
	\begin{enumerate}[label=\upshape{(\roman*)}]
		\item $\left\{(1,0),\ (0,1),\ r(-\sin(\theta),\cos(\theta))\right\}$
		\item $\left\{r(\cos(\theta),\sin(\theta))\right\}$
		\item $
		\Big\{\Big(2(R^2-r^2)+\frac{(\lambda+3\mu)r^2\cos(2\theta)}{\lambda+\mu},\frac{(\lambda+3\mu)r^2\sin(2\theta)}{\lambda+\mu}\Big)$, 
		$\Big(\frac{(\lambda+3\mu)r^2\sin(2\theta)}{\lambda+\mu},2(R^2-r^2)-\frac{(\lambda+3\mu)r^2\cos(2\theta)}{\lambda+\mu}\Big)\Big\}$
		\item $\left\{r^{n-1}\left(\cos((n-1)\theta),-\sin((n-1)\theta)\right),r^{n-1}\left(\sin((n-1)\theta),\cos((n-1)\theta)\right)\right\}$
		\item $\left\{(f_1(r,\theta),f_2(r,\theta)),(f_3(r,\theta),f_4(r,\theta))\right\},$
		where
			\end{enumerate}
		$ f_1(r,\theta)=\frac{r^{n-1}}{(\lambda+\mu)n}\left(-(\lambda+\mu)(n+1)(r^2-R^2)\cos((n-1)\theta)+(\lambda+3\mu)r^2\cos((n+1)\theta)\right),$\newline
		$ f_2(r,\theta)=\frac{r^{n-1}}{(\lambda+\mu)n}\left((\lambda+\mu)(n+1)(r^2-R^2)\sin((n-1)\theta)+(\lambda+3\mu)r^2\sin((n+1)\theta)\right),$\newline
		$ f_3(r,\theta)=\frac{r^{n-1}}{(\lambda+\mu)n}\left((\lambda+\mu)(n+1)(r^2-R^2)\sin((n-1)\theta)-(\lambda+3\mu)r^2\sin((n+1)\theta)\right),$\newline
		$ f_4(r,\theta)=\frac{r^{n-1}}{(\lambda+\mu)n}\left((\lambda+\mu)(n+1)(r^2-R^2)\cos((n-1)\theta)+(\lambda+3\mu)r^2\cos((n+1)\theta)\right).$

\end{thm}

\begin{proof} The eigenvalues can be determined by imposing 
\begin{equation}
\label{eigeq}
Ae(\bo u)\bo n  =  \Lambda \bo u
\end{equation} 
at the boundary of the disk which can be assumed to be centered at the origin and so, on the boundary we have $r=R.$ We separate the study in the cases $n=0,$ $n=1$ and $n\geq2.$

\underline{\bf Case $n=0$}: 

The boundary condition is given by 
\[\begin{bmatrix}(\lambda+2\mu)c_0^{r}\ '(R)+\frac{\lambda}{R}c_0^r(R)\\ \mu\left(c_0^\theta\ '(R)-\frac{1}{R}c_0^\theta(R)\right)\end{bmatrix}=\Lambda\begin{bmatrix}c_0^r(R)\\ c_0^\theta(R)\end{bmatrix}\]
and taking into account \eqref{eq:n0} we obtain
\begin{align*}\begin{bmatrix}(\lambda+2\mu)A_0+\lambda A_0\\ \mu\left(B_0-B_0\right)\end{bmatrix}=\Lambda\begin{bmatrix}A_0 R\\ B_0R\end{bmatrix}&\Longleftrightarrow\begin{bmatrix}(2\lambda+2\mu)A_0\\ 0\end{bmatrix}=\Lambda\begin{bmatrix}A_0 R\\ B_0R\end{bmatrix}\\
&\Longleftrightarrow\underbrace{\begin{bmatrix}\frac{2(\lambda+\mu)}{R}&0\\0&0\end{bmatrix}}_{:=\bo M_0}\begin{bmatrix}A_0\\B_0\end{bmatrix}=\Lambda\begin{bmatrix}A_0\\B_0\end{bmatrix}.
\end{align*}
The Steklov-Lam\'{e} eigenvalues in this case are the eigenvalues of matrix $\bo M_0$, which are $0$ and $\frac{2(\lambda+\mu)}{R}.$ The corresponding eigenfunctions can be obtained from the eigenvectors of matrix $\bo v_1=(1,0)$ (associated to the eigenvalue $\frac{2(\lambda+\mu)}{R}$) and $\bo v_2=(0,1)$ (associated to the eigenvalue $0$).
In the case $\bo v_1=(1,0),$ from \eqref{eq:n0} we obtain $c_0^r(r)=r;\ c_0^\theta(r)=0,$ which implies that
\[\bo u(r,\theta)=r\bo e_r=r(\cos(\theta),\sin(\theta)).\]
In the case $\bo v_2=(0,1),$ again from \eqref{eq:n0} we obtain $c_0^r(r)=0;\ c_0^\theta(r)=r,$ which implies that
\[\bo u(r,\theta)=r\bo e_\theta=r(-\sin(\theta),\cos(\theta)).\]

\underline{\bf Case $n=1$}:

The boundary condition is given by
\begin{align*}&\begin{bmatrix}(\lambda+2\mu)c_1^{r}\ '(R)+\frac{\lambda}{R}c_1^r(R)+\frac{\lambda}{R}s_1^\theta(R)\\ \mu\left(\frac{1}{R}s_1^r(R)+c_1^\theta\ '(R)-\frac{1}{R}c_1^\theta(R)\right)\end{bmatrix}\cos(\theta)\\
+&\begin{bmatrix}(\lambda+2\mu)s_1^{r}\ '(R)+\frac{\lambda}{R}s_1^r(R)-\frac{\lambda}{R}c_1^\theta(R)\\ \mu\left(-\frac{1}{R}c_1^r(R)+s_1^\theta\ '(R)-\frac{1}{R}s_1^\theta(R)\right)\end{bmatrix}\sin(\theta)\\
=&
\Lambda\left(\begin{bmatrix}c_1^r(R)\\ c_1^\theta(R)\end{bmatrix}\cos(\theta)+\begin{bmatrix}s_1^r(R)\\ s_1^\theta(R)\end{bmatrix}\sin(\theta)\right)
\end{align*}
and since the previous equality shall hold for all values of $\theta$ we conclude that we must have
\[\begin{bmatrix}(\lambda+2\mu)c_1^{r}\ '(R)+\frac{\lambda}{R}c_1^r(R)+\frac{\lambda}{R}s_1^\theta(R)\\ \mu\left(\frac{1}{R}s_1^r(R)+c_1^\theta\ '(R)-\frac{1}{R}c_1^\theta(R)\right)\\
(\lambda+2\mu)s_1^{r}\ '(R)+\frac{\lambda}{R}s_1^r(R)-\frac{\lambda}{R}c_1^\theta(R)\\ \mu\left(-\frac{1}{R}c_1^r(R)+s_1^\theta\ '(R)-\frac{1}{R}s_1^\theta(R)\right)\end{bmatrix}=\Lambda\begin{bmatrix}c_1^r(R)\\ c_1^\theta(R)\\
s_1^r(R)\\ s_1^\theta(R)\end{bmatrix}.\]
Taking into account \eqref{eq:n1}, 
\[\scriptsize\hspace{-1cm}\begin{bmatrix}(\lambda+2\mu)A_1\left(\frac{-\lambda+\mu}{\lambda+\mu}\right)2R-\frac{\lambda}{R}A_1^0+\lambda A_1\left(\frac{-\lambda+\mu}{\lambda+\mu}\right)R+\frac{\lambda}{R}A_1^0+\lambda A_1\left(\frac{3\lambda+5\mu}{\lambda+\mu}\right)R\\ \mu\left(\frac{1}{R}B_1^0-B_1\left(\frac{-\lambda+\mu}{\lambda+\mu}\right)R+2B_1\left(\frac{3\lambda+5\mu}{\lambda+\mu}\right)R-\frac{1}{R}B_1^0-B_1\left(\frac{3\lambda+5\mu}{\lambda+\mu}\right)R\right)\\
-(\lambda+2\mu)B_1\left(\frac{-\lambda+\mu}{\lambda+\mu}\right)2R+\frac{\lambda}{R}B_1^0-\lambda B_1\left(\frac{-\lambda+\mu}{\lambda+\mu}\right)R-\frac{\lambda}{R}B_1^0-\lambda B_1\left(\frac{3\lambda+5\mu}{\lambda+\mu}\right)R\\ \mu\left(\frac{1}{R}A_1^0-A_1\left(\frac{-\lambda+\mu}{\lambda+\mu}\right)R+A_1\left(\frac{3\lambda+5\mu}{\lambda+\mu}\right)2R-\frac{1}{R}A_1^0-A_1\left(\frac{3\lambda+5\mu}{\lambda+\mu}\right)R\right)\end{bmatrix}=\]
\[\scriptsize=\Lambda\begin{bmatrix}-A_1^0+A_1 \left(\frac{-\lambda+\mu}{\lambda+\mu}\right) R^2\\ B_1^0+B_1\left(\frac{3\lambda+5\mu}{\lambda+\mu}\right)R^2\\
B_1^0-B_1  \left(\frac{-\lambda+\mu}{\lambda+\mu}\right) R^2\\ A_1^0+A_1\left(\frac{3\lambda+5\mu}{\lambda+\mu}\right)R^2\end{bmatrix}\Longleftrightarrow
\begin{bmatrix}4\mu A_1R\\ 4\mu B_1R\\
-4\mu B_1R\\ 4\mu A_1R\end{bmatrix}=\Lambda\begin{bmatrix}-A_1^0+A_1 \left(\frac{-\lambda+\mu}{\lambda+\mu}\right) R^2\\ B_1^0+B_1\left(\frac{3\lambda+5\mu}{\lambda+\mu}\right)R^2\\
B_1^0-B_1  \left(\frac{-\lambda+\mu}{\lambda+\mu}\right) R^2\\ A_1^0+A_1\left(\frac{3\lambda+5\mu}{\lambda+\mu}\right)R^2\end{bmatrix}\]	
which can be written as
\begin{equation}\scriptsize
\label{primigual}
\bo N_1 \begin{bmatrix}A_1^0\\ B_1^0\\ A_1\\B_1\end{bmatrix}=\Lambda \bo P_1 \begin{bmatrix}A_1^0\\ B_1^0\\ A_1\\B_1\end{bmatrix},
\end{equation}
where
\[\scriptsize\bo N_1=\begin{bmatrix}0 &0&4\mu R&0\\
0 & 0&0&4\mu R\\
0 & 0&0&-4\mu R\\
0 &0&4\mu R&0
\end{bmatrix}\quad\text{and}\quad\bo P_1=\begin{bmatrix}-1&0&\left(\frac{-\lambda+\mu}{\lambda+\mu}\right)R^2&0\\
0&1&0&\left(\frac{3\lambda+5\mu}{\lambda+\mu}\right)R^2\\
0 &1&0&\left(\frac{\lambda-\mu}{\lambda+\mu}\right)R^2\\
1&0&\left(\frac{3\lambda+5\mu}{\lambda+\mu}\right)R^2&0\end{bmatrix}.\]	
We have $\displaystyle{\det(\bo P_1)=-\frac{4(\lambda+3\mu)^2R^4}{(\lambda+\mu)^2}}<0$ which justifies the invertibility of the matrix $\bo P_1$ and we conclude that \eqref{primigual} is equivalent to
\begin{equation}
\label{primigual2}
\underbrace{\bo P_1^{-1}\cdot\bo N_1}_{:=\bo M_1} \begin{bmatrix}A_1^0\\ B_1^0\\ A_1\\B_1\end{bmatrix}=\Lambda  \begin{bmatrix}A_1^0\\ B_1^0\\ A_1\\B_1\end{bmatrix},
\end{equation}
and the Steklov-Lam\'{e} eigenvalues are the eigenvalues of matrix $\bo M_1$, which are $0$ (double eigenvalue) and $\frac{4\mu(\lambda+\mu)}{(\lambda+3\mu)R}$ (double eigenvalue). The eigenfunctions can be calculated from the eigenvectors, $\bo v_1=(-2R^2,0,1,0)$ and $\bo v_2=(0,-2R^2,0,1)$ (associated to the eigenvalue $\frac{4\mu(\lambda+\mu)}{(\lambda+3\mu)R}$) and $\bo v_3=(-1,0,0,0)$ and $\bo v_4=(0,1,0,0)$ (associated to the eigenvalue $0$). For instance, for $\bo v_1$ we get\newline
$ c_1^r(r)=2R^2+\left(\frac{-\lambda+\mu}{\lambda+\mu}\right)r^2;\ c_1^\theta(r)=0;\ s_1^r(r)=0;\ s_1^\theta(r)=-2R^2+\left(\frac{3\lambda+5\mu}{\lambda+\mu}\right)r^2$
and
$u_r(r,\theta)=\left(2R^2+\left(\frac{-\lambda+\mu}{\lambda+\mu}\right)r^2\right)\cos(\theta)$,
$u_\theta(r,\theta)=\left(-2R^2+\left(\frac{3\lambda+5\mu}{\lambda+\mu}\right)r^2\right)\sin(\theta)$
which implies that
\begin{align*}
\bo u(r,\theta)=&u_r(r,\theta)\bo e_r+u_\theta(r,\theta)\bo e_\theta\\ =&\left(2(R^2-r^2)+\frac{(\lambda+3\mu)r^2\cos(2\theta)}{\lambda+\mu},\frac{(\lambda+3\mu)r^2\sin(2\theta)}{\lambda+\mu}\right).\end{align*}
The eigenfunction associated to $\bo v_2$ is computed in a similar way and is given by
\begin{align*}\bo u(r,\theta)=&u_r(r,\theta)\bo e_r+u_\theta(r,\theta)\bo e_\theta\\=&\left(\frac{(\lambda+3\mu)r^2\sin(2\theta)}{\lambda+\mu},2(R^2-r^2)-\frac{(\lambda+3\mu)r^2\cos(2\theta)}{\lambda+\mu}\right).\end{align*}
The computation of the eigenfunction associated to $\bo v_3$ is similar, obtaining $c_1^r(r)=1,\ c_1^\theta(r)=0,\ s_1^r(r)=0,\ s_1^\theta(r)=-1 \Longrightarrow u_r(r,\theta)=\cos(\theta);\ u_\theta(r,\theta)=-\sin(\theta)$
which implies that
\[\bo u(r,\theta)=\left(\cos^2(\theta)+\sin^2(\theta),\cos(\theta)\sin(\theta)-\sin(\theta)\cos(\theta)\right)=(1,0).\]
Using the eigenvector $\bo v_4$ we get $\bo u(r,\theta)=(0,1)$

\underline{\bf Case $n\geq2$}:

The computations in this case are similar to those of the case $n=1$. We have
\begin{align*}&
\begin{bmatrix}(\lambda+2\mu)c_n^{r}\ '(R)+\frac{\lambda}{R}c_n^r(R)+\frac{\lambda}{R}ns_n^\theta(R)\\ \mu\left(\frac{n}{R}s_n^r(R)+c_n^\theta\ '(R)-\frac{1}{R}c_n^\theta(R)\right)\end{bmatrix}\cos(n\theta)\\
+&\begin{bmatrix}(\lambda+2\mu)s_n^{r}\ '(R)+\frac{\lambda}{R}s_n^r(R)-\frac{\lambda}{R}nc_n^\theta(R)\\ \mu\left(-\frac{n}{R}c_n^r(R)+s_n^\theta\ '(R)-\frac{1}{R}s_n^\theta(R)\right)\end{bmatrix}\sin(n\theta)\\
=&\Lambda\left(\begin{bmatrix}c_n^r(R)\\ c_n^\theta(R)\end{bmatrix}\cos(n\theta)+\begin{bmatrix}s_n^r(R)\\ s_n^\theta(R)\end{bmatrix}\sin(n\theta)\right)
\end{align*}
which implies that
\begin{equation}
\label{eqn}
\begin{bmatrix}(\lambda+2\mu)c_n^{r}\ '(R)+\frac{\lambda}{R}c_n^r(R)+\frac{\lambda}{R}ns_n^\theta(R)\\ \mu\left(\frac{n}{R}s_n^r(R)+c_n^\theta\ '(R)-\frac{1}{R}c_n^\theta(R)\right)\\
(\lambda+2\mu)s_n^{r}\ '(R)+\frac{\lambda}{R}s_n^r(R)-\frac{\lambda}{R}nc_n^\theta(R)\\ \mu\left(-\frac{n}{R}c_n^r(R)+s_n^\theta\ '(R)-\frac{1}{R}s_n^\theta(R)\right)\end{bmatrix}=\Lambda\begin{bmatrix}c_n^r(R)\\ c_n^\theta(R)\\
s_n^r(R)\\ s_n^\theta(R)\end{bmatrix}.
\end{equation}

Using \eqref{eq:ngen} we see that \eqref{eqn} can be written as
\begin{equation}
\label{primigualn}
\bo N_n \begin{bmatrix}A_n^0\\ B_n^0\\ A_n\\B_n\end{bmatrix}=\Lambda \bo P_n \begin{bmatrix}A_n^0\\ B_n^0\\ A_n\\B_n\end{bmatrix},
\end{equation}
where
\[\bo N_n=\begin{bmatrix}-2\mu(n-1)R^{n-2} &0&-2\mu\frac{(n-2)(n+1)}{n} R^n&0\\
0 & 2\mu(n-1)R^{n-2}&0&2\mu(n+1)R^n\\
0 & 2\mu(n-1)R^{n-2}&0&2\mu\frac{(n-2)(n+1)}{n}R^n\\
2\mu(n-1)R^{n-2} &0&2\mu(n+1)R^n&0
\end{bmatrix}\]
and
\[\bo P_n=\begin{bmatrix}-R^{n-1}&0&-\frac{\mu(n-2)+\lambda n}{n(\lambda+\mu)}R^{n+1}&0\\
0&R^{n-1}&0&\frac{\lambda(n+2)+\mu(n+4)}{n(\lambda+\mu)}R^{n+1}\\
0&R^{n-1}&0&\frac{\mu(n-2)+\lambda n}{n(\lambda+\mu)}R^{n+1}\\
R^{n-1}&0&\frac{\lambda(n+2)+\mu(n+4)}{n(\lambda+\mu)}R^{n+1}&0\end{bmatrix}.\]	
The matrix $\bo P_n$ is invertible because $\displaystyle{\det(\bo P_n)=-\frac{4(\lambda+3\mu)^2R^{4n}}{(\lambda+\mu)^2n^2}}<0$ 
and \eqref{primigualn} is equivalent to
\begin{equation}
\label{primigualfinal}
\underbrace{\bo P_n^{-1}\cdot\bo N_n}_{:=\bo M_n} \begin{bmatrix}A_n^0\\ B_n^0\\ A_n\\B_n\end{bmatrix}=\Lambda  \begin{bmatrix}A_n^0\\ B_n^0\\ A_n\\B_n\end{bmatrix},
\end{equation}
and the Steklov-Lam\'{e} eigenvalues are the eigenvalues of matrix $\bo M_n$, which are $\frac{2\mu(n-1)}{R}$ (double eigenvalue) and $\frac{2(n+1)\mu(\lambda+\mu)}{(\lambda+3\mu)R}$ (double eigenvalue).

The eigenfunctions can be calculated from the eigenvectors, $\bo v_1=(-\frac{(n+1)R^2}{n},0,1,0)$ and $\bo v_2=(0,-\frac{(n+1)R^2}{n},0,1)$ (associated to the eigenvalue $\frac{2(n+1)\mu(\lambda+\mu)}{(\lambda+3\mu)R}$) and $\bo v_3=(-1,0,0,0)$ and $\bo v_4=(0,1,0,0)$ (associated to the eigenvalue $\frac{2\mu(n-1)}{R}$.)
Using the eigenvector $\bo v_3$ we get
\[c_n^r(r)=r^{n-1},\ c_n^\theta(r)=0,\ s_n^r(r)=0,\ s_n^\theta(r)=-r^{n-1}\]
and
\[u_r(r,\theta)=r^{n-1}\cos(n\theta),\ u_\theta(r,\theta)=-r^{n-1}\sin(n\theta).\]
Therefore, we obtain
\begin{align*}\bo u(r,\theta)=
r^{n-1}\left(\cos((n-1)\theta),-\sin((n-1)\theta)\right)
\end{align*}
Following the same steps using the eigenvector $\bo v_4$ we obtain
\[\bo u(r,\theta)=r^{n-1}\left(\sin((n-1)\theta),\cos((n-1)\theta)\right).\]

Finally, from the eigenvector $\bo v_1$ we get, for $n=2,3,...$
\[	\begin{array}{c}
c_n^r(r)=\frac{(n+1)R^2}{n} r^{n-1}+ \left(\frac{-n\lambda-(n-2)\mu}{n(\lambda+\mu)}\right) r^{n+1};\quad c_n^\theta(r)=0\\
s_n^r(r)=0;\quad	s_n^\theta(r)=-\frac{(n+1)R^2}{n}r^{n-1}+\left(\frac{(n+2)\lambda+(n+4)\mu}{n(\lambda+\mu)}\right)r^{n+1}\\
\end{array},
\label{eq:ngendm}\]
which implies that
\[u_r(r,\theta)=\left(\frac{(n+1)R^2}{n} r^{n-1}+ \left(\frac{-n\lambda-(n-2)\mu}{n(\lambda+\mu)}\right) r^{n+1}\right)\cos(n\theta)\]
and
\[u_\theta(r,\theta)=\left(-\frac{(n+1)R^2}{n}r^{n-1}+\left(\frac{(n+2)\lambda+(n+4)\mu}{n(\lambda+\mu)}\right)r^{n+1}\right)\sin(n\theta).\]
Therefore,{\small 
\[\textstyle \bo u(r,\theta)_1=\frac{r^{n-1}}{(\lambda+\mu)n}\left(-(\lambda+\mu)(n+1)(r^2-R^2)\cos((n-1)\theta)+(\lambda+3\mu)r^2\cos((n+1)\theta)\right)\]
}
and in a similar fashion, we get
{\small \[\textstyle \bo u(r,\theta)_2=\frac{r^{n-1}}{(\lambda+\mu)n}\left((\lambda+\mu)(n+1)(r^2-R^2)\sin((n-1)\theta)+(\lambda+3\mu)r^2\sin((n+1)\theta)\right)\]}
which concludes the proof.
\end{proof}

Denote by $c_2(\lambda,\mu)=\frac{2(\lambda+\mu)}{R}$, $c_3(\lambda,\mu)=\frac{4\mu(\lambda+\mu)}{(\lambda+3\mu)R}$ and $c_4(\lambda,\mu)=\frac{2\mu}{R}$, which are the smallest eigenvalues obtained, respectively in cases (ii), (iii) and (iv) in Theorem~\ref{thm:eigdisk}. Then the following result helps establish what is the smallest non-zero eigenvalue of the disk.
\begin{prop}\label{prop:help-order}
	We have
	\begin{itemize}
		\item $c_2(\lambda,\mu)\leq c_4(\lambda,\mu)\leq c_3(\lambda,\mu)$, in the region $\left\{(\mu,\lambda)\in\mathbb{R}^2:0<\mu,\lambda<-3\mu\right\}$
		\item $c_4(\lambda,\mu)\leq c_3(\lambda,\mu)\leq c_2(\lambda,\mu)$, in $\left\{(\mu,\lambda)\in\mathbb{R}^2:0<\mu,\lambda\geq\mu\right\}$
		\item $c_3(\lambda,\mu)\leq c_2(\lambda,\mu)\leq c_4(\lambda,\mu)$, in $\left\{(\mu,\lambda)\in\mathbb{R}^2:0<\mu,-3\mu<\lambda\leq0\right\}$
		\item $c_3(\lambda,\mu)\leq c_4(\lambda,\mu)\leq c_2(\lambda,\mu)$, in $\left\{(\mu,\lambda)\in\mathbb{R}^2:0<\mu,0<\lambda\leq\mu\right\}.$
	\end{itemize}
\end{prop}

Since in dimension two $\lambda+\mu>0$ implying that $\lambda>-\mu>-3\mu$, the first situation listed in Proposition \ref{prop:help-order} cannot hold. Therefore, we have the following characterization for the first non-zero Steklov-Lam\'e eigenvalue for the disk.

\begin{prop}
	The smallest strictly positive eigenvalue for the disk $D_R$ of radius $R$ is given by
	\begin{itemize}
		\item $\Lambda_1(D_R) = \frac{2\mu}{R}$ when $\lambda> \mu$. In this case the associated eigenspace has dimension $2$, generated by 
		\[ \bo u_1 = (x_1,-x_2), \bo u_2 = (x_2,x_1),\]
		verifying $Ae(\bo u_i):e(\bo u_i) \equiv 4\mu$ on $\Bbb{R}^2$ and $|\bo u_i|^2 = r^2$ on $\Bbb{R}^2$.
		\item $\Lambda_1(D_R)=\frac{4\mu(\lambda+\mu)}{(\lambda+3\mu)R}$ when $\lambda\leq  \mu$. The associated eigenspace has dimension two and is generated by
		\[ \bo u_1 = \left(2(R^2-x_1^2-x_2^2)+\frac{\lambda+3\mu}{\lambda+\mu}(x_1^2-x_2^2), \frac{\lambda+3\mu}{\lambda+\mu}2x_1x_2\right).\]
		\[ \bo u_2 = \left( \frac{\lambda+3\mu}{\lambda+\mu}2x_1x_2,2(R^2-x_1^2-x_2^2)-\frac{\lambda+3\mu}{\lambda+\mu}(x_1^2-x_2^2)\right).\]
		Furthermore, we have
		\[ |\bo u_1|^2+|\bo u_2|^2 = 8(R^2-r^2)^2+2\left(\frac{\lambda+3\mu}{\lambda+\mu}\right)^2 r^4\]
		and
		\[ Ae(\bo u_1):e(\bo u_1)+Ae(\bo u_2):e(\bo u_2) = 32\mu \frac{\lambda+3\mu}{\lambda+\mu} r^2.\]
	\end{itemize}
	\label{prop:first-eig}
\end{prop}

The proof is immediate by investigating the order of the eigenvalues found in Theorem \ref{thm:eigdisk} in view of the observations made in Proposition \ref{prop:help-order}. Knowing the eigenstructure for the disk allows us to prove the following result similar to the scalar case according to Weinstock \cite{weinstock} and Brock \cite{brock}. 


\begin{thm}\label{thm:optimality-disk}
	Suppose $\lambda>\mu$ then the disk maximizes $\Lambda_1(\Omega)$ when:
	
	{\rm (a)} $\Omega$ has fixed volume.
	
	{\rm (b)} $\Omega$ is convex with fixed perimeter.
\end{thm}

\begin{proof} For simplicity, suppose $\Omega$ has area $\pi$ (or perimeter $2\pi$). 
In view of Proposition \ref{prop:first-eig}, the first non-zero eigenvalue of the unit disk $\Bbb D$ in this case is $\Lambda_1(\Bbb D)=\Lambda_2(\Bbb D)=2\mu$.
Consider the corresponding eigenfunctions $\bo u_1=(r\cos \theta,-r\sin \theta),\bo u_2=(r\sin \theta,r\cos \theta)$. Then it is straightforward to notice that $Ae(\bo u_i):e(\bo u_i) = 4\mu$ and $|\bo u_i|^2= r^2$ on $\Bbb{R}^2$.

Consider now a general $\Omega \subset \Bbb{R}^2$ with $|\Omega| = \pi$. Let us take the space $X_1=\{\bo u_{0,1},\bo u_{0,2},\bo u_{0,3},\bo u_1\}$ as a test space in \eqref{eq:rayleigh} for $\Lambda_1(\Omega)$, with $\bo u_1 = (r\cos \theta,-r\sin\theta)$, an eigenfunction associated to the first non-zero eigenvalue of the disk. We denote
\[ \bo u_{0,1} = (1,0), \bo u_{0,2} = (0,1), \bo u_{0,3} = (-x_2,x_1),\]
a basis for the rigid motions in dimension two. We may observe that
\[ \bo u_{0,1} \cdot \bo u_1 = x_1,\ \bo u_{0,2} \cdot \bo u_1 = -x_2,\ \bo u_{0,3} \cdot \bo u_1 = -2x_1x_2.\]

Therefore the shape $\Omega$ can be translated and rotated such that $\int_{\partial \Omega} \bo u_{0,j}\cdot \bo u_1=0$. Indeed, for a fixed orientation $\alpha \in [0,2\pi]$ of $\Omega$ we can translate $\Omega$ such that $\int_{\partial \Omega} x_1 = \int_{\partial \Omega} x_2 = 0$. Denote by $\Omega_\alpha$ the resulting shape. One may observe that $\int_{\Omega_0} (-2x_1x_2) = -\int_{\partial \Omega_{\pi/2}} (-2x_1x_2)$. Therefore, there exists an $\alpha \in [0,\pi/2]$ such that $\int_{\Omega_\alpha} x_1x_2=0$. Suppose now that $\Omega$ is translated and rotated such that $\int_{\partial \Omega}\bo u_1 \cdot \bo u_{0,j} = 0,\ j=1,2,3$. Let $\bo u = \alpha_1 \bo u_{0,1}+\alpha_2 \bo u_{0,2}+\alpha_3 \bo u_{0,3}  +c_1\bo u_1$ be an element of the test space $X_1$ defined above. 

It is straightforward to observe that 
\[ Ae(\bo u ): e(\bo u) = 4\mu c_1^2 \text{ and } \int_{\partial \Omega} |\bo u|^2 = \int_{\partial \Omega} (\alpha_1^2+\alpha_2^2+\alpha_3^2r^2+ c_1^2 r^2).\]
Therefore, the maximum of the associated Rayleigh quotient is
\[ \max_{\bo u \in X_1} \frac{\int_\Omega Ae(\bo u):e(\bo u)}{\int_{\partial \Omega} |\bo u|^2}=\frac{4\mu |\Omega|}{\int_{\partial \Omega}r^2}.\]
As a direct consequence, $\Lambda_1(\Omega) \leq \frac{4\mu|\Omega|}{\int_{\partial \Omega}r^2}$. We can now answer the two questions raised in the statement of the theorem.

(a) In \cite{brock} it is shown that $\int_{\partial \Omega} r^2$ is minimized by the disk at fixed volume.

(b) In \cite{weinstock} it is shown that $2|\Omega|/\int_{\partial \Omega} r^2$ is again maximized by the disk, among convex domains with fixed perimeter. This is a consequence of the inequality
\[ \frac{2|\Omega|}{ \int_{\partial \Omega} r^2} \leq \frac{2\pi}{|\partial \Omega|},\]
which holds for all convex domains according to \cite{weinstock}.

Moreover, in both cases above, when $\Omega$ is a disk, we have $\Lambda_1(\Omega) = \frac{4\mu|\Omega|}{\int_{\partial \Omega} r^2}$, showing that the upper bound is actually attained by the disk. The conclusion follows.
\end{proof}

\begin{rem}
	The case $\lambda \leq \mu$ is more challenging. Indeed, as indicated in Proposition \ref{prop:first-eig} in this case $Ae(u_j):e(u_j)$, $j=1,2$ is no longer a constant and the proof above no longer applies. Nevertheless, numerical results shown in Section \ref{sec:num-results} show that the disk is still a maximizer even when $\lambda \leq \mu$.
\end{rem}


\subsection{Upper bounds for the Steklov-Lam\'e eigenvalues}

In order to motivate the existence of solutions for optimization problems depending on the Steklov-Lam\'e eigenvalues, we derive upper bounds for these eigenvalues in terms of the classical Steklov eigenvalues $\sigma_n(\Omega)$ defined by \eqref{eq:steklov-eigs}.

Variational characterizations exist for the Steklov eigenvalues, using Rayleigh quotients. For simplicity, consider the following one (see \cite{Bucur-Nahon} for example)
\begin{equation}
\sigma_n(\Omega) = \min_{\dim S = n+1} \max_{u \in S}\frac{\int_\Omega |\nabla u|^2}{\int_{\partial \Omega} u^2}
\end{equation}
where the minimum is taken over all subspaces $S$ of $H^1(\Omega)\setminus H_0^1(\Omega)$ having dimension $n+1$.

Various results concerning the upper bounds for Steklov eigenvalues exist, depending on different geometric quantities:
\begin{itemize}
	\item $\sigma_k(\Omega)\Per(\Omega) \leq 2k\pi$ among simply connected domains in dimension two:  \cite{hersch-payne-schiffer}, in \cite{girouard-polterovich} it is shown that the inequality is sharp. 
	\item $\sigma_k(\Omega) \leq c_d k^{2/d} \frac{|\Omega|^{\frac{d-2}{d}}}{\Per(\Omega)}$: valid in arbitrary dimension \cite{colbois-elsoufi-girouard}.
	\item $\sigma_k(\Omega) \leq C(d,k) \frac{\displaystyle |\Omega|^{\frac{1}{d-1}}}{\displaystyle \diam(\Omega)^{\frac{2d-1}{d-1}}}$: among convex sets, in arbitrary dimension, where $\diam(\Omega)$ denotes the diameter of the set $\Omega$  \cite{alsayed-bogosel-henrot-nacry}.
\end{itemize} 
It is not our purpose here to give an exhausting list. For a more complete survey see \cite{survey-girouard-polterowich}. Using these results, analogue ones can be found for the Steklov-Lam\'e eigenvalues using the result below. In the following, for simplicity, we denote by $t(d) = d(d+1)/2$, the triangular number associated to the positive integer $d$, the dimension of the space of rigid motions $\bo R(\Omega)$.

\begin{prop}	\label{prop:upper-bounds}
	For every $n\geq 1$ we have 
	\[ \Lambda_n(\Omega) \leq (2\mu+d\lambda)\sigma_{dn+d^2(d+1)/2-1}(\Omega). \]
\end{prop}

\begin{proof} Given $\bo u = (u_i)_{i=1}^d\in \bo H^1(\Omega)$ we have
\begin{align*}
&Ae(\bo u): e(\bo u)  = 2\mu |e(\bo u)|^2+\lambda (\di \bo u)^2 \\
& = 2\mu \sum_{i,j=1}^d\frac{1}{2}( \partial_{x_i} u_j+\partial_{x_j} u_i)^2 + \lambda( \sum_{i=1}^d\partial_{x_i}u_i)^2\\
&\leq 2\mu\sum_{i,j=1}^d (\partial_{x_i} u_j)^2+d\lambda \sum_{i=1}^d (\partial_{x_i}u_i)^2\leq (2\mu+d\lambda)\sum_{i=1}^d\|\nabla u_i\|_{\bo L^2(\Omega)}^2,
\end{align*}
where we used the classical inequality $(\sum_{i=1}^d x_i)^m \leq m\sum_{i=1}^m x_i^2$. Consider now the the first $d(n+t(d))$ eigenfunctions associated to the eigenvalues $\sigma_0(\Omega), ..., \sigma_{d(n+t(d))-1}(\Omega)$ for the Steklov problem \eqref{eq:steklov-eigs} on $\Omega$, giving a subspace of dimension $d(n+t(d))$ in $H^1(\Omega)$. Taking $n+t(d)$ vectors made of $d$ of these eigenfunctions we obtain a subspace $\bo S$ of $\bo H^1(\Omega)$ of dimension $n+t(d)$. 

Every $\bo u=(u_1,...,u_d)\in \bo S$ verifies $\|\nabla u_j\|_{L^2(\Omega)}^2 \leq \sigma_{d(n+t(d))-1}(\Omega)\|u_j\|_{L^2(\partial \Omega)}^2$, for every $j=1,...,d$. In view of the inequality proven above, we have
\[ \int_\Omega Ae(\bo u):e(\bo u) \leq (2\mu+d\lambda) \sum_{i=1}^d \|\nabla u_i\|_{\bo L^2(\Omega)}^2 \leq (2\mu+d\lambda)\sigma_{d(n+t(d))-1}(\Omega) \int_{\partial \Omega} |\bo u|^2.\]

Therefore, considering $\bo S$ as a test space in \eqref{eq:rayleigh} we obtain 
\begin{align*}
\Lambda_n(\Omega) &\leq  \max_{\bo u\in \bo S\setminus \bo H_0^1(\Omega)} \frac{\int_\Omega Ae(\bo u):e(\bo u)}{\int_{\partial \Omega} |\bo u|^2}
\leq (2\mu+d\lambda)\sigma_{nd+dt(d)-1}(\Omega).
\end{align*}
\end{proof}

Under the hypothesis $2\mu+d\lambda>0$ we have the following bounds for the Steklov-Lam\'e eigenvalues, depending on classical constraints.
\begin{thm}
	\label{thm:upper-bounds}
	Let $\Omega$ be a bounded Lipschitz domain. Then we have:
	\begin{enumerate}[label=\upshape{(\roman*)}]
		\item If the perimeter of $\Omega$ is fixed then $\Lambda_k(\Omega) \Per(\Omega)^{\frac{1}{d-1}}$ is bounded from above.
		\item If the volume of $\Omega$ is fixed then $\Lambda_k(\Omega)$ is bounded from above.
		\item If the diameter of the convex set $\Omega$ is fixed the $\Lambda_k(\Omega)$ is bounded from above.
	\end{enumerate}
\end{thm}

\begin{proof} (a) and (b) are a consequence of the inequality $\sigma_k(\Omega) \leq c_d k^{2/d} \frac{|\Omega|^{\frac{d-2}{d}}}{\Per(\Omega)}$ proved in   \cite{colbois-elsoufi-girouard} and of  the isoperimetric inequality. 

(c) is a consequence of the inequality $\sigma_k(\Omega) \leq C(d,k) \frac{ |\Omega|^{\frac{1}{d-1}}}{ \diam(\Omega)^{\frac{2d-1}{d-1}}}$ proved in \cite{alsayed-bogosel-henrot-nacry} and the isodiametric inequality.
\end{proof}

\section{Stability of the spectrum on variable domains}
\label{sec:existence}

In the scalar case, the behavior of the Steklov eigenvalues \eqref{eq:steklov-eigs} with respect to domain perturbations was investigated in \cite{Bogosel}, \cite{Bucur-Nahon}, \cite{stability-steklov}. It is possible to generalize all these results to the Steklov-Lam\'e case. 

For $y \in \Bbb{R}^d$, $\xi$ a unit vector and $\varepsilon>0$ we define the cone 
\[ C(y,\xi,\varepsilon) = \{x \in \Bbb{R}^d : (z-y)\cdot \xi \geq \cos \varepsilon |z-y| \text{ and } 0<|z-y|<\varepsilon\}.\]
Following \cite[Chapter 2]{henrot-pierre-english}, we say that $\Omega$ verifies the \emph{$\varepsilon$-cone condition} if for every $x \in \partial \Omega$ there exists a unit vector $\xi_x$ such that for every $y \in \overline \Omega \cap B(x,\varepsilon)$ we have $C(y,\xi_x,\varepsilon)\subset \Omega$. It can be shown that this condition is equivalent to $\Omega$ being Lipschitz with a prescribed upper bound on the Lipschitz constant. In particular, convex domains or domains star-shaped with respect to a ball verify an $\varepsilon$-cone property.

In \cite[Proposition 2.3]{Bogosel} it is shown that if $D\subset \Bbb{R}^d$ is bounded and open and $\Omega \subset D$ verifies an $\varepsilon$-cone condition then $\Per(\Omega)$ is uniformly bounded by a constant depending only on $\varepsilon$ and $D$.

In order to underline the behaviour of the Steklov-Lam\'e eigenvalues with respect to sequences of domains for which the perimeter is not continuous, let us define the weighted Steklov-Lam\'e eigenvalues. Consider $\Theta\in L^\infty(\Omega), \Theta\geq \beta>0$ and define $\Lambda(\Omega,\Theta)$ by
\begin{equation}
\left\{ \begin{array}{rcll}
-\di A(e(\bo u)) & = & 0 & \text{ in } \Omega \\
Ae(\bo u)\bo n & = & \Lambda(\Omega,\Theta) \Theta \bo u & \text{ on } \partial \Omega,
\end{array}\right.
\label{eq:weighted-Steklov-eig}
\end{equation}
with the associated variational characterization
\begin{equation}
\Lambda_n(\Omega,\Theta) = \min_{\bo S_{n}\subset \bo H^1(\Omega)}
\max_{\bo u \in \bo S_n\setminus\bo H_0^1(\Omega)} \frac{\int_\Omega Ae(\bo u):e(\bo u)}{\int_{\partial \Omega} \Theta |\bo u|^2}
\label{eq:weighted-rayleigh}
\end{equation}
where the minimum is taken over all subspaces of $\bo H^1(\Omega)$ having dimension $n+t(d)$. 
It is obvious that $\Theta \equiv 1$ gives the Steklov-Lam\'e eigenvalues. The weighted eigenvalues \eqref{eq:weighted-Steklov-eig} enter into the framework presented in \cite{Sebastian}. Moreover, \eqref{eq:weighted-rayleigh} shows that $\Theta\geq \Theta'$ implies $\Lambda_n(\Omega,\Theta)\leq \Lambda_n(\Omega,\Theta')$. 
Furthermore, Proposition \ref{prop:upper-bounds} and Theorem \ref{thm:upper-bounds} extends to weighted Steklov-Lam\'e eigenvalues, since under the hypotheses considered, we have $\Lambda_n(\Omega,\Theta) \leq \frac{1}{\beta}\Lambda_n(\Omega)$. 

We say that a sequence of domains $\Omega_\varepsilon$ converges to a domain $\Omega$ if the Hausdorff distance between their complements converges to zero. See \cite[Chapter 2]{henrot-pierre-english} for introductory aspects related to the convergence in the Hausdorff metric.

The following result is proved in \cite{Bucur-Nahon} and extends a result from \cite{Bogosel}. 

\begin{prop}	\label{prop:conv-traces}
	Let $\Omega$ be a bounded Lipschitz domain and let $(\Omega_n)$ be a sequence of domains verifying the $\varepsilon$-cone condition. Consider a weight function $\Theta \in L^\infty(\partial \Omega)$ and a sequence of weight functions $\Theta_n \in L^\infty(\partial \Omega_n)$ such that $\Theta, \Theta_n \geq \beta>0$ and
	\[ \limsup_{n\to \infty} \|\Theta_n\|_{L^\infty(\partial \Omega_n)}<\infty \text{ and } \Theta_n \mathcal H^1_{\lfloor \partial \Omega_n} \rightharpoonup \Theta \mathcal H^1_{\lfloor \partial \Omega} \]
	weackly-$*$ in the sense of measures. If $(u_n) \subset H^1(\Bbb{R}^d)$ converges weakly to $u$ in $H^1(\Bbb{R}^d)$ then
	\[ \int_{\partial \Omega_n} \Theta_n u_n^2 \to \int_{\partial \Omega} \Theta u^2 \text{ as } \varepsilon\to 0.\]
\end{prop}
This result is a first step towards the desired stability result. In addition, the proof requires some uniform dependence on the domain for the constant in Korn's inequality
\begin{equation}\label{eq:korn}
\|\nabla \bo u\|_{L^2(\Omega)}^2\leq C_1(\Omega) \|e(\bo u)\|_{L^2(\Omega)}^2+C_2(\Omega)\|\bo u\|_{L^2(\Omega)}^2,
\end{equation}
for all $\bo u\in\bo H^1(\Omega)$. There are few cases in which the dependence of the constants of the domain is explicited. In particular, in \cite{korn-book} it is shown that if $\Omega$ has bounded diameter, is star-shaped with respect to a ball $B_{r_1}$ of radius $r_1$ and $\gamma$ is the distance between $\partial \Omega$ and $B_{r_1}$ then we may choose
\begin{equation}
C_1(\Omega)=C_1(\diam(\Omega)/r_1)^{d+1} \text{ and }C_2(\Omega) = C_2(\diam(\Omega)/r_1)^d\gamma^{-2},
\label{eq:uniform-constants-korn}
\end{equation}
with $C_1,C_2$ dimensional constants. Inequality \eqref{eq:korn} with constants \eqref{eq:uniform-constants-korn} is a consequence of \cite[Theorem 2.10]{korn-book} and is also presented in \cite{kondratiev-korn}. This result is of particular interest in our case, since bounds on the symmetrized gradient, together with bounds on the gradient in a small ball are enough to obtain global bounds on the gradient. For the sake of completeness, we recall the result below.

\begin{thm}(Theorem 2.10 from \cite{korn-book})
	\label{thm:small-to-large}
	Suppose that $\Omega\subset \Bbb{R}^d$ is a bounded domain of bounded diameter and $\Omega$ is star-shaped with respect to the ball $B_{r_1} = \{|x|<r_1\}$. Then for any $\bo u \in \bo H^1(\Omega)$ we have the inequality
	\begin{equation}
	\label{eq:small-to-large}
	\|\nabla \bo u\|_{\bo L^2(\Omega)} \leq C_1\Big( \frac{\diam(\Omega)}{r_1}\Big)^{d+1}\|e(\bo u)\|^2_{\bo L^2(\Omega)}+ C_2\Big( \frac{\diam(\Omega)}{r_1}\Big)^{d} \|\nabla \bo u\|_{\bo L^2(B_{r_1})}^2,
	\end{equation}
	where $C_1,C_2$ are constants depending on the dimension $d$.
\end{thm}

In the following, we prove a result similar to \cite[Theorem 2]{Sebastian} for the case of moving domains.

\begin{thm}\label{thm:uniform-bounds-H1}
	Let $\Omega_n\subset \Bbb{R}^2$ be a sequence of domains verifying the $\varepsilon$-cone property, converging in the Hausdorff metric to the bounded open domain $\Omega$. Suppose that $\Omega_n, \Omega$ are star shaped with respect to $B_{r_1}$ with compact support inside $\Omega$. Suppose the weights $\Theta_n$ verify the hypotheses of Proposition \ref{prop:conv-traces}. 
	
	Let $K\subset \Omega$ be a compact set with $B_{r_1}\subset K$. Then for every sequence $\bo u_n \in \bo H^1(\Omega_n)$ there exists a constant $C$, independent of $\bo u_n$, such that
	\[ \|\bo u_n\|_{\bo H^1(K)}^2 \leq C\left(\|e(\bo u_n)\|_{\bo L^2(\Omega_n)}^2+ \int_{\partial \Omega_n} \Theta_n |\bo u_n|^2\right).\]
	
	Moreover, there exists a constant $C$, independent of $\bo u_n$, such that 
	\begin{equation}\label{eq:uniform-bound-H1}
	\|\bo u_n\|_{\bo H^1(\Omega_n)}^2 \leq C\left(\|e(\bo u_n)\|_{\bo L^2(\Omega_n)}^2+ \int_{\partial \Omega_n} \Theta_n |\bo u_n|^2\right).
	\end{equation}
\end{thm} 

\begin{proof} Since $\Omega_n$ verify the $\varepsilon$-cone property and converge to $\Omega$, which is bounded, we may assume without loss of generality that $\Omega_n$ are contained in a ball $B$ for $n$ large enough.

Following \cite[Proposition 2.2.17]{henrot-pierre-english} the compact $K$ is contained in $\Omega_n$ for $n$ large enough. In order to prove the first inequality, assume that there exists a sequence $\bo u_n\in \bo H^1(\Omega_n)$ such that 
\[ \|\bo u_n\|_{\bo H^1(K)} = 1 \text{ and } \|e(\bo u_n)\|_{\bo L^2(\Omega_n)}^2+ \int_{\partial \Omega_n} \Theta_n |\bo u_n|^2<\frac{1}{n}.\]
The uniform bounds for $\|e(\bo u_n)\|_{\bo L^2(\Omega_n)}$ and $\|\nabla \bo u_n\|_{\bo L^2(B_{r_1})}$ together with Theorem \ref{thm:small-to-large} imply that $\|\nabla \bo u_n\|_{\bo L^2(\Omega_n)}$ is bounded uniformly with respect to $n$. Moreover, 
\[\frac{1}{n}> \int_{\partial \Omega_n} \Theta_n |\bo u_{n}|^2 \geq \beta \int_{\partial \Omega_n} |\bo u_{n}|^2.\]
Thus, if $u$ is a generic component of $\bo u_{n}$ we find that $\|\nabla u\|_{L^2(\Omega_n)}^2+\beta \|u\|_{L^2(\partial \Omega_n)}^2$ is uniformly bounded from above. The first Robin-Laplace eigenvalue defined for $\beta>0$ by 
\[\lambda_{1,\beta}(\Omega) = \inf_{u \in H^1(\Omega), u\neq 0}\frac{\int_\Omega |\nabla u|^2+ \beta \int_{\partial \Omega} u^2}{\int_\Omega u^2}\]
is minimized by the ball when the volume of $\Omega$ is fixed. The reader can consult \cite{robin-bucur-giacomini} and the references therein. Moreover, if $|\Omega|$ has an upper bound, then, in view of \cite[Corollary 3.2]{robin-bucur-giacomini}, $r\mapsto \lambda_{1,\beta}(B_r)$ is strictly decreasing and therefore has a strictly positive lower bound $q_\beta>0$, since $r$ is bounded from above. In our case, $\Omega_n$ have uniformly bounded perimeters, since $\Omega_n$ have the $\varepsilon$-cone property (see \cite[Proposition 2.3]{Bogosel}). Therefore, in view of the isoperimetric inequality, they also have uniformly bounded volumes. In view of the arguments above, there exists $q_\beta>0$ such that 
\[\int_{\Omega_n} |\nabla u|^2+ \beta \int_{\partial \Omega_n} u^2 \geq q_\beta \int_{\Omega_n} u^2.\]
As a consequence $\|\bo u_n\|_{\bo L^2(\Omega_n)}$ are bounded, implying that $\|\bo u_n\|_{\bo H^1(\Omega_n)}$ are uniformly bounded from above.

The sets $\Omega_n$ have Lipschitz boundaries with uniformly bounded constants, therefore, the extension operators from $H^1(\Omega_n)$ to $H^1(B)$ are uniformly bounded. Thus, we may consider the extensions $\widetilde{\bo u_n}\in \bo H^1(B)$ of $\bo u_n$, which are uniformly bounded in $\bo H^1(B)$. Up to extracting a subsequence, we may assume $\widetilde{\bo u_n}$ converge weakly to $\tilde{\bo u} \in \bo H^1(B)$ and thus strongly in $\bo L^2(B)$. 

The set $\Omega$ is star shaped with respect to $B_{r_1}$ and the distance between $B_{r_1}$ and $\partial \Omega$ is strictly positive. It is, thus, possible to write $\Omega$ as a union of compact sets $K_m$, $m\geq 1$ such that $B_{r_1}\subset K_m$, $K_m \subset K_{m+1}$ and $K_m$ are star-shaped with respect to $B_{r_1/2}$. Assuming $B_{r_1}$ is centered at the origin, it is enough to consider $K_m = \overline{ (1-\frac{1}{m+1})\Omega}$ for $m\geq 1$.

Fix $m$ and a compact $K_m$ defined as above. In \cite[Proposition 2.2.17]{henrot-pierre-english} it is proved that if $\Omega_n \to \Omega$ in the Hausdorff metric and $K$ is a compact contained in $\Omega$ then $K$ is contained in $\Omega_n$ for all $n$ large enough. Therefore for $n$ large enough $K_m \subset \Omega_n$. Moreover, in view of the definition of $K_m$, the constants in Korn's inequality \eqref{eq:korn} may be chosen uniform with respect to $m$ as in \eqref{eq:uniform-constants-korn}. Therefore, for $n \geq m$ and $p \geq 0$ we have
\[\|\nabla \bo (\bo u_{n+k}-\bo u_n)\|_{\bo L^2(K_m)}^2\leq C_1 \|e(\bo u_{n+k})-e(\bo u_n)\|_{\bo L^2(K_m)}^2+C_2\|\bo u_{n+k}-\bo u_n\|_{\bo L^2(K_m)}^2.\]
Since $(\bo u_n)$ converge stronglky in $\bo L^2(K_m)$ and $\|e(\bo u_n)\|_{\bo L^2(K_m)} \to 0$ we find that $(\nabla \bo u_{n})$ is a Cauchy sequence in $\bo L^2(K_m)$, implying that $\bo u_n$ converges strongly to $\widetilde{\bo u}$ in $\bo H^1(K_m)$. As a consequence $e(\widetilde{\bo u})=0$ in $K_m$, implying that $\widetilde{\bo u}$ is a rigid motion in $K_m$ for every $m$. Since $K_m$ is an increasing sequence of compacts, taking $m\to \infty$ we find that $\widetilde{\bo u}$ is a rigid motion on $\Omega$. Applying Proposition \ref{prop:conv-traces} we also find that $\int_{\partial \Omega} \Theta |\widetilde{\bo u}|^2 = 0$, showing that $\widetilde{\bo u}=0$ on $\partial \Omega$. In conclusion $\tilde{\bo u}=0$ in $\Omega$. In particular, for $m$ large enough we have $K \subset K_m$ so $\widetilde{\bo u}=0$ on $K$. However, the strong convergence of $\bo u_n$ to $\tilde{\bo u}$ in $\bo H^1(K)$ implies $\|\tilde{\bo u}\|_{\bo H^1(K)}=1$, a contradiction.

In order to prove \eqref{eq:uniform-bound-H1} it is enough to pick $K=\overline{B_{r_1}}$ and use \eqref{eq:small-to-large}. \end{proof}

We are now ready to prove the stability result for Steklov-Lam\'e eigenvalues. 

\begin{thm}\label{thm:weight-continuity}
	Let $\Omega$ be a bounded Lipschitz domain and let $\Omega_n$ be a sequence of domains with $\varepsilon$-cone property, converging to $\Omega$ for the Hausdorff complementary distance. 
	Assume the hypotheses of Theorem \ref{thm:uniform-bounds-H1} are verified. 
	Consider weights $\Theta\in L^\infty(\partial \Omega), \Theta_n \in L^\infty(\partial \Omega_n)$ verifying the hypotheses of Proposition \ref{prop:conv-traces}.
	
	Then for all $k\geq 1$ we have
	\[ \lim_{n\to \infty} \Lambda_k (\Omega_n,\Theta_n)= \Lambda_k(\Omega,\Theta).\]
\end{thm}

\begin{proof} The proof is divided in two steps. 

{\bf Lower semicontinuity.} For each $n$ consider $\bo S_n\subset \bo H^1(\Omega_n)$ a subspace which attains $\Lambda_k(\Omega,\Theta_n)$ in \eqref{eq:weighted-rayleigh}. Following \cite{Bucur-Nahon} consider $(\bo u_{p,n})_{p=-t(d)+1,...,0,...,k}$ an adapted basis for $\bo S_n$, i.e. a basis orthonormal with respect to $\bo u \mapsto \int_{\partial \Omega_n} \Theta_n |\bo u|^2$ and orthogonal relative to $\bo u \mapsto \int_{\Omega_n} Ae(\bo u):e(\bo u)$.

Since $(\Omega_n)_{n\geq 1}$ verify an $\varepsilon$-cone condition, it follows that $\Per( \Omega_n)$ is uniformly bounded. Without loss of generality, up to choosing a converging subsequence, suppose that $\Per(\Omega_n)$ converges. Also, since $\Omega_n$ converges to $\Omega$ all domains $\Omega_n$ have uniformly bounded diameters. In view of Proposition \ref{prop:upper-bounds} and Theorem \ref{thm:upper-bounds} we find that $\Per(\Omega_n) \Lambda_k(\Omega_n,\Theta_n)$ are uniformly bounded.
As a direct consequence, $\int_{\Omega_n} Ae(\bo u_{p,n}):e(\bo u_{p,n})$ are uniformly bounded, implying that $\|e(\bo u_{p,n})\|_{\bo L^2(\Omega_n)}$ are uniformly bounded. Theorem \ref{thm:uniform-bounds-H1} implies that $\|\bo u_{p,n}\|_{\bo H^1(\Omega_n)}$ are uniformly bounded.


Since $\Omega_n$ are Lipschitz with a controlled constant (coming from the $\varepsilon$-cone property, see \cite[Remark 2.4.8]{henrot-pierre-english}), each of the functions $\bo u_{p,n}$ can be extended to $\bo H^1(\Bbb{R}^d)$ with a controlled extension constant depending on the upper bound on the Lipschitz constant of the domains. Thus, denoting the extensions with the same symbols, we find that $(\bo u_{p,n})_{n\geq 1}$ are bounded in $\bo H^1(\Bbb{R}^d)$. Up to the extraction of a sub-sequence we suppose that $(\bo u_{p,n})$ converge weakly in $\bo H^1(\Bbb{R}^d)$ to $(\bo u_p)_{p=-t(d)+1,...,0,...,k}$. Proposition \ref{prop:conv-traces} implies that
\[ \delta_{pp'} = \int_{\partial \Omega_n} \Theta_n \bo u_{p,n}\cdot \bo u_{p',n} \to \int_{\partial \Omega} \Theta \bo u_p\cdot \bo u_{p'} \text{ as } \varepsilon \to 0.\]
Therefore $(\bo u_p)$ is orthonormal for the scalar product $(\bo u,\bo v)\mapsto \int_{\partial \Omega} \Theta \bo u \cdot \bo v$ and, as a consequence, $(\bo u_p)_{p=-t(d)+1,...,0,...,k}$ generate a subspace of dimension $k+t(d)$ when restricted to $\Omega$. In view of the weak convergence in $\bo H^1(\Bbb{R}^n)$ of the sequences $(\bo u_{p,n})$ we have for all $a_p$ with $\sum_{p=-t(d)+1}^k a_p^2=1$
\[ \Lambda_k(\Omega,\Theta) \leq \liminf_{n \to \infty} \max_{(a_p)\in \Bbb{S}^{n+t(d)-1}} Ae(\sum_p a_p \bo u_p):e(\sum_p  a_p\bo u_p) \leq \liminf_{n \to \infty} \Lambda_k(\Omega_n,\Theta_n).\]

{\bf Upper semicontinuity.} Let $\bo V$ be a subspace of $\bo H^1(\Omega)$ having dimension $k+t(d)$ which attains $\Lambda_k(\Omega,\Theta)$ in \eqref{eq:weighted-rayleigh} and let $(\bo v_{-t(d)+1},...,\bo v_0,... \bo v_k)$ an adapted basis for it (as before, orthonormal for the corresponding scalar product on $\partial \Omega$ and orthogonal for the scalar product on $\Omega$). Extending the functions $(\bo v_p)_{p=-t(d)+1,...,k}$ from $\Omega$ to $\Bbb{R}^d$ these functions still make an independent family on $\bo H^1(\Omega_n)$ for $n$ large enough. In the following consider $\bo a_n = (a_{p,n})_{p=-t(d)+1,...,k} \in \Bbb S^{k+t(d)-1}$ (unit sphere in $\Bbb{R}^{k+t(d)}$) such that, denoting $\bo W = \text{Span}(\bo v_{-t(d)+1},...,\bo v_0,...,\bo v_k)$,
\begin{equation}\label{eq:ub-lamk} \sup_{\bo w \in \bo W}  \frac{\displaystyle \int_{\Omega_n} Ae(\bo w):e(\bo w)}{\displaystyle \int_{\partial \Omega_n} \Theta_n |\bo w|^2 }=\frac{\displaystyle \int_{\Omega_n} Ae\Big(\sum_{p=-t(d)+1}^k a_{p,n}\bo v_{p}\Big):e\Big(\sum_{p=-t(d)+1}^k a_{p,n}\bo v_{p}\Big)}{\displaystyle \int_{\partial \Omega_n} \Theta_n \Big|\sum_{p=-t(d)+1}^k a_{p,n}\bo v_{p}\Big|^2 }.\end{equation}
Of course, \eqref{eq:ub-lamk} gives an upper bound for $\Lambda_k(\Omega_n,\Theta_n)$. 
Up to extracting a subsequence, we suppose that $a_{p,n}$ converges to $a_p$ for every $p=-t(d)+1,...,-1,0,...,k$. Denoting $\bo w_n = \sum_{p=-t(d)+1}^k a_{p,n}\bo v_{p}$, $\bo w = \sum_{p=-t(d)+1}^k a_{p}\bo v_{p}$ and using Proposition \ref{prop:conv-traces} we find that 
\[\int_{\partial \Omega_n} \Theta_n |\bo w_n|^2 \to \int_{\partial \Omega} \Theta |\bo w|^2 \text{ as }\varepsilon \to 0.\]
The convergence of $\Omega_n$ to $\Omega$ and of $a_{p,n}$ to $a_p$ implies that
\[  \int_{\Omega_n} Ae(\bo w_n):e(\bo w_n) \to \int_{\Omega} Ae(\bo w):e(\bo w)\]
as $n\to \infty$. As a consequence we find that
\[ \limsup_{n \to \infty} \Lambda_k(\Omega_n,\Theta_n) \leq \frac{\int_{\Omega} Ae(\bo w):e(\bo w)}{\int_{\partial \Omega} \Theta |\bo w|^2} \leq \Lambda_k(\Omega,\Theta),\]
which finishes the proof. \end{proof}

The result regarding the convergence of weighted Steklov-Lam\'e eigenvalues allows us to find in a straightforward way the upper semicontinuity of the eigenvalues, needed in order to prove existence results analogue to those in \cite{Bogosel}.

\begin{cor}\label{cor:semi-continuity}
	Let $\Omega$ be a bounded Lipschitz domain and $\Omega_n$ a sequence of domains with the $\varepsilon$-cone property verifying the hypotheses of Theorem \ref{thm:weight-continuity}. 
	We have the following:
	
	{\rm (a)} If $\Per(\Omega_n) \to \Per(\Omega)$ then $\Lambda_k(\Omega_n) \to \Lambda_k(\Omega)$ for every $k \geq 1$.
	
	{\rm (b)} In general, $\limsup_{n \to \infty} \Lambda_k(\Omega_n) \leq \Lambda_k(\Omega)$ for every $k \geq 1$.
\end{cor}
\begin{proof} The result follows at once from Theorem \ref{thm:weight-continuity}. 

(a) It suffices to take $\Theta_n=\Theta \equiv 1$ in Theorem \ref{thm:weight-continuity}.

(b) Taking $\Theta_n \equiv 1$, the lower-semicontinuity of the perimeter implies that any weak-* limit $\Theta$ of the sequence $\Theta_n = \mathcal H^{d-1}\lfloor \Omega_n$ verifies $\Theta \geq 1$. Therefore
\[ \limsup_{n \to \infty} \Lambda_k(\Omega_n) = \limsup_{n \to \infty} \Lambda_k(\Omega_n,\Theta_n) = \Lambda_k(\Omega,\Theta) \leq \Lambda_k(\Omega).\] \end{proof}

In the following we focus on the question of existence of solutions. It is classical that the existence of solutions for a shape maximization problem depends on compactness properties for a maximizing sequence. 
We focus our attention on the class of convex domains and the volume constraint. An initial result is given below.

\begin{thm}\label{thm:diameter-bound}
	Suppose that $(\Omega_n)$ is a sequence of open, convex sets with unit volume such that $\text{diam}(\Omega_n) \to \infty$. Then $\Lambda_k(\Omega_n) \to 0$.
\end{thm}

\begin{proof} This is a direct consequence of Proposition \ref{prop:upper-bounds} and of the analogue results for the Steklov eigenvalues shown in \cite{Bogosel} or \cite{alsayed-bogosel-henrot-nacry}. \end{proof}

We are now ready to state the existence result in the class of convex sets. 

\begin{thm}
	\label{thm:existence-convex}
	For every $k\geq 1$ there exists a solution to the problem of maximizing the Steklov-Lam\'e eigenvalue $\Lambda_k(\Omega)$ in the class of convex sets having unit volume. 
\end{thm}

\begin{proof} From Theorem \ref{thm:diameter-bound} any maximizing sequence $(\Omega_n)_{n\geq 1}$ has an upper bound on the diameter. Considering convex domains $\Omega_n$ which have bounded diameters, the Blaschke selection theorem \cite[Chapter 1]{schneider} implies the existence of a subsequence $\Omega_n$ converging to a convex set $\Omega$ in the Hausdorff distance. Moreover, $\Omega$ also has unit volume and verifies the same bound on the diameter as $\Omega_n$.

Therefore, since $\Omega$ has non-void interior, it contains a closed disk of radius $r_1$ and $\Omega$ is star-shaped with respect to $B_{r_1}$ since it is convex. For $n$ large enough, $B_{r_1}$ is also contained in $\Omega_n$, and therefore the hypotheses of Theorem \ref{thm:weight-continuity} and Corollary \ref{cor:semi-continuity} are verified, which implies
\[ \limsup_{n \to \infty} \Lambda_k(\Omega) \leq \Lambda_k(\Omega).\]
Thus $\Omega$ is a maximizer of $\Lambda_k(\Omega)$ among convex sets with unit volume. \end{proof}

\begin{rem} Possible generalizations can be attempted in multiple directions:
	
	(i) Proposition \ref{prop:upper-bounds} suggests that problems that are well posed for the scalar Steklov problem should behave in a similar way for the Steklov-Lam\'e eigenvalues. The perimeter constraint may also be considered, since an upper bound exists \cite{hersch-payne-schiffer}. However, in the scalar case in dimension two, in \cite{girouard-polterovich} it is proved that the upper bound from \cite{hersch-payne-schiffer} is tight but is never attained in the class of simply connected sets. 
	
	(ii) It is not clear what is the most general class of admissible domains for which an existence result like Theorem \ref{thm:existence-convex} might hold. The key ingredients are the existence of an upper bound on the diameter and the stability arguments, which depend on the availability of uniformly bounded constants in Korn's inequality. In view the results in \cite{horgan-payne} it is possible that such a result could hold in the class of star-shaped domains.
	
	(iii) Relaxed formulations can be considered like those in \cite{bogosel-bucur-giacomini}, which allow to recover existence results among simply connected sets and volume constraint. 
	
	(iv) Theorem \ref{thm:optimality-disk} and the numerical results suggest that the first eigenvalue is maximized by the disk under perimeter and volume constraints (among simply connected sets). Following the similarity of results in \cite{Bucur-Nahon} with the stability result in Theorem \ref{thm:weight-continuity}, it is possible that Theorem 2.7 from \cite{Bucur-Nahon} regarding the negative answer to the stability of Weinstock's equality could generalize to the Steklov-Lam\'e eigenvalues. We mention also the stability results for Steklov eigenvalues proved in \cite{stability-steklov} which can also be generalized to this case.
	
	(v) The existence result generalizes to shape functionals of the form
	\[ \Omega \mapsto F(\Lambda_1(\Omega),...,\Lambda_k(\Omega)),\]
	where $F:\Bbb{R}^k \to \Bbb{R}$ is upper semi-continuous and increasing in each variable.
\end{rem}

\section{An estimate inspired by Moler and Payne}
\label{sec:moler-payne}

The results found by Moler and Payne in \cite{moler-payne} are classical for studying the precision for the method of fundamental solutions. It gives bounds controlling how the solutions of a PDE change in terms of perturbations of the boundary conditions. An analogue result for the case of Steklov and Wentzell eigenvalues was given in \cite{Bogosel2}. In the following we aim to give such a result related to the Steklov-Lam\'e problem.

Consider the problem $\bo u \in \bo H^1(\Omega)$, verifying
\begin{equation}
\left\{ 
\begin{array}{rcll}
-\di Ae(\bo u) & = & 0 & \text{ in }\Omega \\
Ae(\bo u)\bo n & = & \bo f & \text{ on }\partial \Omega 
\end{array}
\right.
\label{eq:rhs-problem}
\end{equation}
for $\bo f$ in $\bo L^2(\partial \Omega)$. It can be readily be observed that \eqref{eq:rhs-problem} does not have a unique solution, as stated. Moreover, the variational formulation
\begin{equation}
\int_\Omega Ae(\bo u):e(\bo v) = \int_{\partial \Omega} \bo f \cdot \bo v \ \ \ \forall \bo v \in \bo H^1(\Omega),
\end{equation}
implies the compatibility condition $\int_{\partial \Omega} \bo f\cdot \bo v =0$ for every $v \in \bo R(\Omega)$ (the space of rigid motions).

Denote by $\bo V(\partial \Omega)$ the subspace of $\bo L^2(\partial \Omega)$ which is orthogonal to the rigid motions with respect to the usual scalar product. Also denote by $\bo H(\Omega)$ the orthogonal to the space of rigid motions $\bo R(\Omega)$ in $\bo H^1(\Omega)$ with respect to the scalar product in $\bo L^2(\Omega)$. Then define the resolvent operator $\bo{Res} : \bo V(\partial \Omega) \to \bo H(\Omega)$ such that $\bo{Res}(\bo f) = \bo u$, where $\bo u$ solves \eqref{eq:rhs-problem}. The well-posedness of \eqref{eq:rhs-problem} for $\bo u \in \bo H(\Omega)$ and $\bo f \in \bo V(\Omega)$ is classical and discussed, for example, in \cite[Section 3.2]{Sebastian}.

\begin{thm} Let $\Omega$ be a bounded, open domain with Lipschitz boundary. Suppose that $\bo f \in \bo V(\partial \Omega)$ and $\bo u = \bo{Res}(\bo f)$. Then there exists a constant $C$ depending only on $\Omega$ such that
	\[ \|\bo u\|_{\bo L^2(\partial \Omega)} \leq C \|\bo f\|_{\bo L^2(\partial \Omega)}.\]
	\label{thm:estimate-perturbation}
\end{thm}

\begin{proof}  Use the variational formulation to see that
\[ \int_\Omega Ae(\bo u):e(\bo u) = \int_{\partial \Omega}\bo f\cdot \bo u.\]
Since $\bo u \in \bo H(\Omega)$, we can apply Korn's inequality (\cite[Theorem 2.3]{ciarlet-korn}, \cite[Theorem 2.5]{korn-book}) to obtain $\|\bo u\|_{\bo H^1(\Omega)}\leq C_K\|e(\bo u)\|_{\bo L^2(\Omega)}$. Also recall that the trace inequality (see for example \cite{evans-gariepy}) says that there exists a constant $C_\Omega$ depending on the Lipschitz constant of $\partial \Omega$ such that $\|\bo u\|_{\bo L^2(\partial \Omega)} \leq C_\Omega \|\bo u\|_{\bo H^1(\Omega)}$.
Therefore we have the sequence of inequalities:
\begin{align*} \|\bo u\|_{\bo L^2(\partial \Omega)}^2 &\leq C_\Omega^2 \|\bo u\|_{\bo H^1(\Omega)}^2 \leq C_\Omega^2 C_K^2 \|e(\bo u)\|_{L^2(\Omega)}^2\leq \frac{C_\Omega^2C_K^2}{2\mu} \int_\Omega Ae(\bo u):e(\bo u)\\
&\leq \frac{C_\Omega^2C_K^2}{2\mu}\|\bo f\|_{\bo L^2(\partial \Omega)}\|\bo u\|_{\bo L^2(\partial \Omega)}.
\end{align*}
The conclusion follows. \end{proof}

\begin{thm}\label{thm:moler-payne}
	Consider $\Omega$ a bounded open domain with Lipschitz boundary. Suppose $\bo u_\varepsilon, \bo f_\varepsilon$ belong to $\bo V(\partial \Omega)$ and that $\bo u_\varepsilon \in \bo H(\Omega)$ verifies the perturbed equation
	\begin{equation}
	\left\{ 
	\begin{array}{rcll}
	-\di Ae(\bo u_\varepsilon) & = & 0 & \text{ in }\Omega \\
	Ae(\bo u_\varepsilon)\bo n & = & \Lambda_\varepsilon \bo u_\varepsilon + \bo f_\varepsilon & \text{ on }\partial \Omega 
	\end{array}
	\right.
	\label{eq:pert-problem}
	\end{equation}
	Then there exists $n \geq 1$ such that 
	\begin{equation}
	\label{bound_error}
	\|\bo u_\varepsilon\|_{\bo L^2(\partial \Omega)}  |\Lambda_n(\Omega)-\Lambda_\varepsilon| \leq \|\bo f_\varepsilon\|_{\bo L^2(\partial \Omega)}.
	\end{equation}
	
	Furthermore, suppose that there exists $\delta>0$ such that when $\Lambda_k(\Omega)\neq \Lambda_n(\Omega)$ we have $|\Lambda_k(\Omega)-\Lambda_n(\Omega)|>\delta$. Then there exists $\bo u_n, \bo u_r\in \bo H(\Omega)$ such that $\bo u_\varepsilon = {\bo u_n}+\bo u_r$, ${\bo u_n}$ is an eigenfunction associated to $\Lambda_n$ and 
	\[ \|\bo u_r\|_{\bo L^2(\partial \Omega)} = \|\bo u_\varepsilon-\bo  u_n\|_{\bo L^2(\partial \Omega)}\leq \frac{\|\bo f_\varepsilon\|_{\bo L^2(\partial \Omega)}}{\delta}.\]
\end{thm}

\begin{proof}  It is classical that the eigenfunctions $(\bo u_n)$ associated to positive eigenvalues $\Lambda_n(\Omega)>0$ form an orthonormal basis of $\bo V(\partial \Omega)$ in $\bo L^2(\partial \Omega)$. Using the variational formulations for the eigenvalue problem and for problem \eqref{eq:pert-problem} we have
\[ (\Lambda_n(\Omega)-\Lambda_\varepsilon) \int_{\partial \Omega} \bo u_n \cdot \bo u_\varepsilon = \int_{\partial \Omega} \bo f_\varepsilon \cdot \bo u_n , \forall n \geq 1\]
Denoting $a_n = \int_{\partial \Omega} \bo u_n \cdot \bo u_\varepsilon$ and $b_n = \int_{\partial \Omega} \bo f_\varepsilon \cdot \bo u_n$ the Fourier coefficients of $\bo u_\varepsilon$ and $\bo f_\varepsilon$ in the orthonormal basis given by the eigenfunctions $(\bo u_n)_{n\geq 1}$ we know that 
\[ \|\bo u_\varepsilon\|_{\bo L^2(\partial \Omega)}^2 = \sum_{n\geq 1}a_n^2,\ \ \|\bo f_\varepsilon\|_{\bo L^2(\partial \Omega)}^2 = \sum_{n\geq 1} b_n^2.\]
Therefore, choosing $n$ such that $|\Lambda_n(\Omega)-\Lambda_\varepsilon|$ is minimal (such $n$ exists since $\Lambda_n(\Omega) \to \infty$) we have
\[ \|\bo f_\varepsilon\|_{\bo L^2(\partial \Omega)} \geq |\Lambda_n(\Omega)-\Lambda_\varepsilon| \|\bo u_\varepsilon\|_{\bo L^2(\partial \Omega)}\]

Denote by $I_n = \{ k \geq 1 \text{ such that }\Lambda_k(\Omega)\neq \Lambda_n(\Omega)\}$. Suppose that there exists $\delta>0$ such that  $|\Lambda_k(\Omega)-\Lambda_n(\Omega)| >\delta$ for every $k \in I_n$. In view of the above computations we have $a_k^2 = b_k^2/(\Lambda_k(\Omega)-\Lambda_\varepsilon)^2 \leq b_k^2/\delta^2$ for every $k \in I_n$. This implies that $\bo u_\varepsilon$ can be written as $\bo u_\varepsilon ={\bo u_n}+\bo u_r$ such that
\begin{itemize}
	\item ${\bo u_n}$ verifies the Steklov-Lam\'e eigenvalue problem \eqref{eq:steklov-lame} for $\Lambda=\Lambda_n(\Omega)$.
	\item $\bo u_r$ verifies $\|\bo u_r\|_{L^2(\partial \Omega)}^2 \leq \|\bo f_\varepsilon\|_{L^2(\partial \Omega)}^2/\delta^2$.
\end{itemize}
\end{proof} 

\begin{rem} Theorem \ref{thm:moler-payne} motivates our numerical method in view of the following arguments.
	
	(i) If $\|\bo f_\varepsilon\|_{\bo L^2(\partial \Omega)}$ is small enough then either $\|\bo u_\varepsilon\|_{\bo L^2(\partial \Omega)}$ small or there exists a Steklov-Lam\'e eigenvalue $\Lambda_n$ that is close to $\Lambda_\varepsilon$.
	
	(ii) If  $\|\bo f_\varepsilon\|_{\bo L^2(\partial \Omega)}$ is small and there exists $n$ such that $|\Lambda_n-\Lambda_\varepsilon|$ is small enough then the solution $\bo u_\varepsilon$ is close to an actual Steklov-Lam\'e eigenfunction. In case the eigenspace of $\Lambda_n$ is of dimension one then the result implies that $\|\bo u_\varepsilon-\bo u_n\|_{\bo L^2(\partial \Omega)}$ is small, i.e. the approximate eigenfunction is close to the original one. 
	
	In case the eigenspace has higher dimension, for multiple eigenvalues, the result says that $\bo u_\varepsilon$ can be decomposed using an eigenfunction for $\Lambda_n$ and a remainder term which is small compared to the error $\|\bo f_\varepsilon \|_{\bo L^2(\partial \Omega)}$. 
	
\end{rem}

\section{Numerical methods}
\label{sec:num-methods}
We consider the numerical solution of eigenvalue problem \eqref{eq:steklov-lame} in dimension two using the Method of Fundamental Solutions (MFS) whose approximation can be justified by density results (eg.~\cite{Alves, Alves-Martins}).

We define the tensor
\[\bo{\Phi}_{y}(x):=\bo\Phi(x-y),\]
where
\[\left[\bo\Phi(x)\right]_{i,j}=\frac{\lambda+3\mu}{4\pi\mu(\lambda+2\mu)}\left(-\log|x|\delta_{i,j}+\frac{\lambda+\mu}{\lambda+3\mu}\frac{x_ix_j}{|x|^2}\right)\]
is the fundamental solution of the Lam\'{e} equations and the single layer potential is defined by
\[ (\bo S_{\Gamma}\phi)(x)=\int_{\Gamma} \bo\Phi_{y}(x)\phi(y)ds_{y},\]
for some Jordan curve $\Gamma$ and $\phi\in H^{-\frac{1}{2}}(\Gamma)^2.$ Denote by $\tau_\gamma$ the trace operator defined on some boundary $\gamma$ and define the operator
\begin{align*}
\bo B_\gamma&:H^{-\frac{1}{2}}(\gamma)^2\rightarrow H^{\frac{1}{2}}(\partial\Omega)^2&\\
\bo B_\gamma&\phi=\tau_{\partial\Omega}(\bo S_{\gamma}\phi).
\end{align*}

\begin{thm} 
	\label{densidade}
	Let $\hat{\Omega}\subset\mathbb{R}^2$ be a bounded simply connected domain such that $\bar{\Omega}\subset\hat{\Omega}.$ Then, $\bo B_{\partial\hat{\Omega}}$ has dense range in the functional space
	\[H_\ast^{\frac{1}{2}}(\partial\Omega)^2=\left\{\psi\in H^{\frac{1}{2}}(\partial\Omega)^2:\int_{\partial\Omega}\psi(x)ds_x=\bo0\right\}.\]
	
\end{thm}
\begin{proof} 
We will prove that the adjoint $\bo B_{\partial\hat{\Omega}}^\ast=\tau_{\partial\hat{\Omega}}\bo S_{\partial\Omega}$ is injective by verifying that $\text{Ker}(\bo B_{\partial\hat{\Omega}}^\ast)=\left\{\bo 0\right\}.$ Let $\phi\in H_\ast^{\frac{1}{2}}(\partial\Omega)^2$ such that $\bo B_{\partial\hat{\Omega}}^\ast(\phi)=\bo 0$ and define
\[\bo u=\bo S_{\partial\Omega}\phi.\]
Note that $\bo u$ satisfies $\di A(e(\bo u))  =  \bo 0 $ in $\mathbb{R}^2\backslash\partial\Omega$ and solves the exterior problem
\begin{equation}\label{probext}
\left\{\begin{array}{rcll}
\di A(e(\bo u)) & = & \bo{0} & \text{ in } \mathbb{R}^2\backslash\bar{\hat{\Omega}} \\
\bo u & = &\bo 0 & \text{ on } \partial\hat{\Omega} \\
\bo u (x)&=&\bo c \log|x|+\mathcal{O}(1) & |x|\rightarrow\infty,
\end{array}\right.
\end{equation}
where (cf. \cite{Chen-Zhou})
\[\bo c=-\frac{\lambda+3\mu}{4\pi\mu(\lambda+2\mu)}\int_{\partial\Omega}\phi(y)ds_y=\bo 0\]	
because $\phi\in H_\ast^{\frac{1}{2}}(\partial\Omega)^2.$
Thus,~\eqref{probext} is well posed with trivial solution $\bo u=0.$ By analytic continuation, the external trace of $\bo u$ on $\partial\Omega$ and the external trace of the surface traction vector $Ae(\bo u)\bo n$ are both null. By continuity of the single layer operator through the boundary, the inner trace of $\bo u$ at $\partial\Omega$ is null which implies that $\bo u$ solves
\begin{equation}
\left\{\begin{array}{rcll}
\di A(e(\bo u)) & = & \bo 0 & \text{ in } \Omega \\
\bo u & = &\bo 0 & \text{ on } \partial\Omega,
\end{array}\right.
\label{probint}
\end{equation}
which is well posed with trivial solution $\bo u=\bo0.$ Thus, the inner traces (on $\partial\Omega$) of $\bo u$ and of the surface traction vector $Ae(\bo u)\bo n$ are both null. Taking into account that $\phi$ is equal to the jump of the surface traction vector at the boundary (cf.~\cite{Chen-Zhou}), we conclude that $\phi=\bo 0$ and $\text{Ker}(\bo B_{\partial\hat{\Omega}}^\ast)=\left\{\bo 0\right\}$.\end{proof}
\begin{thm}
	\label{betti}
	Let $\Omega$ be a sufficiently smooth domain, in such a way that the eigenfunctions are in $\bo H^1(\Omega)$. Then, the traces on $\partial\Omega$ of the eigenfunctions associated with positive eigenvalues belong to $H_\ast^{\frac{1}{2}}(\partial\Omega)^2.$
\end{thm}
\begin{proof}
It follows from Betti's formula (cf.~\cite{mclean}), taking $\bo u$ to be an eigenfunction associated to a positive eigenvalue and $\bo v\equiv1$ we get
\[\int_\Omega\Big(\underbrace{\di A(e(\bo u))}_{=0}\cdot \bo v-\underbrace{\di A(e(\bo v))}_{=0}\cdot \bo u\Big)dx=\int_{\partial\Omega}\Big(\underbrace{Ae(\bo u)\bo n}_{=\Lambda(\Omega) \bo u}\cdot \bo v- \underbrace{Ae(\bo v)\bo n}_{=0}\cdot \bo u\Big)ds.\]
Thus,
\[\Lambda(\Omega)\int_{\partial\Omega}\bo uds=\bo 0\]
and since $\Lambda(\Omega)>0$ we conclude that $\int_{\partial\Omega}\bo u ds=\bo 0$ and the conclusion follows.\end{proof}
\begin{rem}
	Theorem~\ref{densidade} ensures density of the traces of the single layer in $H_\ast^{\frac{1}{2}}(\partial\Omega)^2$ and by Theorem~\ref{betti} density in the space of traces of the eigenfunctions associated to positive eigenvalues. Using a similar argument it could be proven that the traces on $\partial\Omega$ of
	\[\bo v=(\bo S_{\partial\hat{\Omega}}\phi)+\alpha_1\begin{pmatrix}1\\0\end{pmatrix}+\alpha_2\begin{pmatrix}0\\1\end{pmatrix},\ \alpha_1,\alpha_2\in\mathbb{R}\]
	are dense in $H^{\frac{1}{2}}(\partial\Omega)^2$ (cf.~\cite{Alves-Martins}).
\end{rem}

Taking into account Theorem~\ref{densidade} we consider a bounded simply connected domain $\hat{\Omega}\subset\mathbb{R}^2$ such that $\bar{\Omega}\subset\hat{\Omega}$ and place $N$ source points on the boundary $\partial\hat{\Omega}.$ 
The Method of Fundamental Solutions (MFS) approximation is a discretization of the single layer operator,
\begin{equation}
\label{mfs}
\bo u(x)=(\bo S_{\partial\hat{\Omega}}\phi)(x)\approx \bo u_N(x)=\sum_{j=1}^N\bo\Phi_{\bo y_j}(x)\cdot\bo a_j,\ \bo a_j\in\mathbb{R}^2.
\end{equation}
The MFS linear combination satisfies Lam\'{e} equations. A straightforward approach for the calculation of the positive Steklov-Lam\'{e} eigenvalues could be defining $N$ collocation points $\bo x_1,...,\bo x_M\in\partial\Omega$ and the unit outward vectors at these points $\bo n_1,...,\bo n_M$ and imposing the boundary condition of problem~\eqref{eq:steklov-lame}. Taking $M=N$ would lead to the solution of a generalized eigenvalue 
\begin{equation}
\label{mateigprob}
\bo A\cdot \bo X=\lambda \bo B\cdot \bo X,
\end{equation}
where
\[\left[\bo A\right]_{i,j}=Ae(\bo \Phi_{\bo y_j} (\bo x_i))\bo n_i \quad\text{and}\quad\left[\bo B\right]_{i,j}=\bo \Phi_{\bo y_j}(\bo x_i)\]
with square matrices $\bo A$ and $\bo B.$ However, several numerical tests revealed that a better approach would be to consider oversampling. Consider $M>N$, for instance $M=2N$ and instead of the generalized eigenvalue problem~\eqref{mateigprob}, compute the factorization $\mathbf{B}=\mathbf{QR}$. Afterwards, solve the generalized eigenvalue problem
\begin{equation}
\label{mateigprob2}
\left(\bo Q'\cdot \bo A\right)\bo\cdot \bo X=\lambda \bo R\cdot\bo X.
\end{equation}

As in previous studies of the application of the MFS for eigenvalue problems (eg.~\cite{Alves-Antunes_2013,Bogosel2}) we define the source points for the MFS by
\[\bo y_j=\bo x_j+\alpha\bo n_j,\]
for a small positive parameter $\alpha$. Details regarding the choice of this parameter are given in the presentation of the numerical results.

We will consider the numerical optimization of Steklov-Lam\'{e} eigenvalues in two classes of domains:
\begin{itemize}
	\item simply connected planar domains with fixed area;
	\item planar convex domains with fixed area for which Theorem \ref{thm:existence-convex} ensures the existence of solutions.
\end{itemize}

We parametrize the boundary of a general simply connected planar domain by
\begin{equation}
\label{defdom1}
\left\{(h_1(t),h_2(t)):t\in[0,2\pi[\right\},
\end{equation}
for some $2\pi$-periodic functions $h_1$ and $h_2$ such that \eqref{defdom1} defines a Jordan curve. 
We consider the approximations
\[h_1(t)\approx \gamma_1(t)=a_0^{(1)}+\sum_{j=1}^Pa_j^{(1)}\cos(jt)+\sum_{j=1}^Pb_j^{(1)}\sin(jt)\]
and
\[h_2(t)\approx \gamma_2(t)=a_0^{(2)}+\sum_{j=1}^Pa_j^{(2)}\cos(jt)+\sum_{j=1}^Pb_j^{(2)}\sin(jt),\]
for some $P\in\mathbb{N}.$ 

The boundary of a planar convex domain is defined by \eqref{defdom1} where
\[ h_1(t) = p(t) \cos(t) - p'(t) \sin(t),\]
\[h_2(t)=p(t) \sin(t) + p'(t) \cos(t)\]
and the support function $p$ is approximated by a truncated Fourier series
\begin{equation}
p(t) = a_0 + \sum_{k=1}^\mathcal{P} \left( a_k \cos (kt) + b_k \sin (kt)  \right).
\label{fouDecomp}
\end{equation}
Note that in this case the convexity constrain correspond to the constraint
\begin{equation}
\label{conv_support}
p''(t)+p(t)\geq0,\ \forall t\in[0,2\pi[.
\end{equation}
Numerically, we will impose \eqref{conv_support} at a discrete set of points in $[0,2\pi[$ which leads to a system of (linear) inequalities (see~\cite{Antunes-Bogosel} for details).

The numerical optimizers are found by determining optimal coefficients in the previous expansions using a gradient type method, as in previous computational studies of extremal eigenvalues~\cite{ak-kao-osting,Antunes-Bogosel,Bogosel2,kao-osting-oudet}. The derivative of each eigenvalue with respect to the variation of each coefficient is calculated by the formula for the shape derivative obtained in~\cite{CDM}. Given a deformation field $\bo V$, the derivative of a (simple) Steklov-Lam\'{e} eigenvalue is given by
\begin{multline*}\Lambda'(\Omega,\bo V)=\int_{\partial\Omega}\Big(Ae(\bo u):e(\bo u)-4Ae(\bo u)\bo n\cdot\Pi e(\bo u)\bo n\\
 -\Lambda(\Omega)\bo u\cdot(\mathcal H\bo u+2\partial_n{\bo u}-4\Pi e(\bo u)\bo n)\Big)\bo V\cdot\bo nds,
\end{multline*}
where $\mathcal H$ is the curvature and $\Pi=\begin{pmatrix}1 &0\\0&0\end{pmatrix}$.

\section{Numerical results}
\label{sec:num-results}

In this section we present numerical results obtained using the numerical framework presented previously. All numerical computations and figures are done in Matlab. 

We start with some tests performed for the disk with unit area. Figure~\ref{fig:convergence_disk} shows the convergence curve of the MFS in the calculation of four eigenvalues ($\Lambda_i,\ i=1,4,20,100$) for $\lambda=1,\ \mu=0.5$ (left plot) and $\lambda=1,\ \mu=3$ (right plot). These results were obtained with $\alpha=0.015$. We note that the matrices involved in the generalized eigenvalue problem are highly ill conditioned if we take large values of $\alpha$ and new techniques for reducing the ill conditioning may be needed (eg.~\cite{Antunes_illcond}). It can be observed that the precision of the computations increases with the number of fundamental solutions $N$. For eigenvalues of small index the precision gets close to machine precision.

\begin{figure}[ht]
	\centering 
	\includegraphics[width=0.49\textwidth]{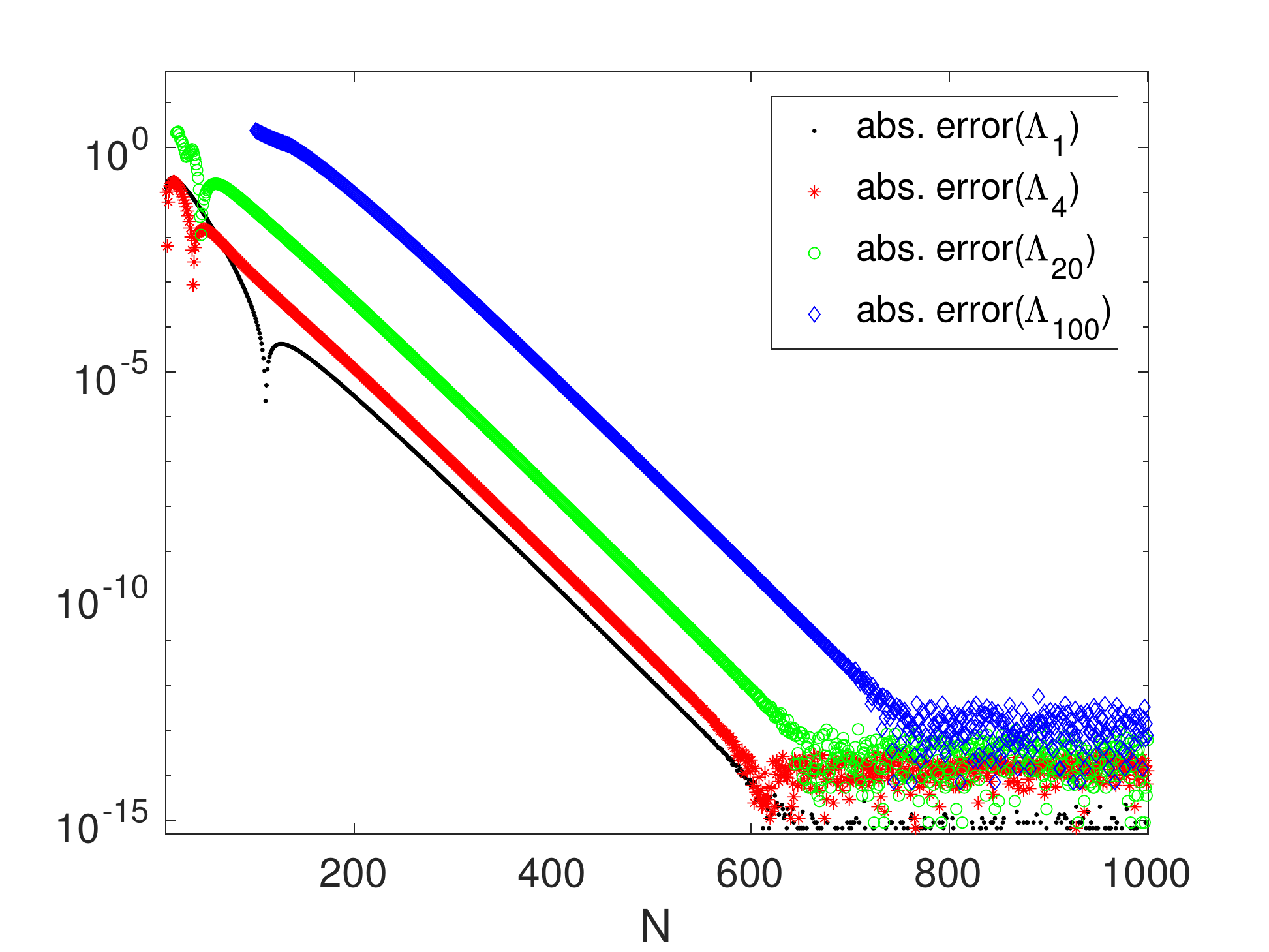}
	\includegraphics[width=0.49\textwidth]{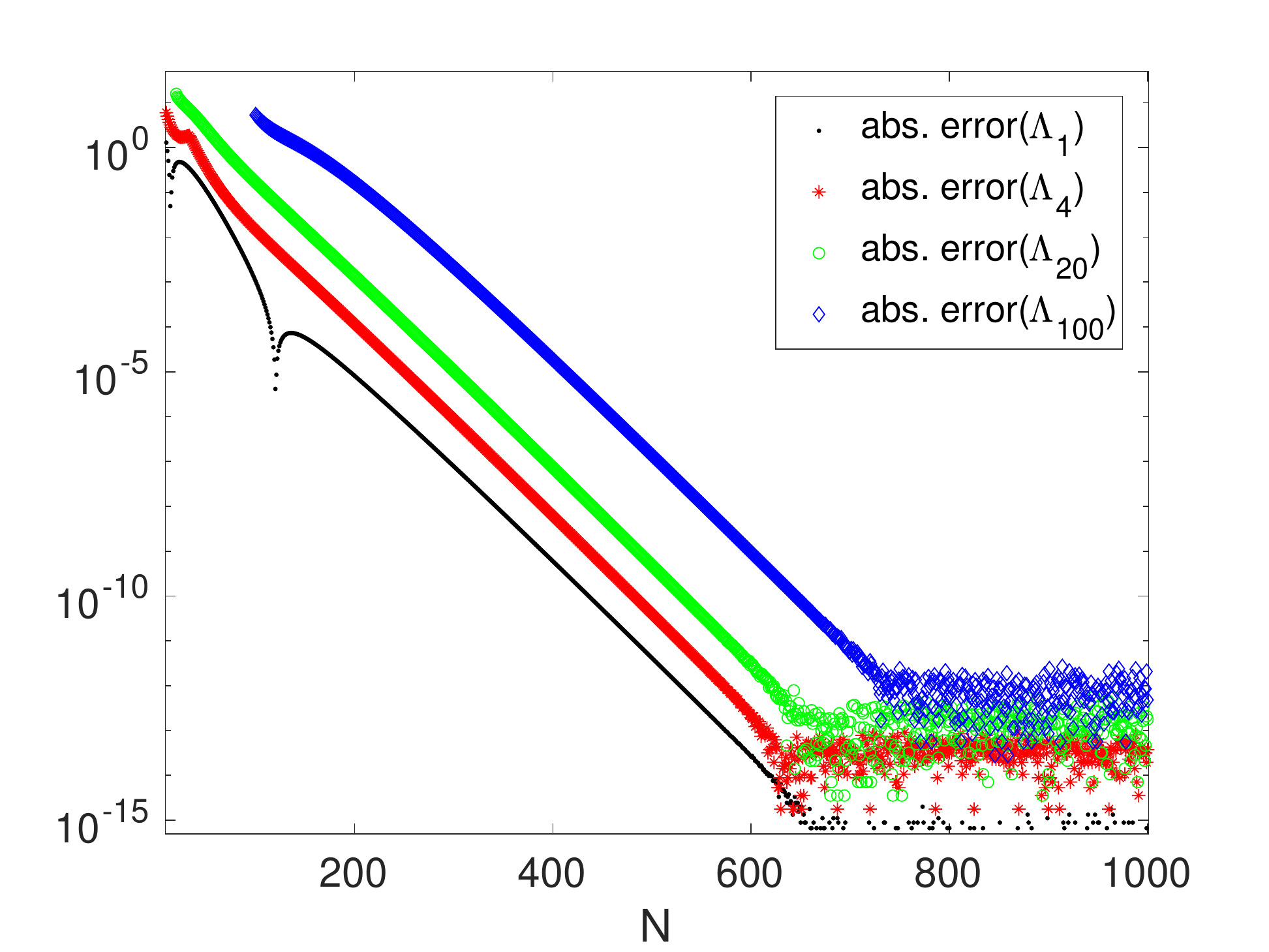}
	\caption{Plots of the absolute errors of the eigenvalues $\Lambda_i,\ i=1,4,20,100$ of the disk with unit area, as a function of the number of MFS basis functions for $\lambda=1,\ \mu=0.5$ (left plot) and $\lambda=1,\ \mu=3$ (right plot)}  \label{fig:convergence_disk}
\end{figure}

Figure~\ref{fig:disk_mu} shows the plot of the first 50 Steklov-Lam\'{e} eigenvalues of the disk with unit area, as a function of $\mu\in[0,10]$, keeping $\lambda=1$ (left plot). The right plot of the same Figure shows a zoom for $\mu\in[0,1].$
\begin{figure}[ht]
	\centering 
	\includegraphics[width=0.49\textwidth]{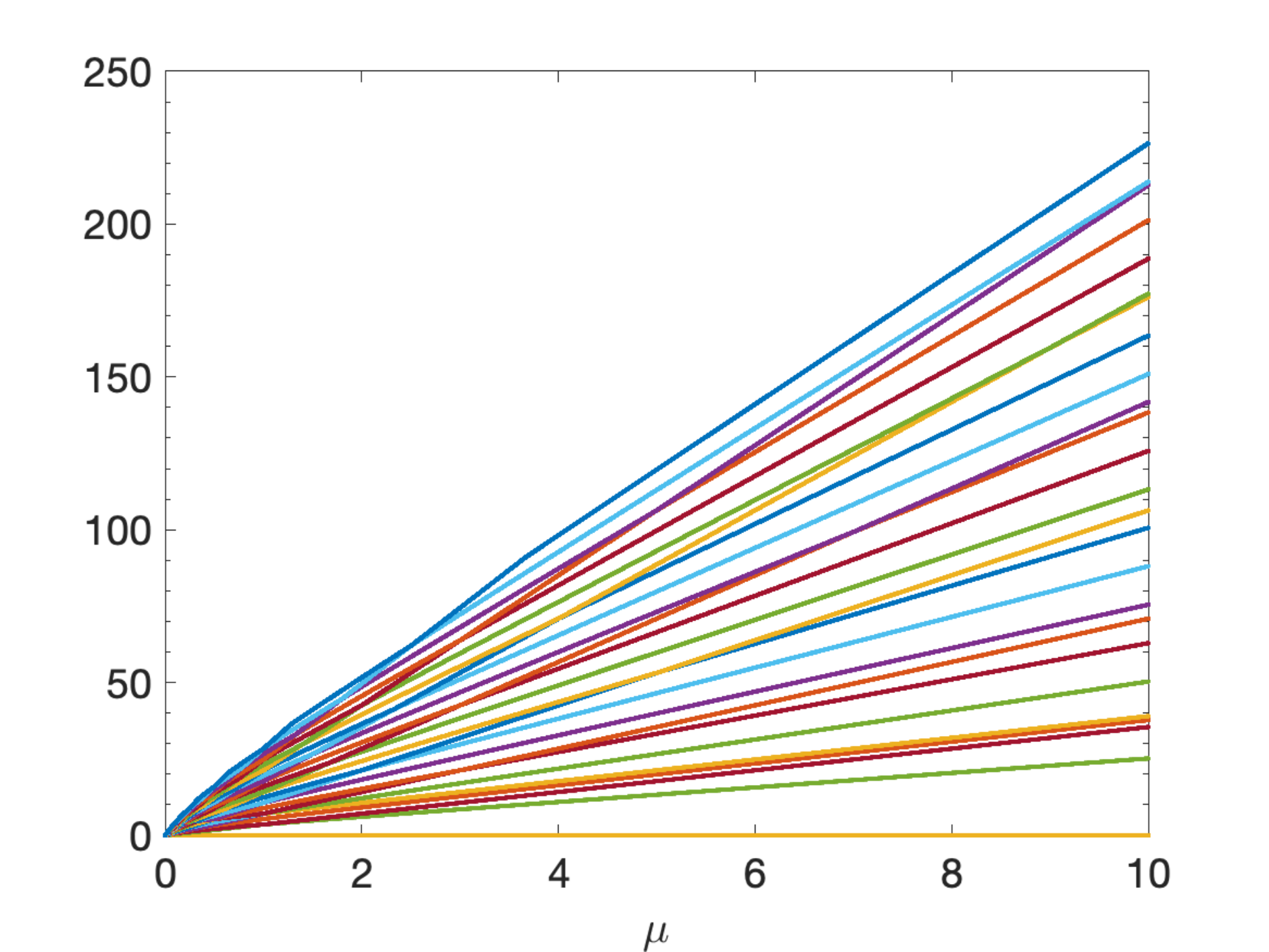}
	\includegraphics[width=0.49\textwidth]{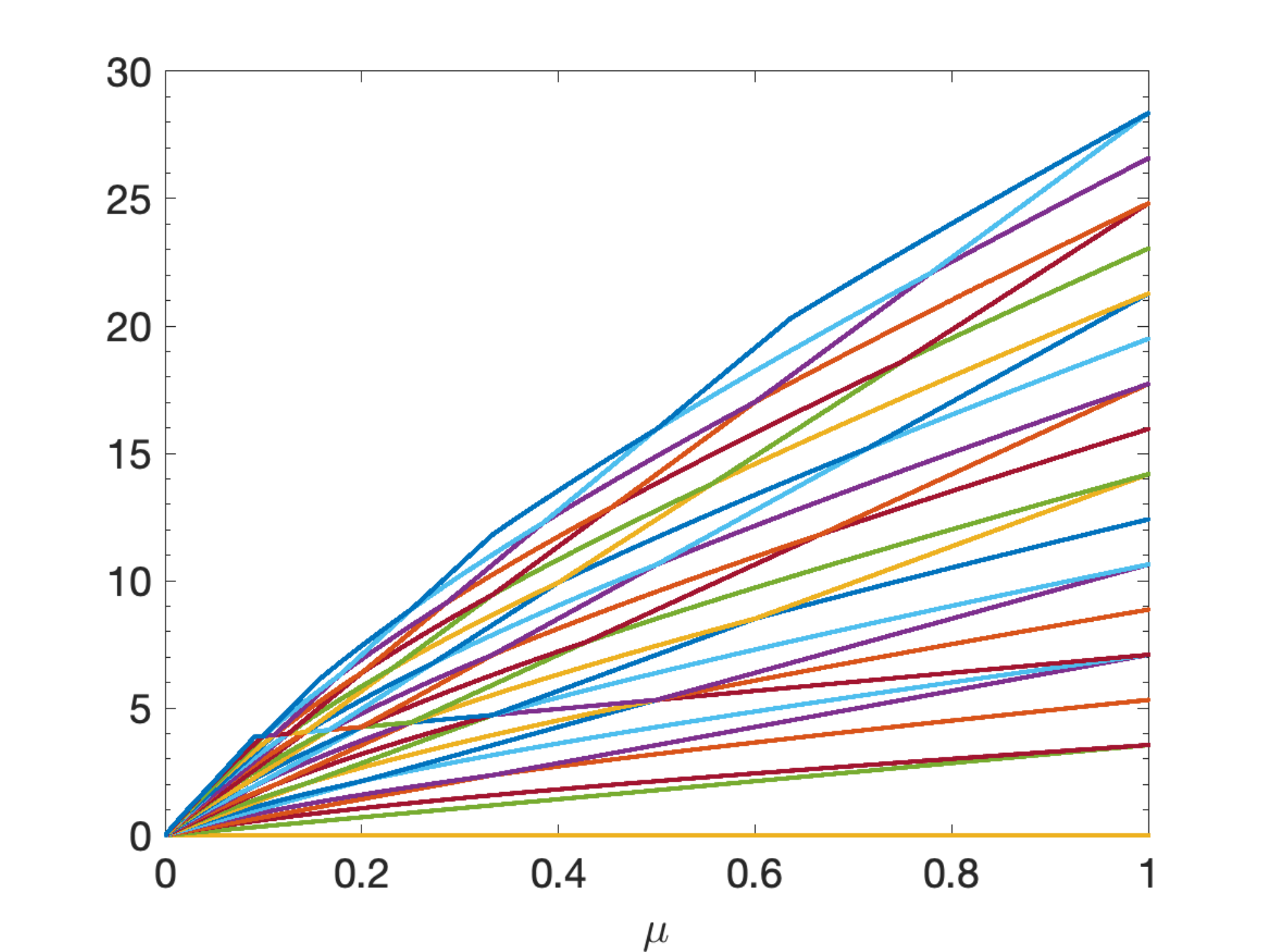}
	\caption{Plot of $\lambda_i,\ i=1,2,...,50$, for the disk with unit area with $\lambda=1$ and $\mu\in[0,10]$ (left plot) and a zoom for $\mu\in[0,1]$ (right plot).} \label{fig:disk_mu}
\end{figure}

Next, we show some numerical results concerning the computation of the Steklov-Lam\'e eigenvalues and eigenfunctions using the MFS for the domain $\Omega_1$ whose boundary is defined by
\[\partial\Omega_1=\left\{\left(\cos(t),\sin(t)+\frac{3}{10}\sin(3t)\right):\ t\in[0,2\pi[\right\}.\]
Figure~\ref{fig:pts_mfs} shows the plot of the collocation points (marked with $\txtr{\bullet}$) and source points (marked with $\txtb{\circ}$) for $\Omega_1$.
\begin{figure}[ht]
	\centering 
	\includegraphics[width=0.49\textwidth]{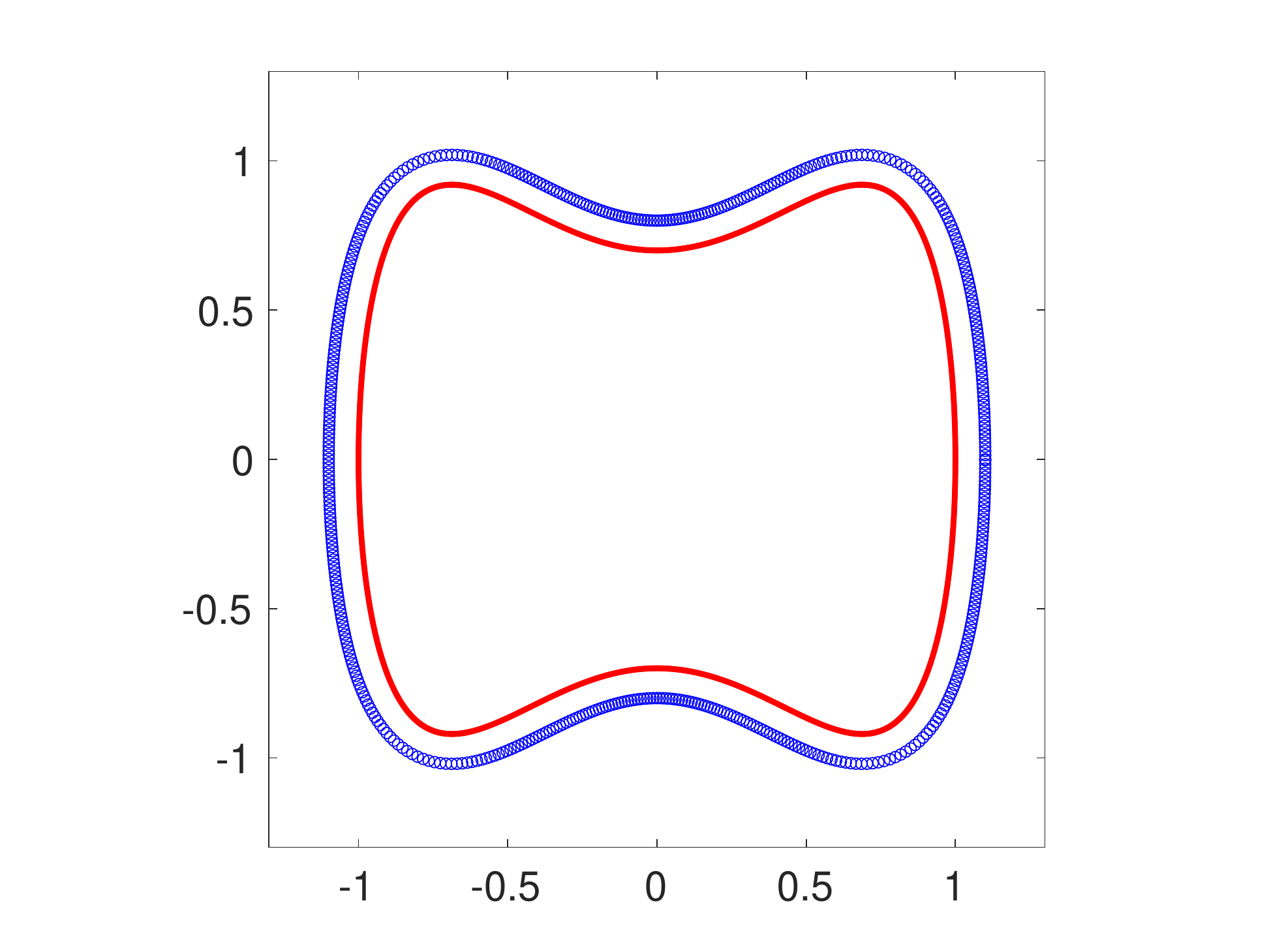}
	\caption{Plot of the collocation points (marked with $\txtr{\bullet}$) and source points (marked with $\txtb{\circ}$) for $\Omega_1$.} \label{fig:pts_mfs}
\end{figure}
Figure~\ref{fig:eig_funct} shows the plots of the first and second components for the eigenfunctions associated to the eigenvalues $\Lambda_i,\ i=1,2,7,20,100$ of $\Omega_1$ with $(\lambda,\mu)=(1,0.5)$ (left-hand side plots) and $(\lambda,\mu)=(1,3)$ (right-hand side plots).
\begin{figure*}
	\centering
	\begin{tabular}{ccm{1cm}cc}
		\includegraphics[width=0.17\textwidth]{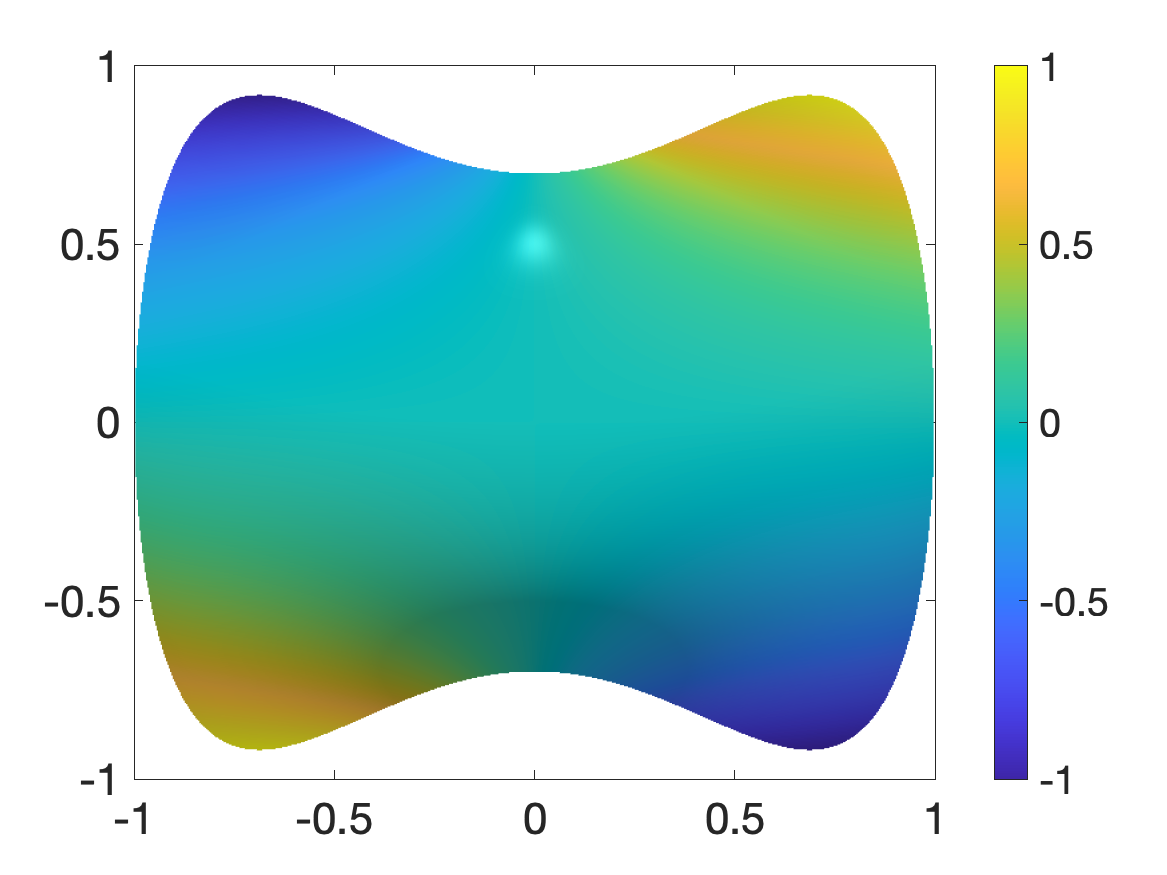}&
		\includegraphics[width=0.17\textwidth]{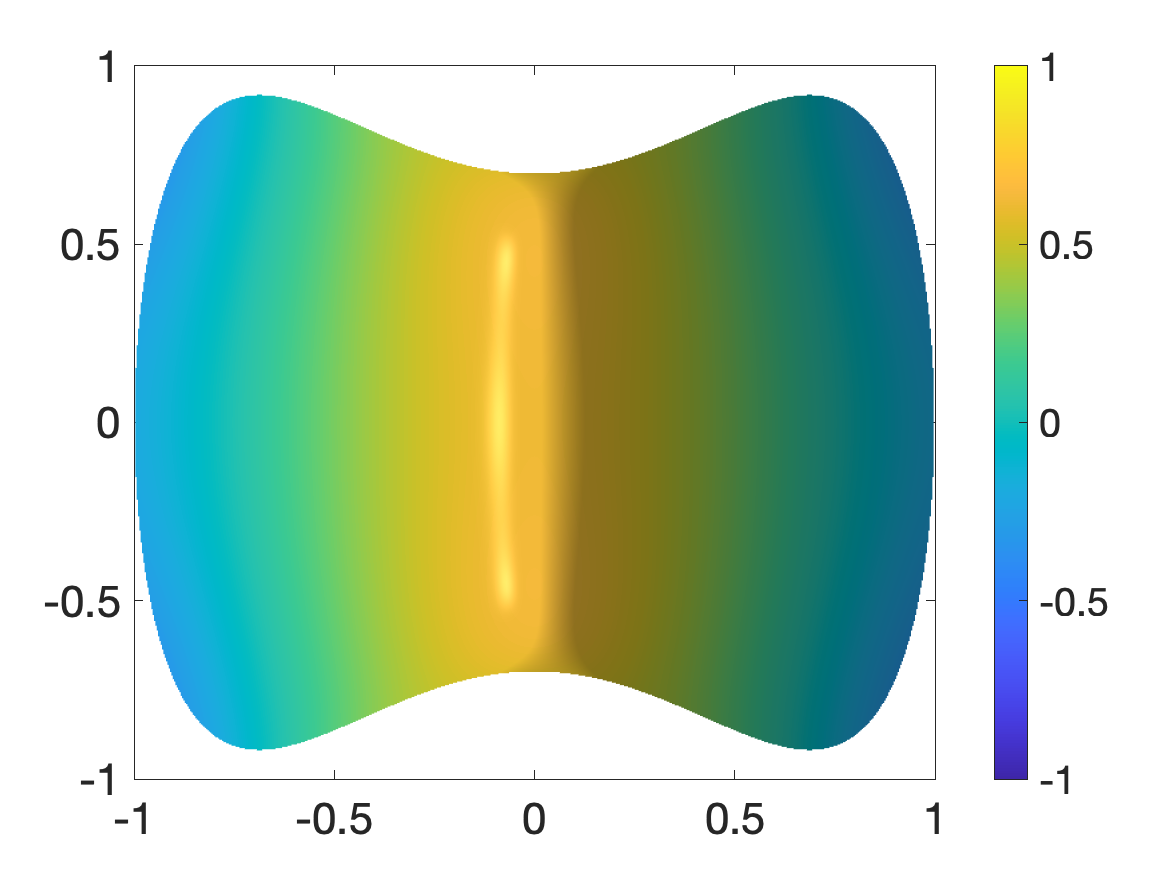}&  $\Lambda_1$ &
		\includegraphics[width=0.17\textwidth]{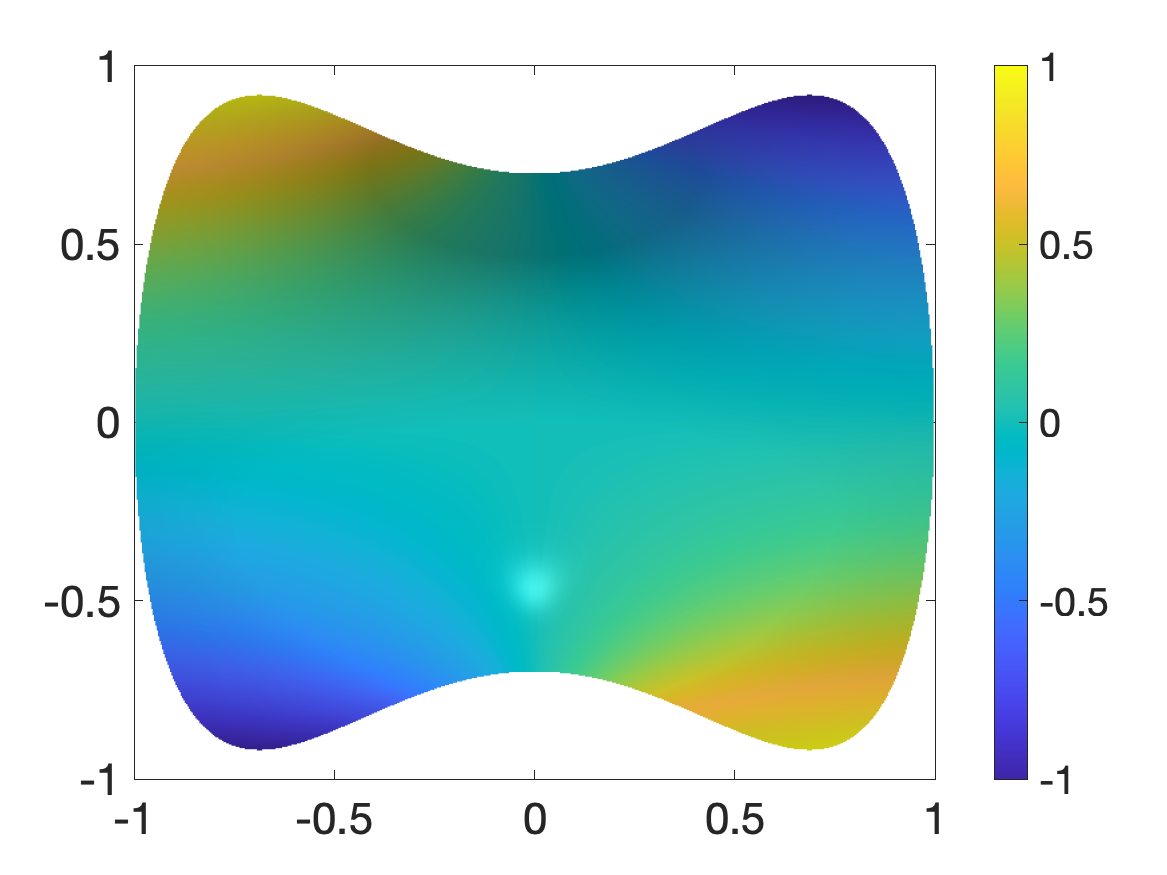}&
		\includegraphics[width=0.17\textwidth]{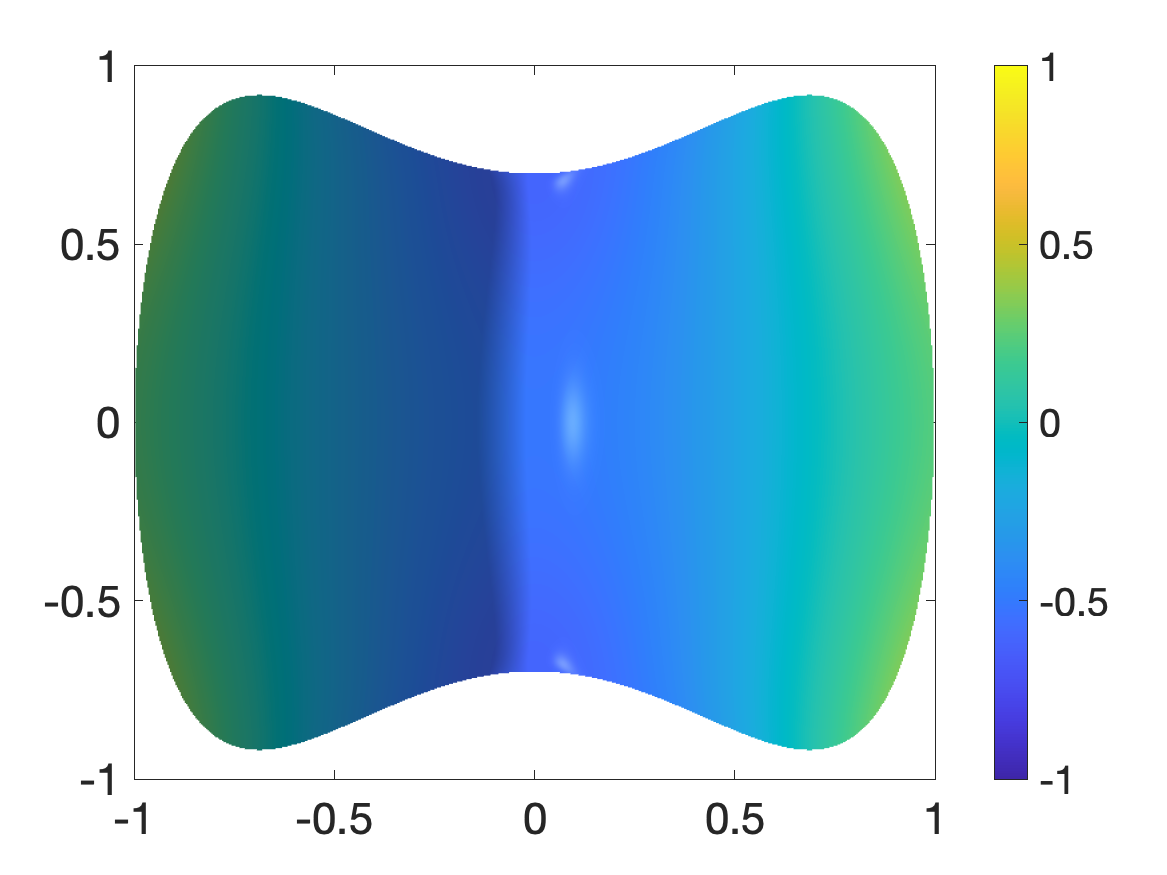}\\
		\includegraphics[width=0.17\textwidth]{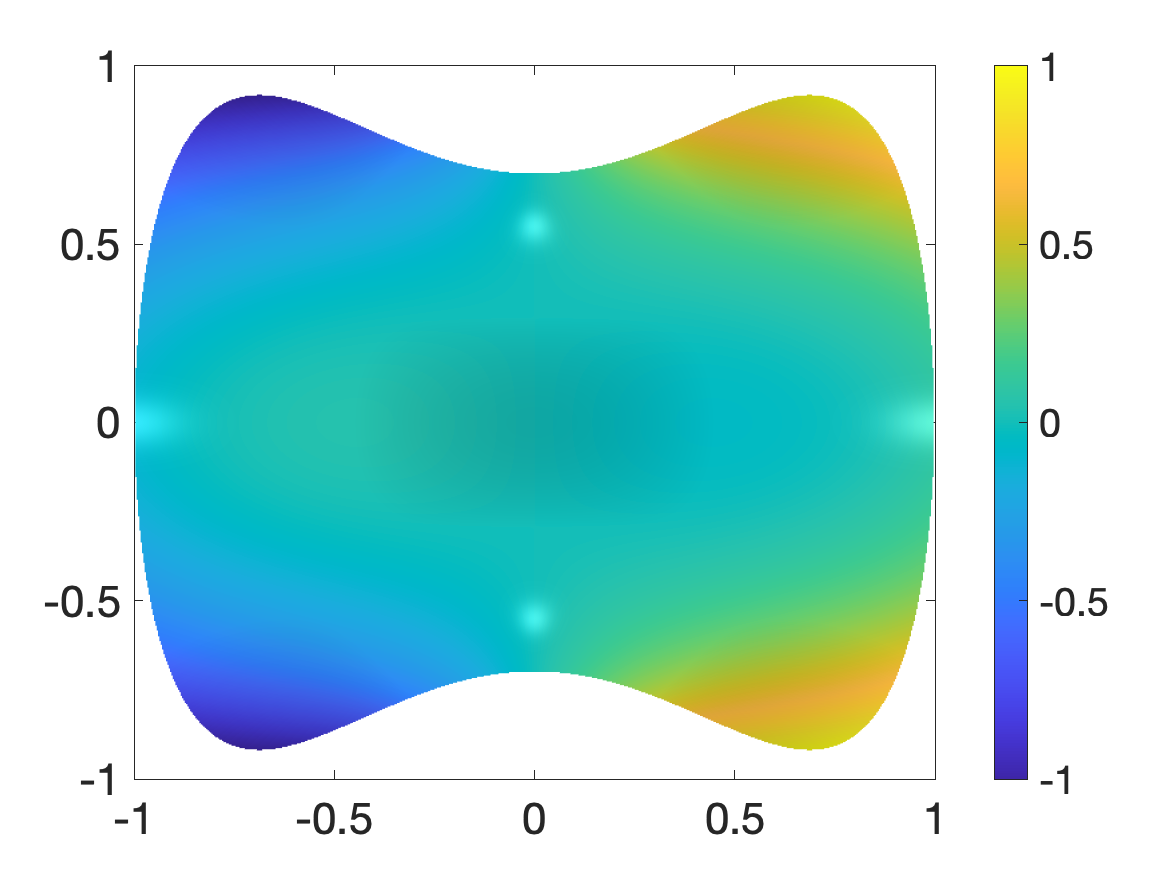}&
		\includegraphics[width=0.17\textwidth]{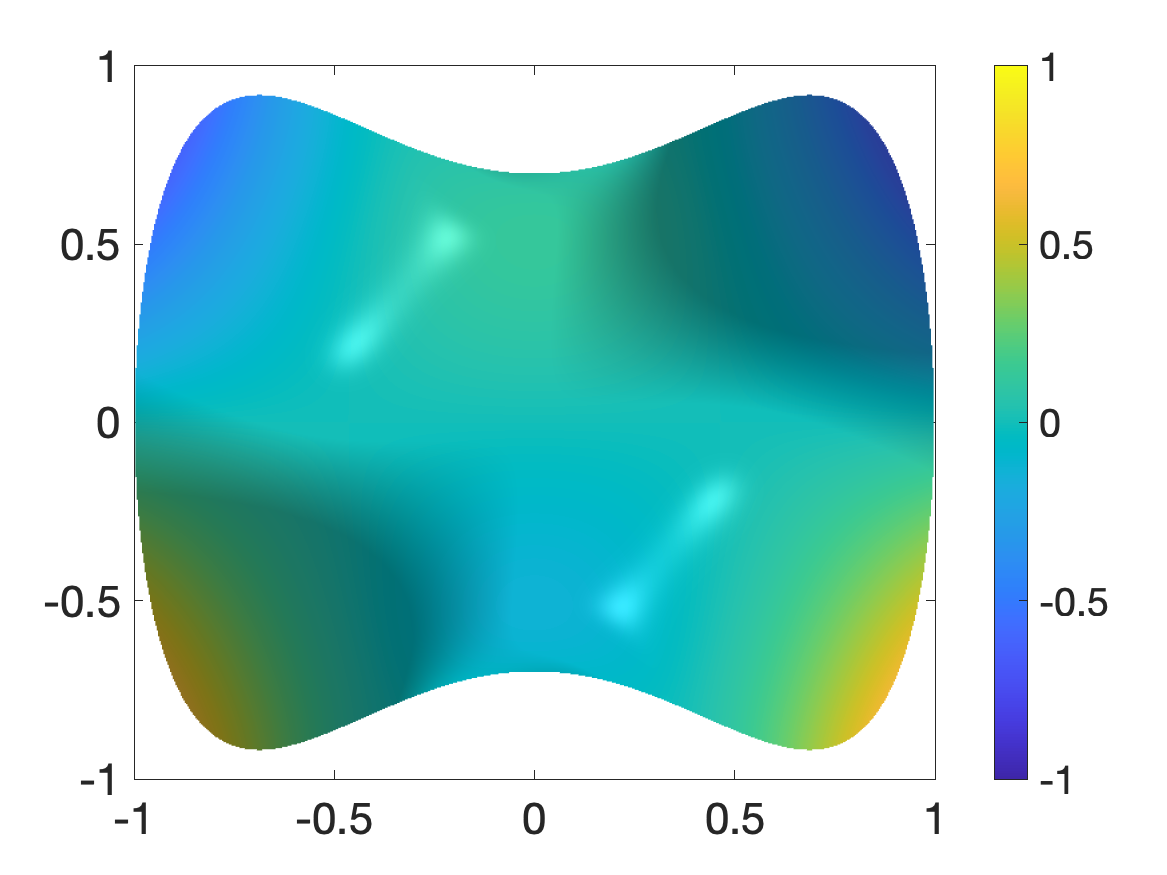}&$\Lambda_2$ &
		\includegraphics[width=0.17\textwidth]{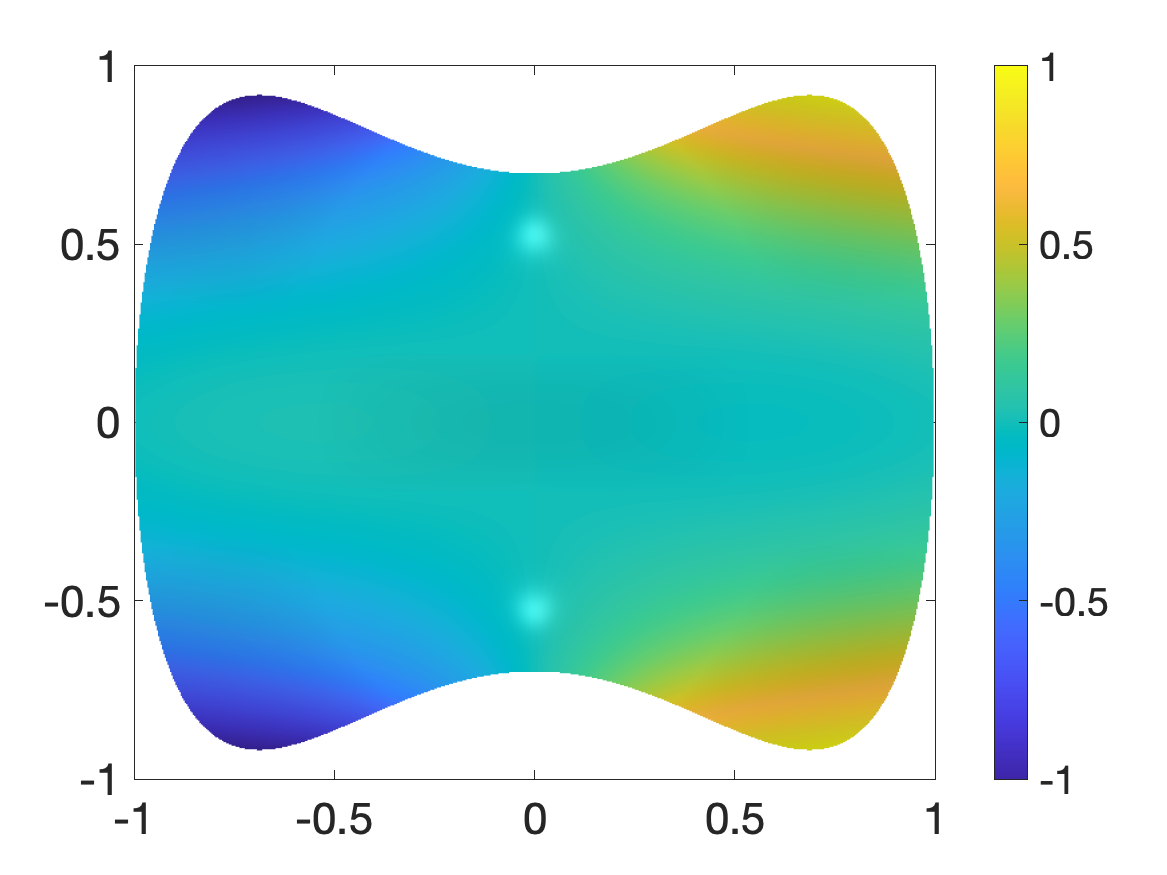}&
		\includegraphics[width=0.17\textwidth]{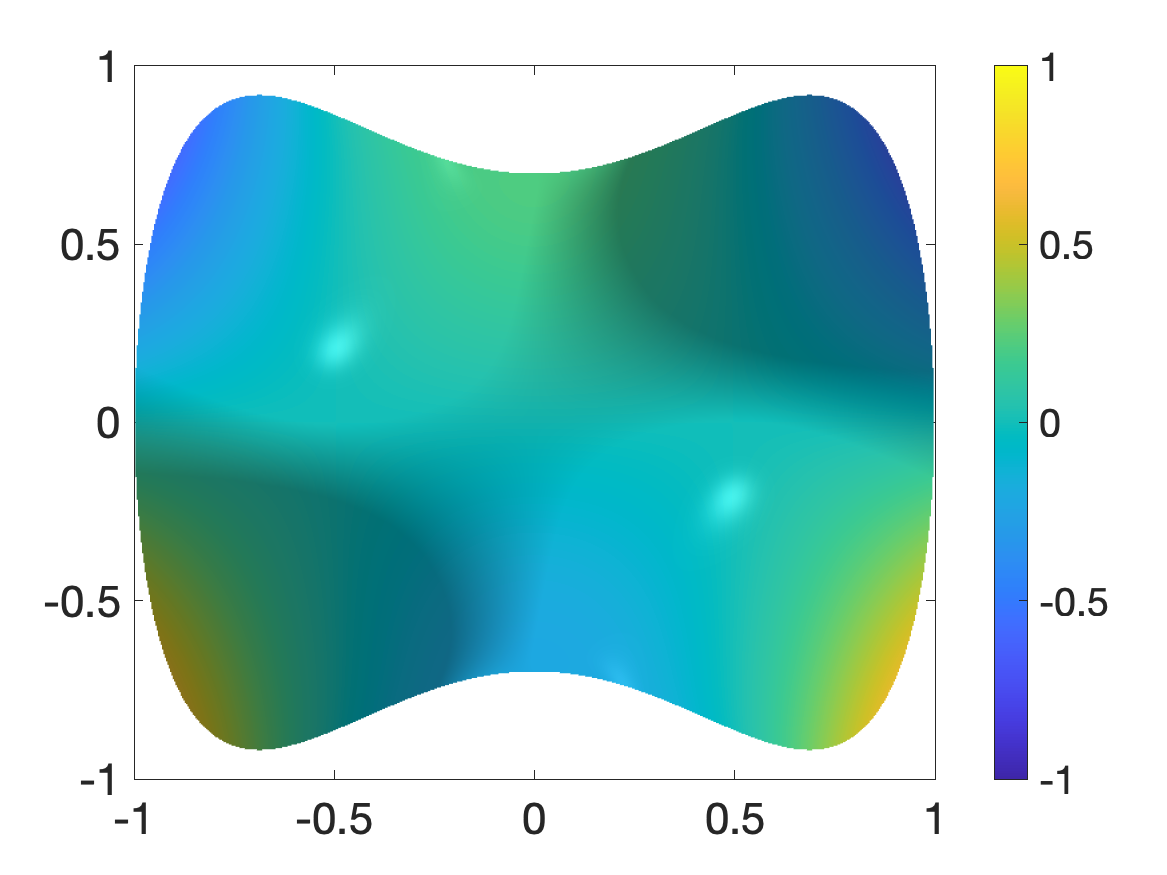}\\
		
		\includegraphics[width=0.17\textwidth]{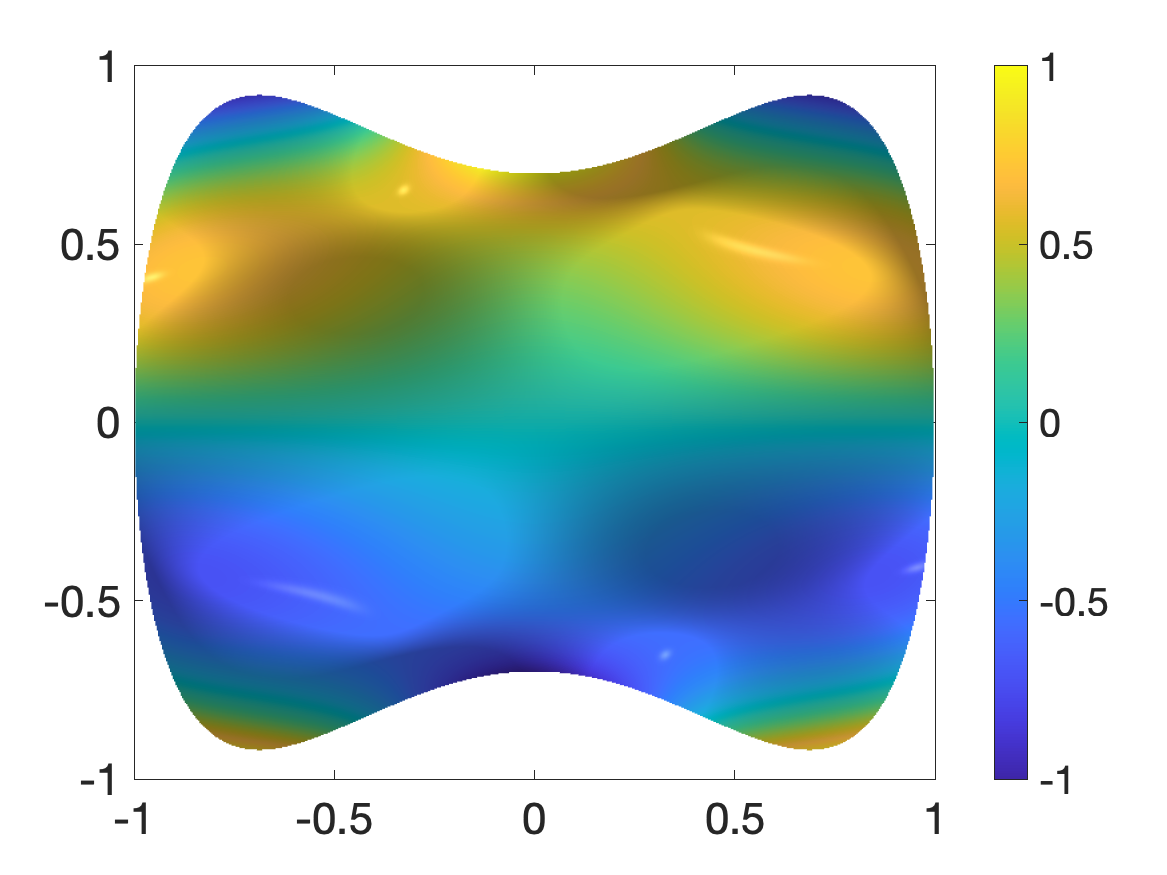}&
		\includegraphics[width=0.17\textwidth]{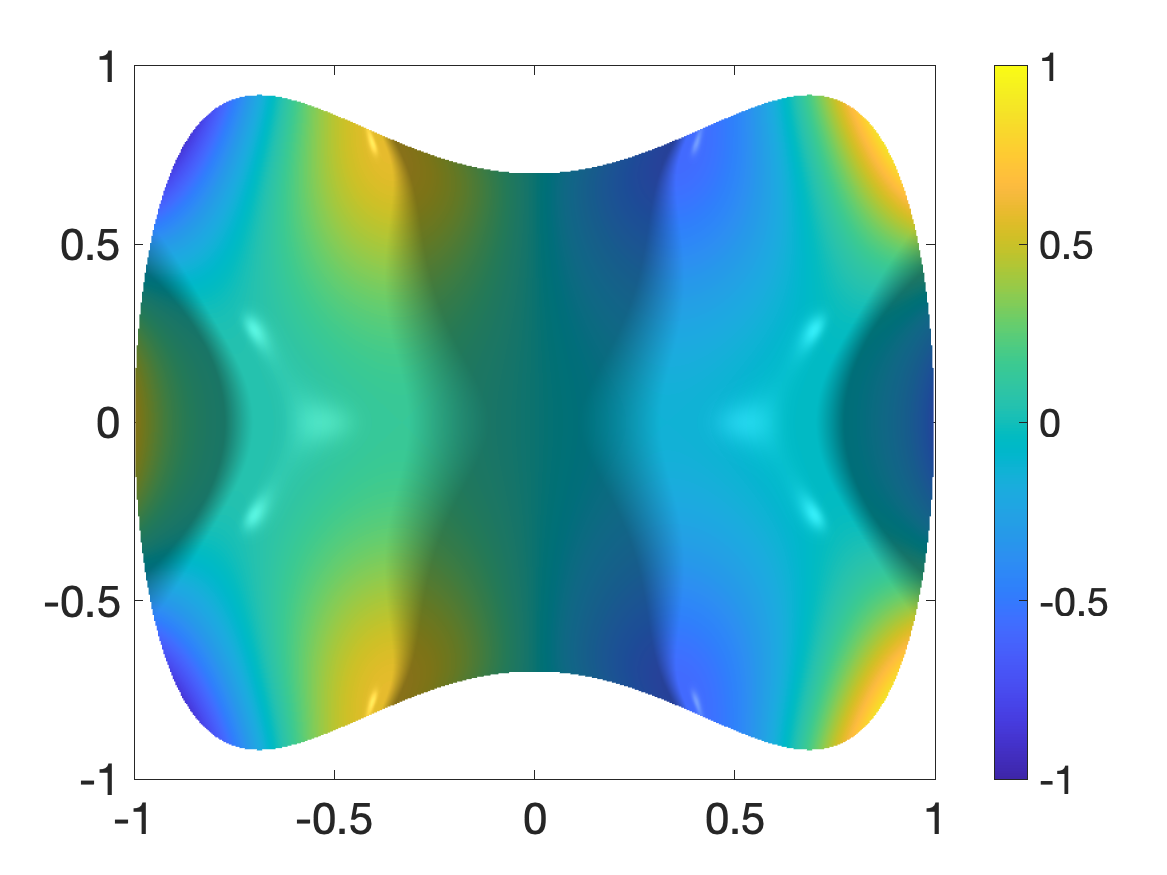}&$\Lambda_7$&
		\includegraphics[width=0.17\textwidth]{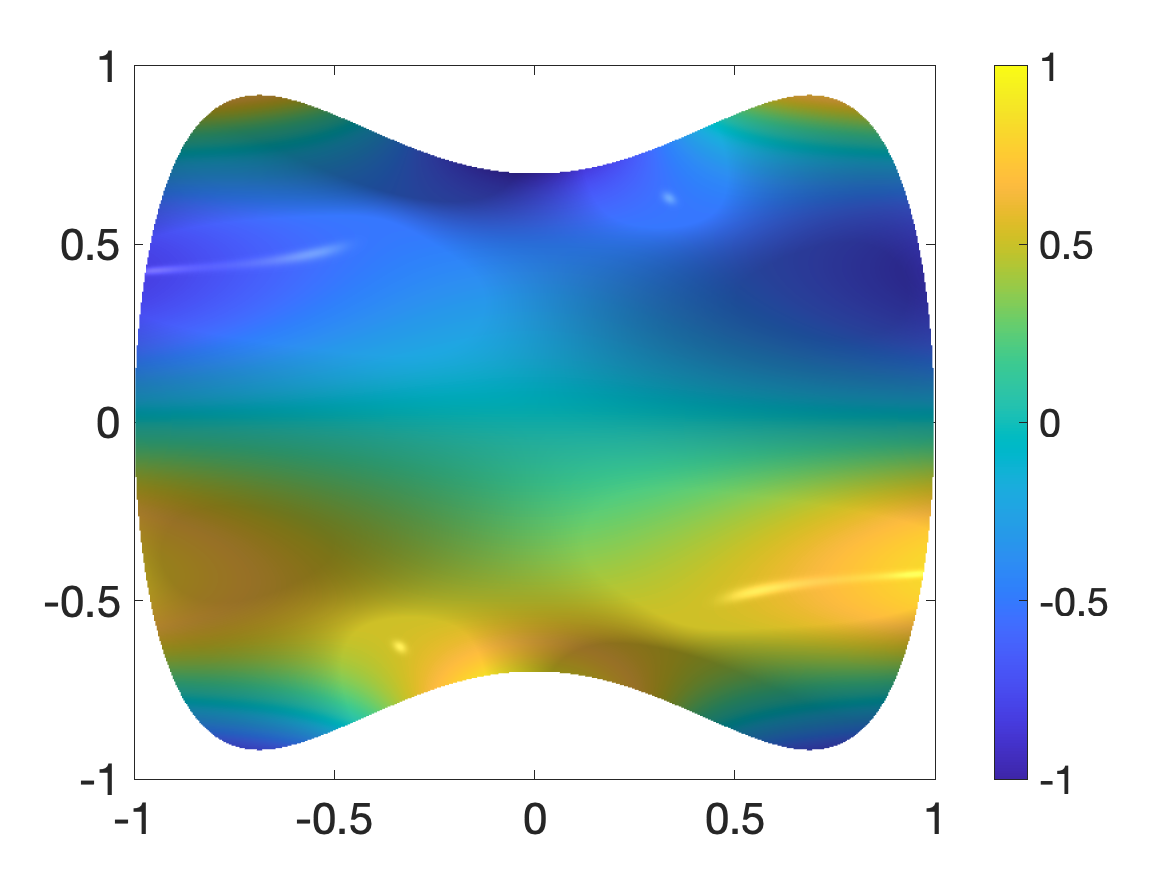}&
		\includegraphics[width=0.17\textwidth]{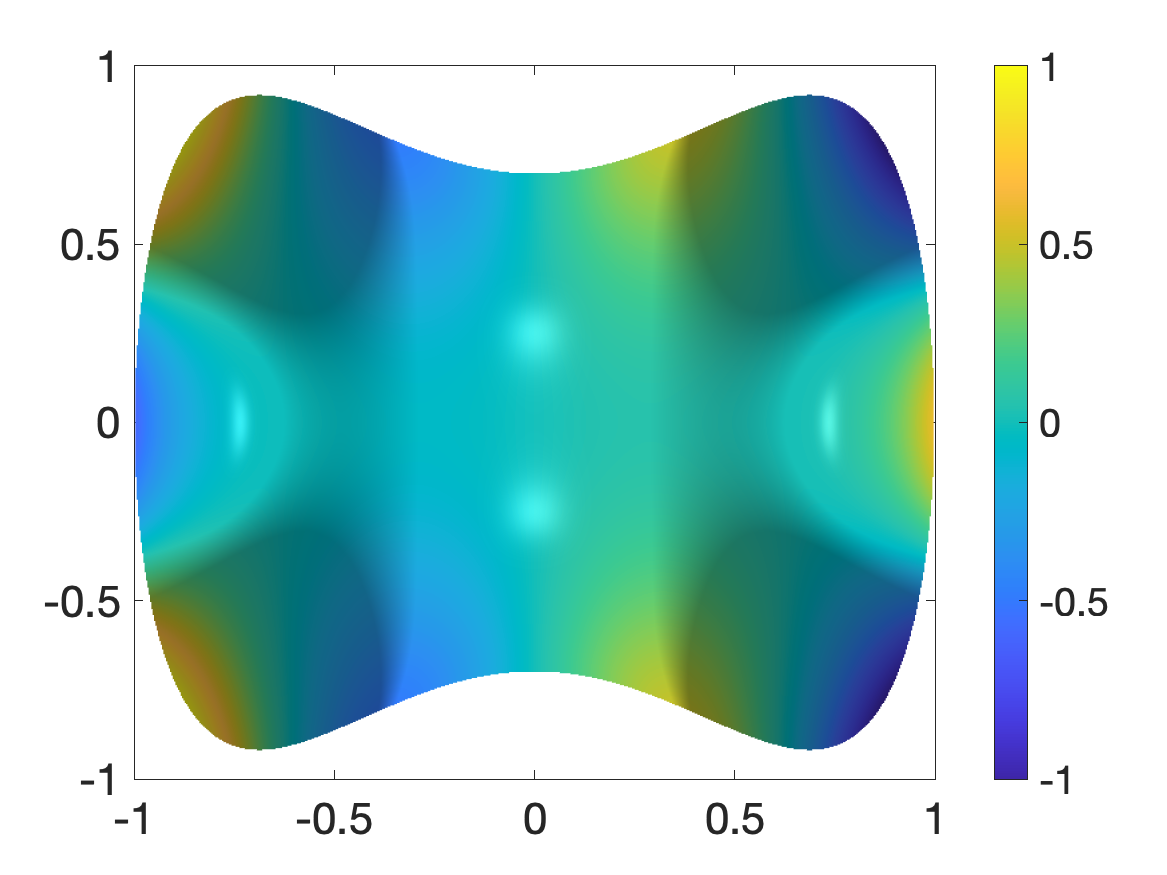}\\
		
		\includegraphics[width=0.17\textwidth]{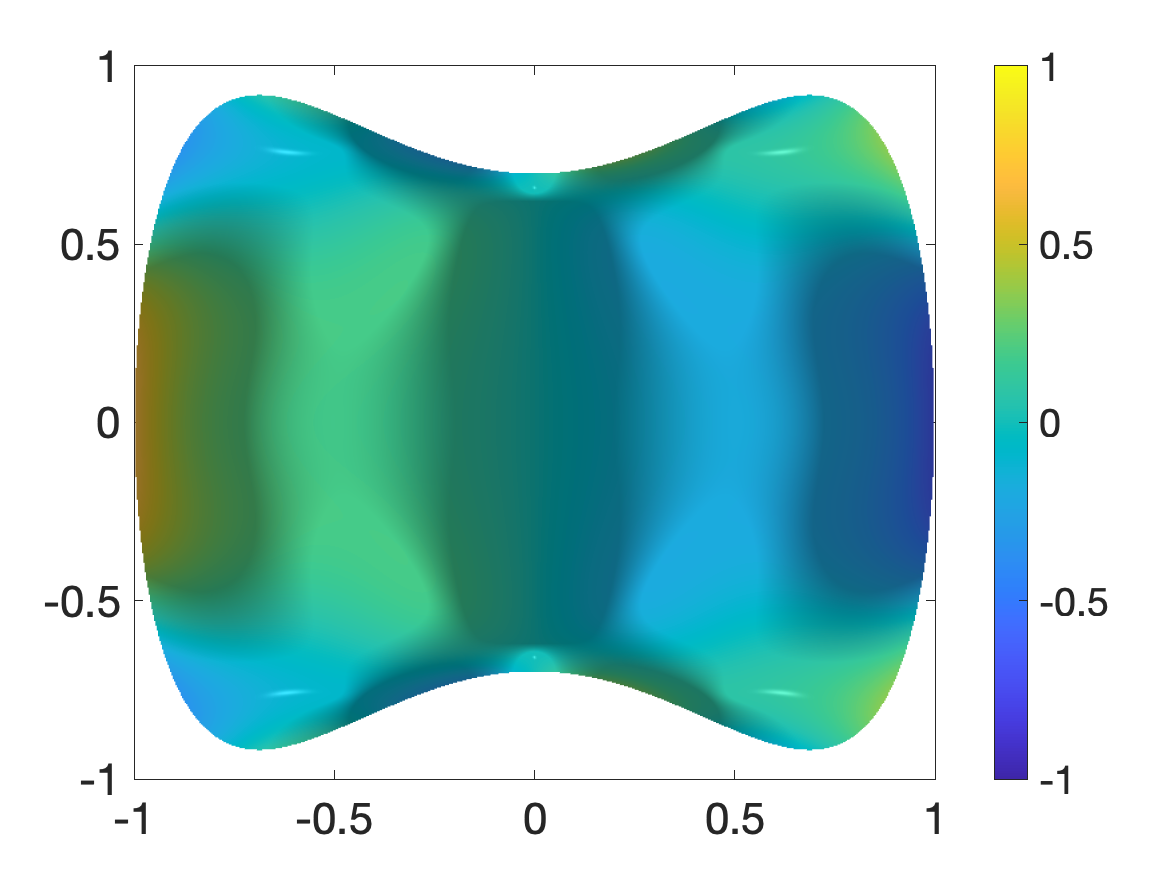}&
		\includegraphics[width=0.17\textwidth]{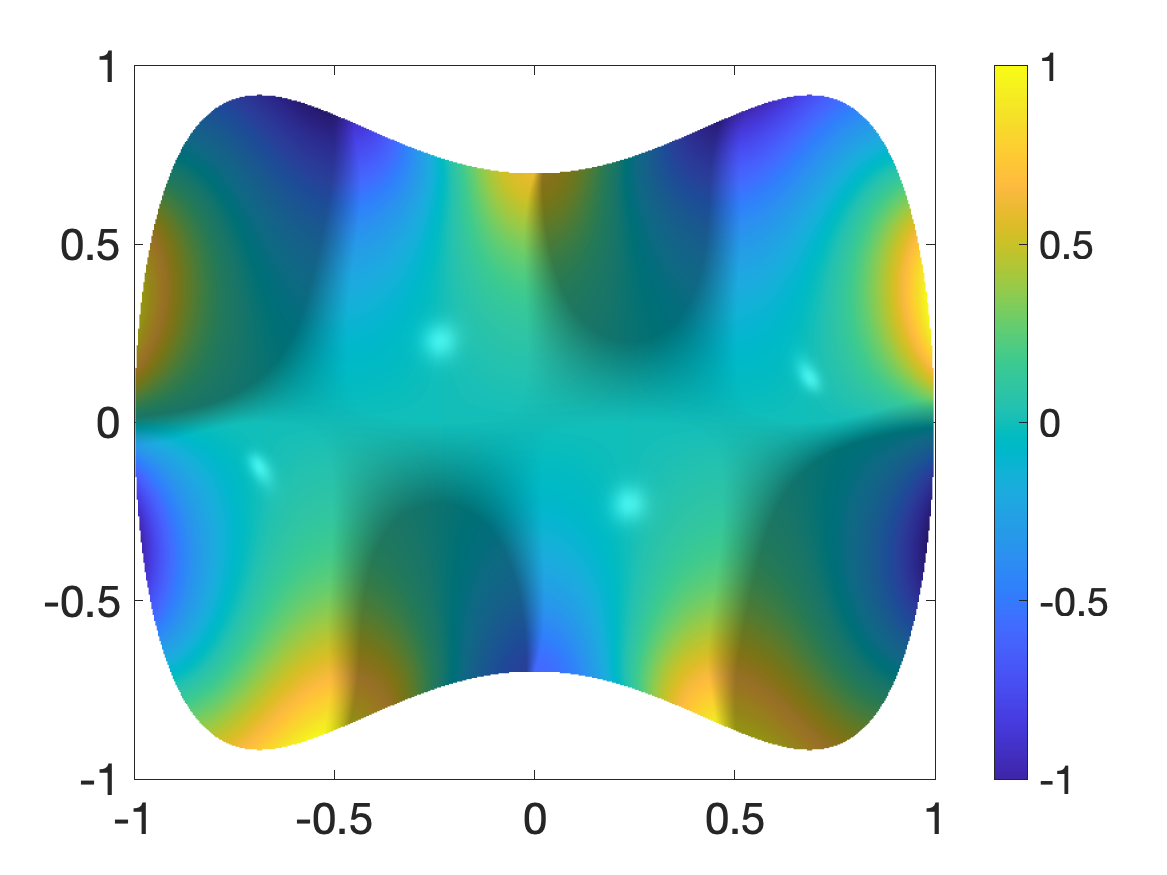}&$\Lambda_{20}$ &
		\includegraphics[width=0.17\textwidth]{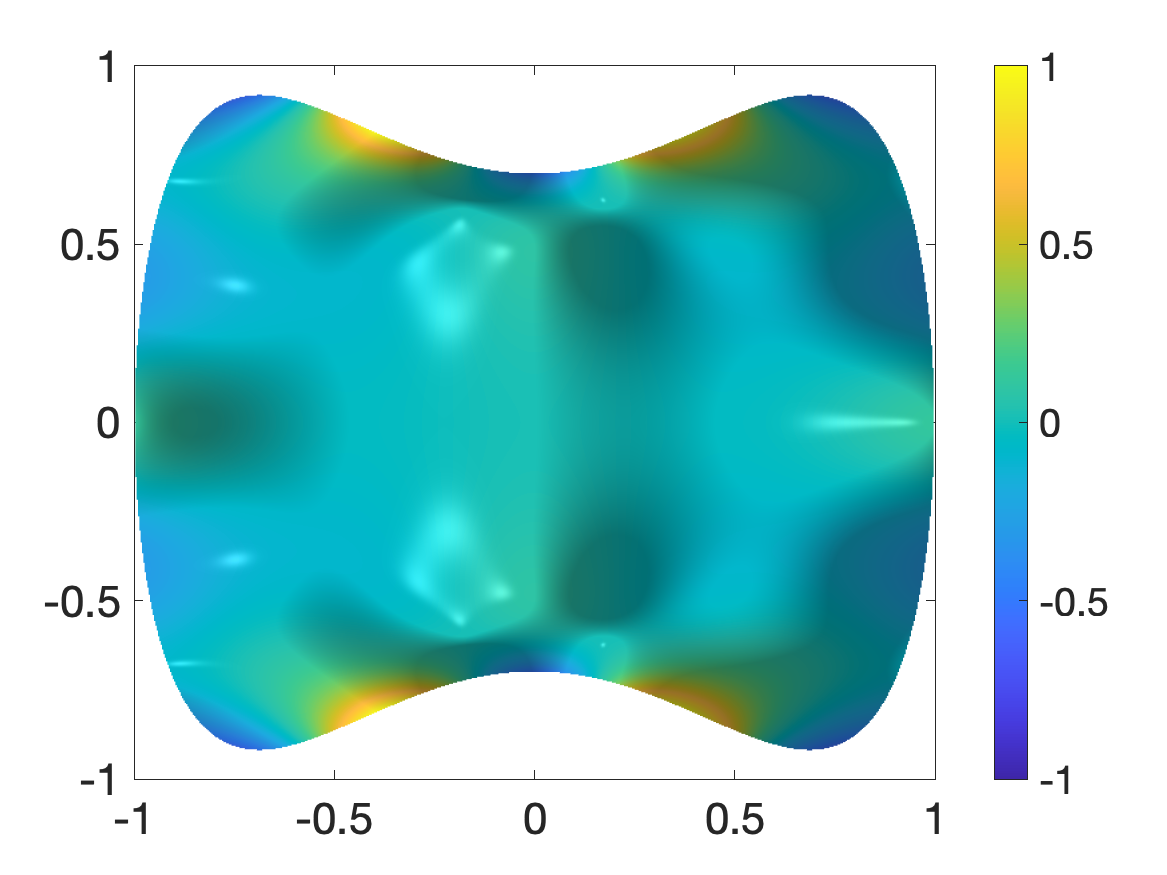}&
		\includegraphics[width=0.17\textwidth]{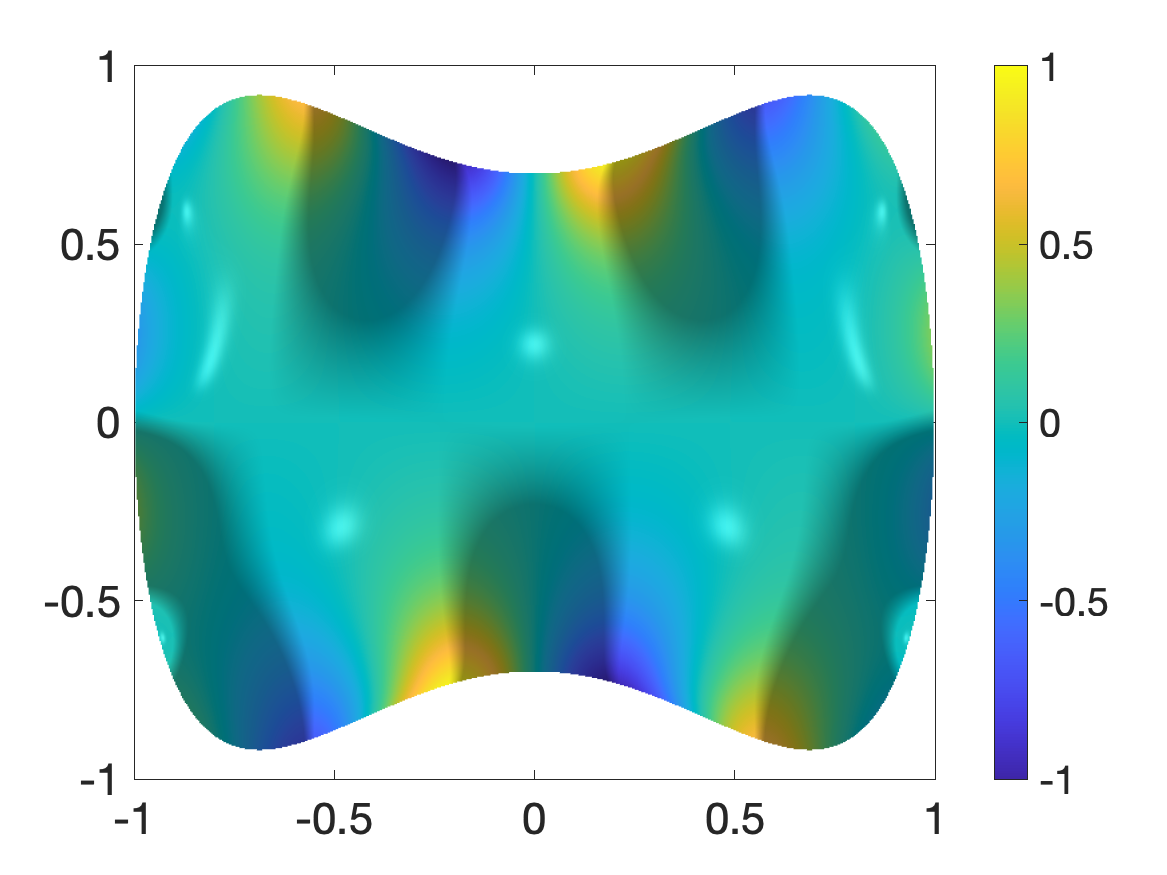}\\

		\includegraphics[width=0.17\textwidth]{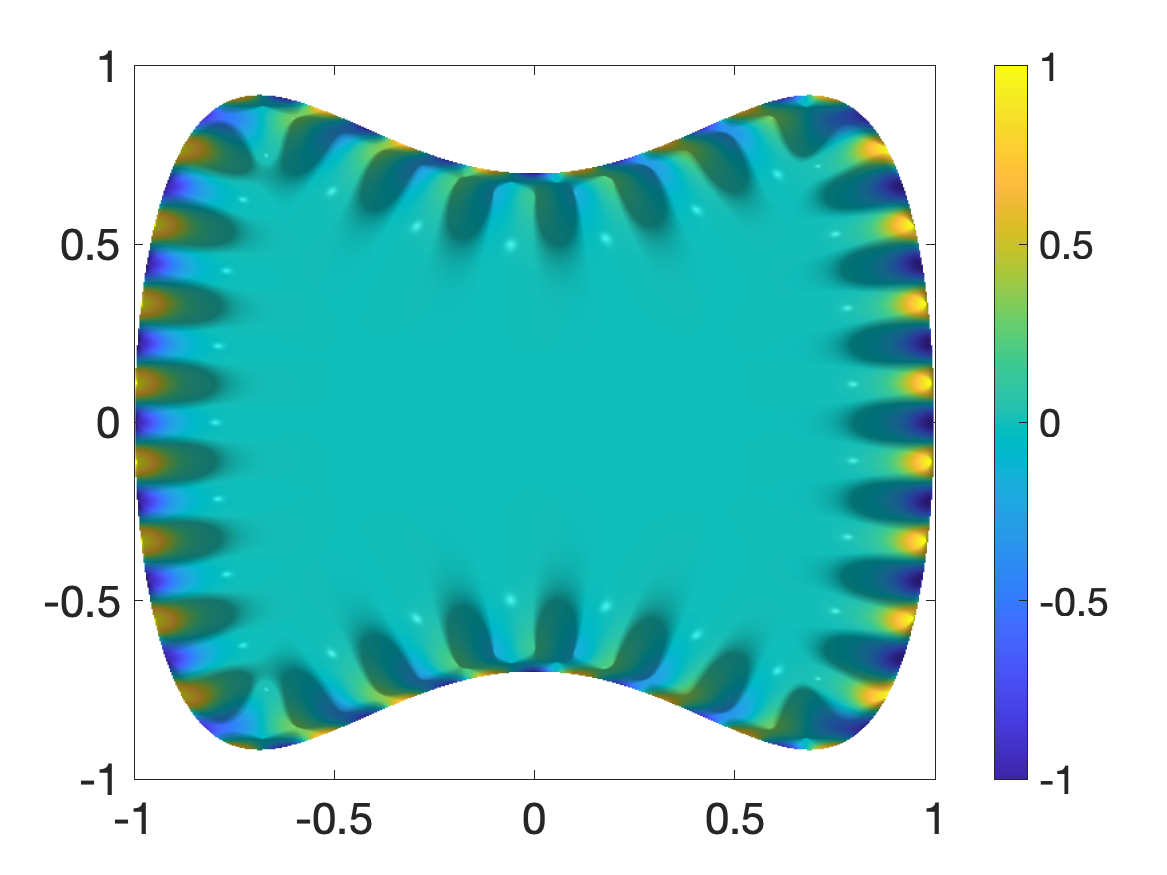}&
		\includegraphics[width=0.17\textwidth]{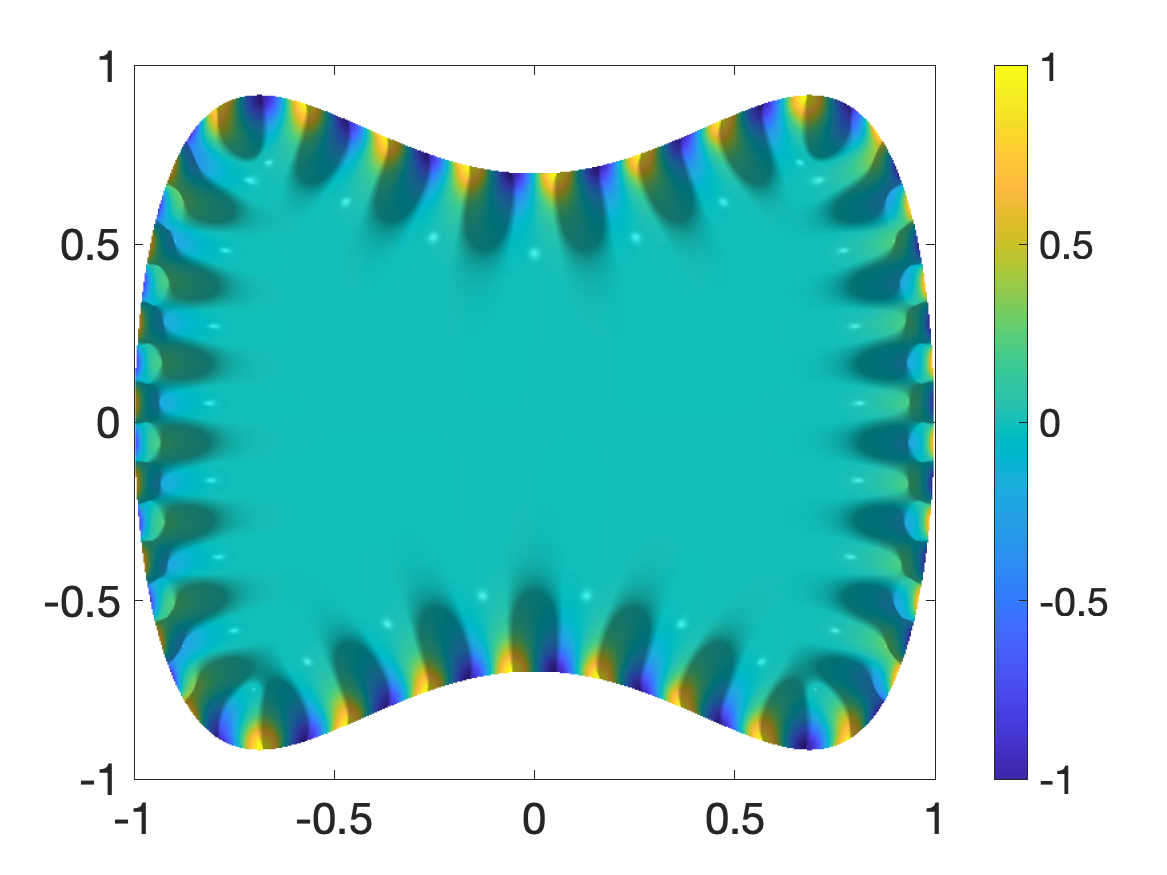}&$\Lambda_{100}$ &
		\includegraphics[width=0.17\textwidth]{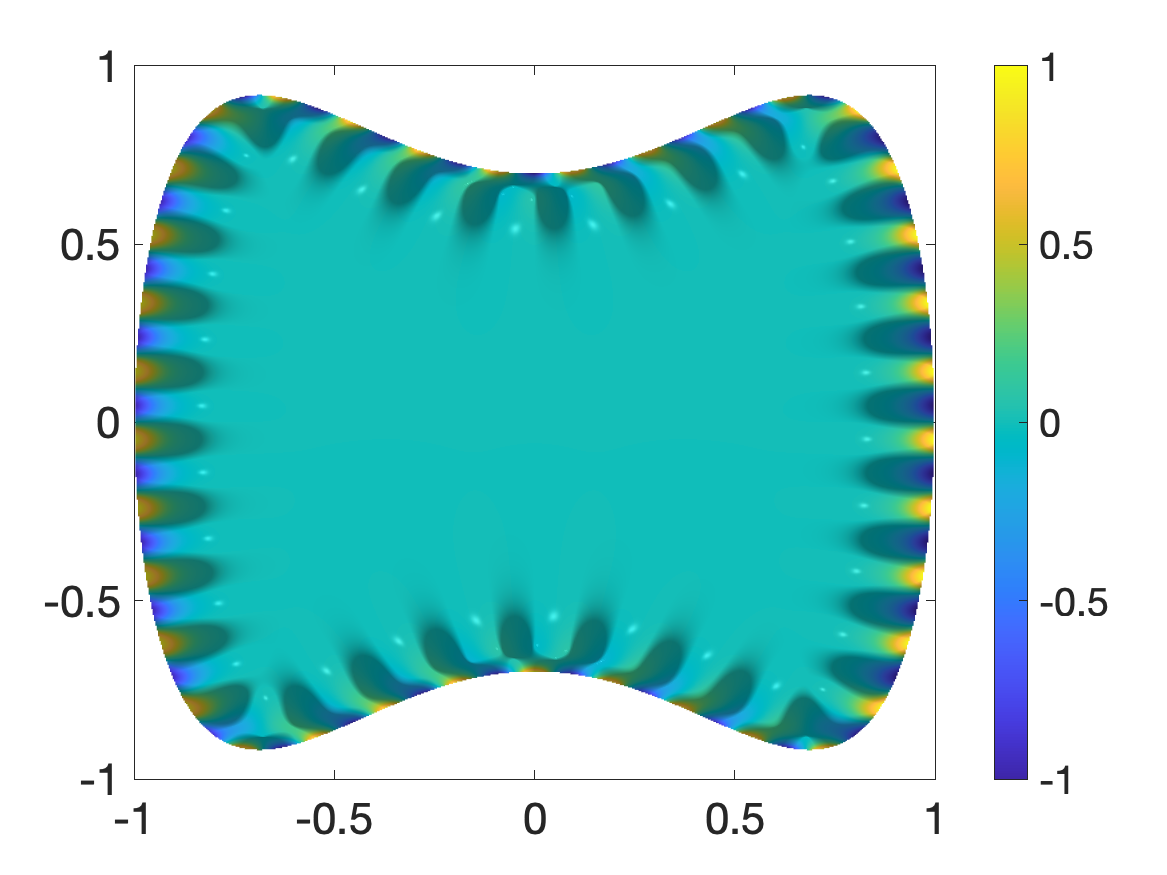}&
		\includegraphics[width=0.17\textwidth]{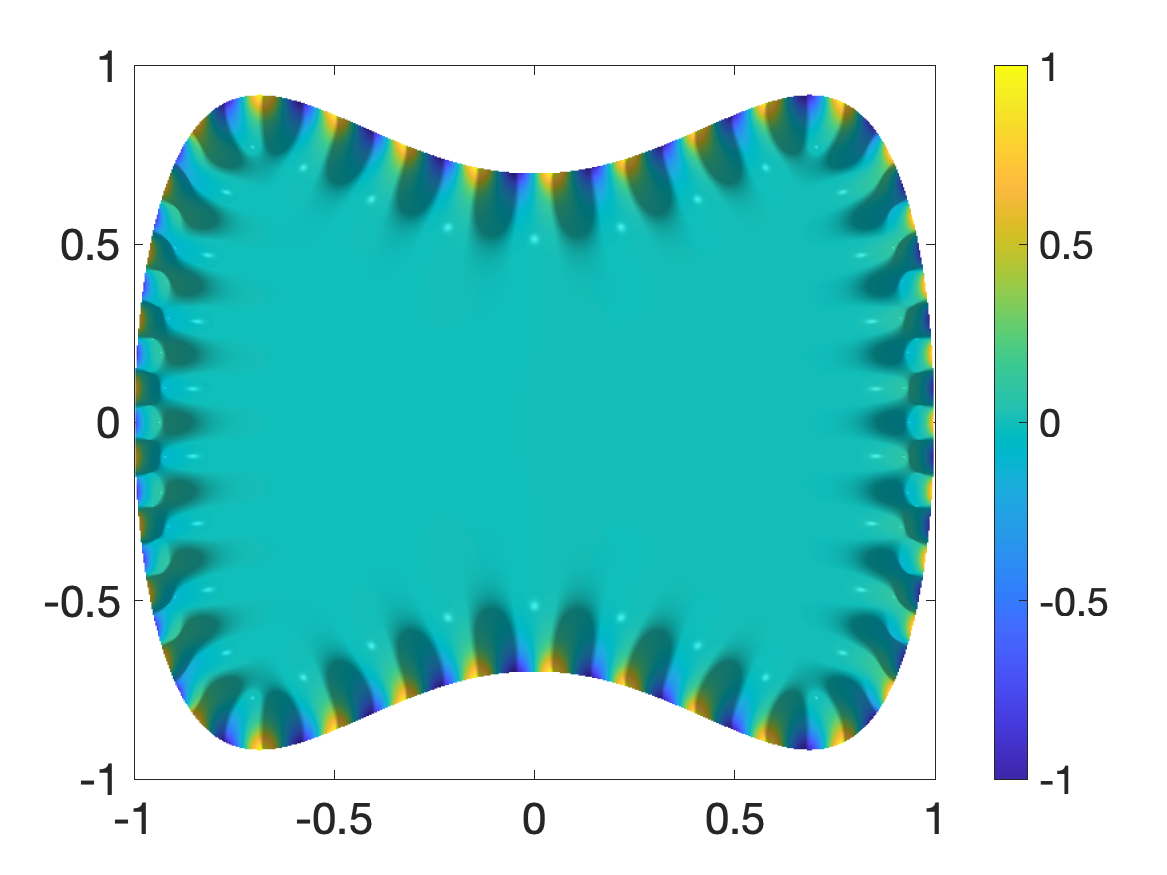}\\
		
		plot of $u_1$ &
		plot of $u_2$& &
		plot of $u_1$&
		plot of $u_2$\\
		$\lambda=1,\ \mu=0.5$& $\lambda=1,\ \mu=0.5$& & $\lambda=1,\ \mu=3$& $\lambda=1,\ \mu=3$
	\end{tabular}
	
	\caption{Plots of the first and second components for the eigenfunctions associated to the eigenvalues $\Lambda_i,\ i=1,2,7,20,100$ of $\Omega_1$ with $(\lambda,\mu)=(1,0.5)$ (left-hand side plots) and $(\lambda,\mu)=(1,3)$ (right-hand side plots).}
	\label{fig:eig_funct} 
\end{figure*}

Next, we illustrate the Moler-Payne type result applied to some eigenvalues and corresponding eigenfunctions of $\Omega_1$. In Figure~\ref{fig:moler_payne} we plot $\|f_\varepsilon\|_{\bo L^2(\partial \Omega)}=\|Ae(\bo u_\varepsilon)\bo n -\Lambda_\varepsilon \bo u_\varepsilon\|_{\bo L^2(\partial \Omega)}$, for $\Lambda_i,\ i=1,20,100$. In each case we took a $L^2$ normalized approximated eigenfunction and by~\eqref{bound_error} we can get an upper bound for the error of the approximation of the eigenvalue simply by measuring $\|\bo f_\varepsilon\|_{\bo L^2(\partial \Omega)}$.  An eigenvalue and eigenfunction computation takes just a few seconds even for the $100$-th eigenvalue. The plot of the errors suggests that the numerical computations are highly accurate, underlining the interest of using MFS when dealing with smooth domains. 

\begin{figure}[ht]
	\centering 
	\includegraphics[width=0.7\textwidth]{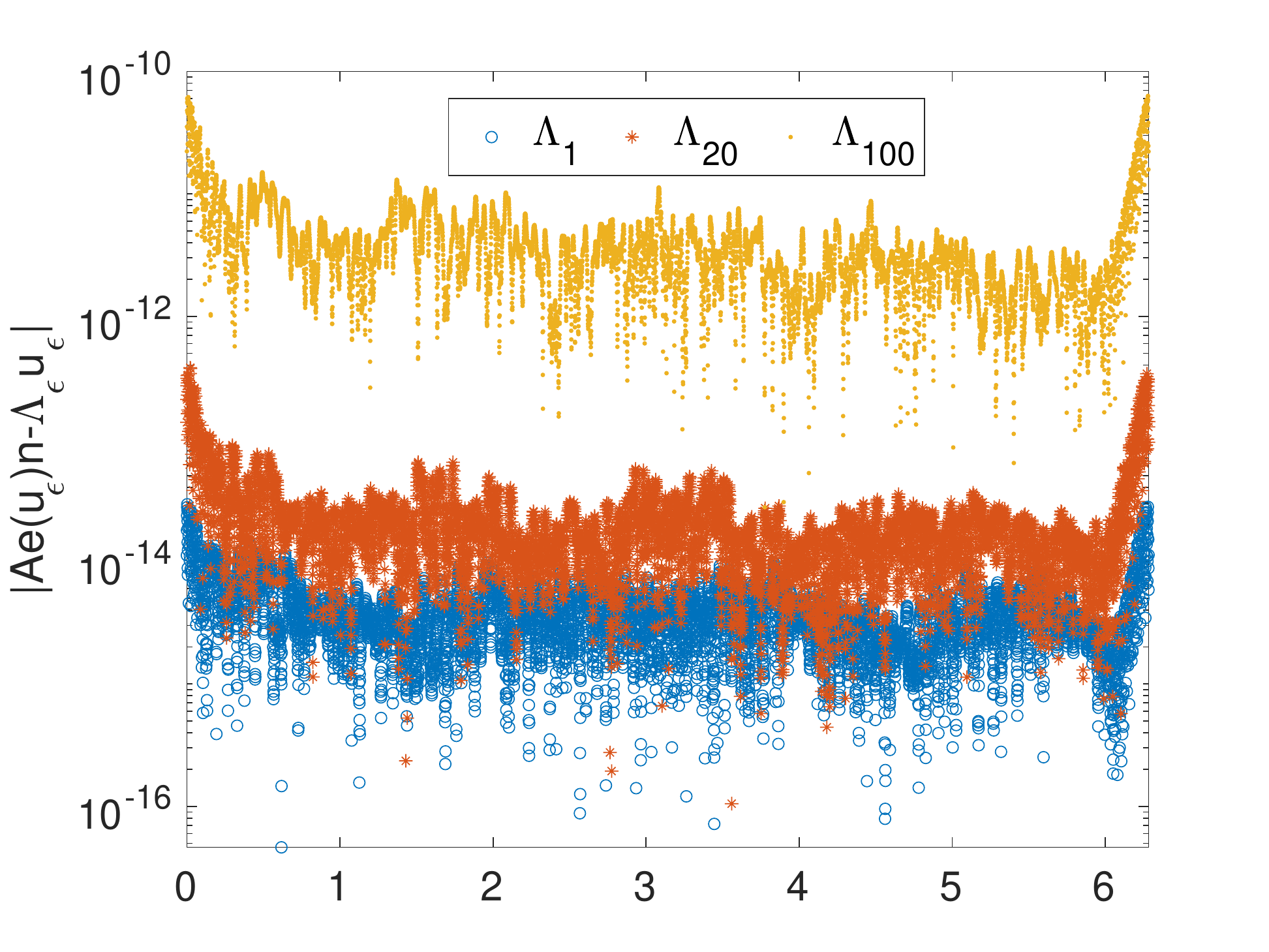}
	\caption{Plots of $|f_\varepsilon|=|Ae(\bo u_\varepsilon)\bo n -\Lambda_\varepsilon \bo u_\varepsilon|$, for $\Lambda_i,\ i=1,20,100$. In each case we took a $\bo L^2(\partial \Omega)$ normalized approximate eigenfunction.} \label{fig:moler_payne}
\end{figure}

Next, we show some numerical results for the solution of shape optimization problems \eqref{shoptprob} and \eqref{shoptprobconv}. Figure~\ref{fig:opt18} shows the plots of the optimal eigenvalues $\Lambda_n^\ast$ using only a volume constraint (marked with $\txtb{\bullet}$) and $\Lambda_n^\#$ using volume and convexity constraint (marked with $\txtr{\circ}$), for $n=1,2,...,8$, together with the representation of the optimal domain. In each case we plot also the eigenvalue obtained for the disk with unit area. Figure~\ref{fig:opt910} shows similar results for $\Lambda_i,\ i=9,10$. The optimization process for $\Lambda_1$ takes around $5$ minutes on a portable laptop. For higher eigenvalues, due to higher multiplicity, computations are more time consuming.

We summarize some observations below:
\begin{itemize}[topsep=0pt]
	\item The disk maximizes $\Lambda_1(\Omega)$ at fixed volume. This result was proved theoretically when $\lambda>\mu$ in Theorem \ref{thm:optimality-disk}. 
	\item The maximizers for $\Lambda_2(\Omega)$ are convex. For $\mu = \lambda$ the disk seems to be optimal. When $\mu<\lambda$ the minimizer is close to the disk, however, the optimal values are slightly larger than those for the disk and the multiplicity clusters do not coincide with those known the disk.
	\item The maximizes for $\Lambda_3(\Omega)$ are convex.
	\item The maximizers for $\Lambda_k(\Omega)$ are convex when $\lambda \leq \mu$ and $k \in \{4,5,6\}$.
\end{itemize}

It is a common observation when studying optimizers of spectral functionals, that the optimal eigenvalue tends to be multiple. For the scalar Steklov problem this fact was observed in \cite{osting-steklov} and \cite{bogosel-bucur-giacomini}. In our case, the numerical results also suggest that often the maximal eigenvalue is multiple, however, the behavior is more complex, as it depends on the Lam\'e parameters $\lambda$ and $\mu$. Consider the following examples:
\begin{itemize}[topsep=0pt]
	\item maximization of $\Lambda_1$: the numerical maximizer is always the disk. Therefore the optimal eigenvalue is double when $\lambda\neq \mu$ and quadruple when $\lambda=\mu$. 
	\item maximization lf $\Lambda_2$: the optimal eigenvalue is double except when $\lambda=\mu$ when the eigenvalue is quadruple.
	\item maximization lf $\Lambda_3$: the optimal eigenvalue is double except when $\mu<2\lambda$ and triple when $\mu \geq 2\lambda$.
\end{itemize}

\begin{figure}[ht]
	\centering 
	\includegraphics[width=0.49\textwidth]{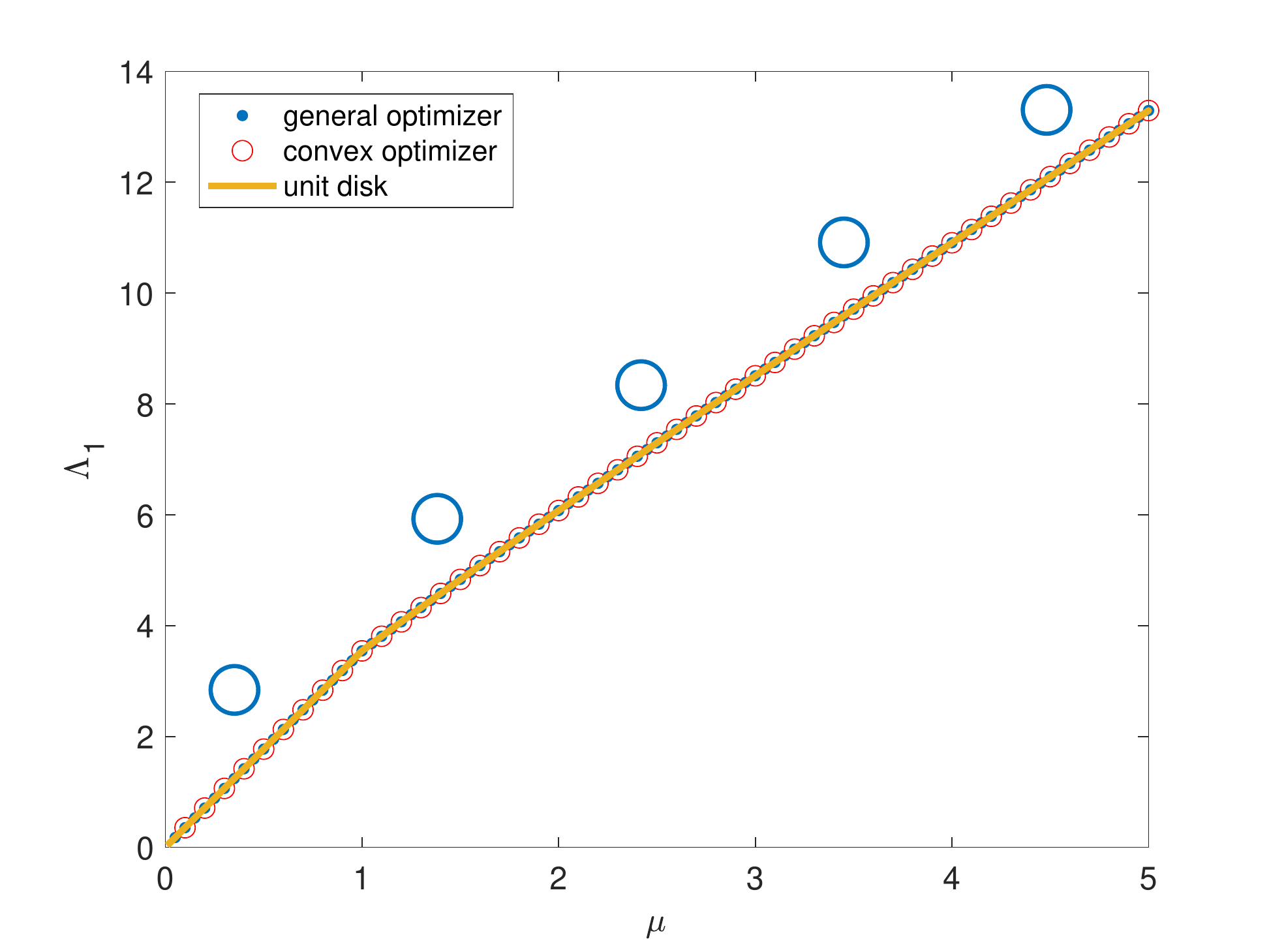}
	\includegraphics[width=0.49\textwidth]{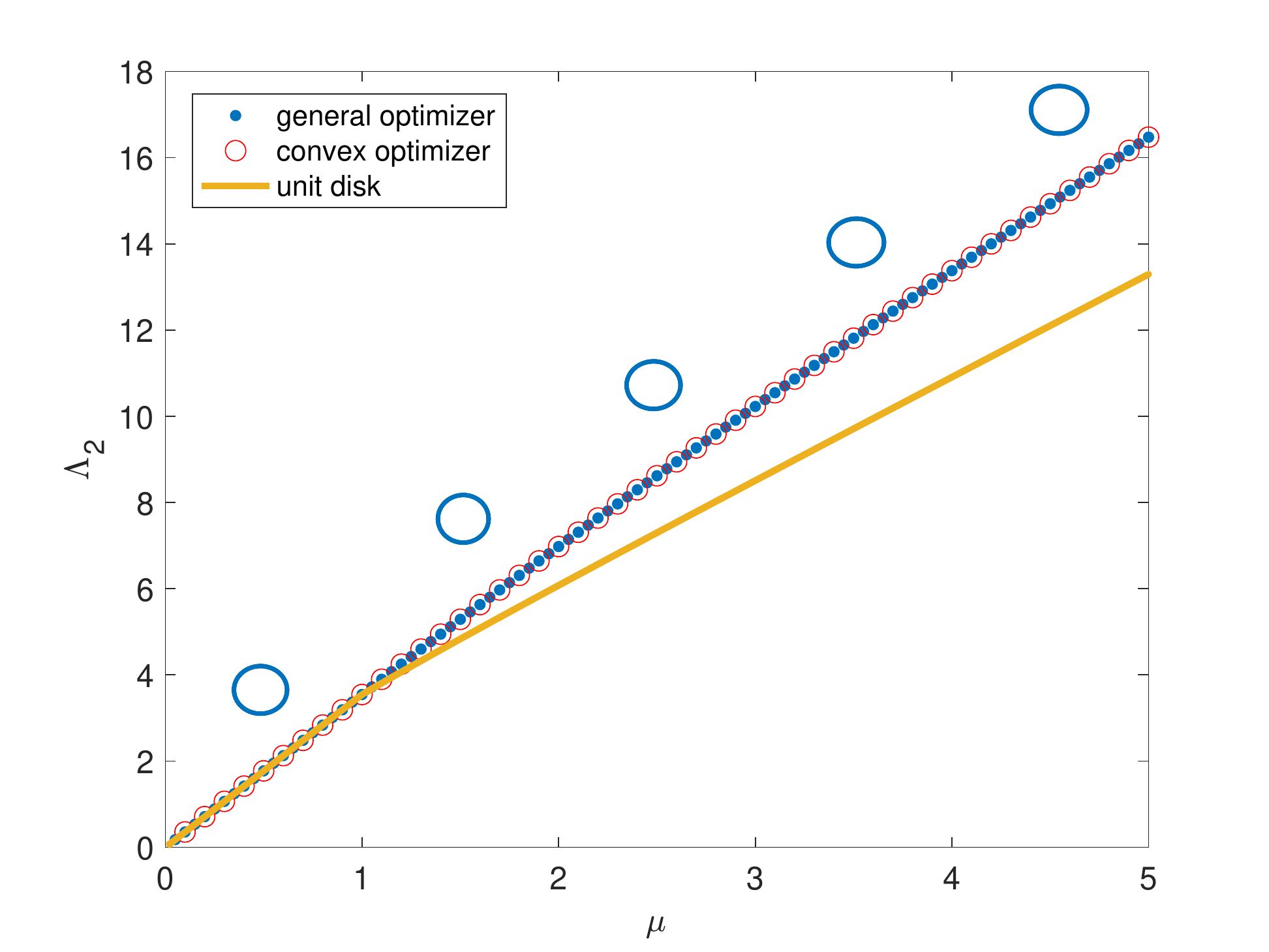}
	\includegraphics[width=0.49\textwidth]{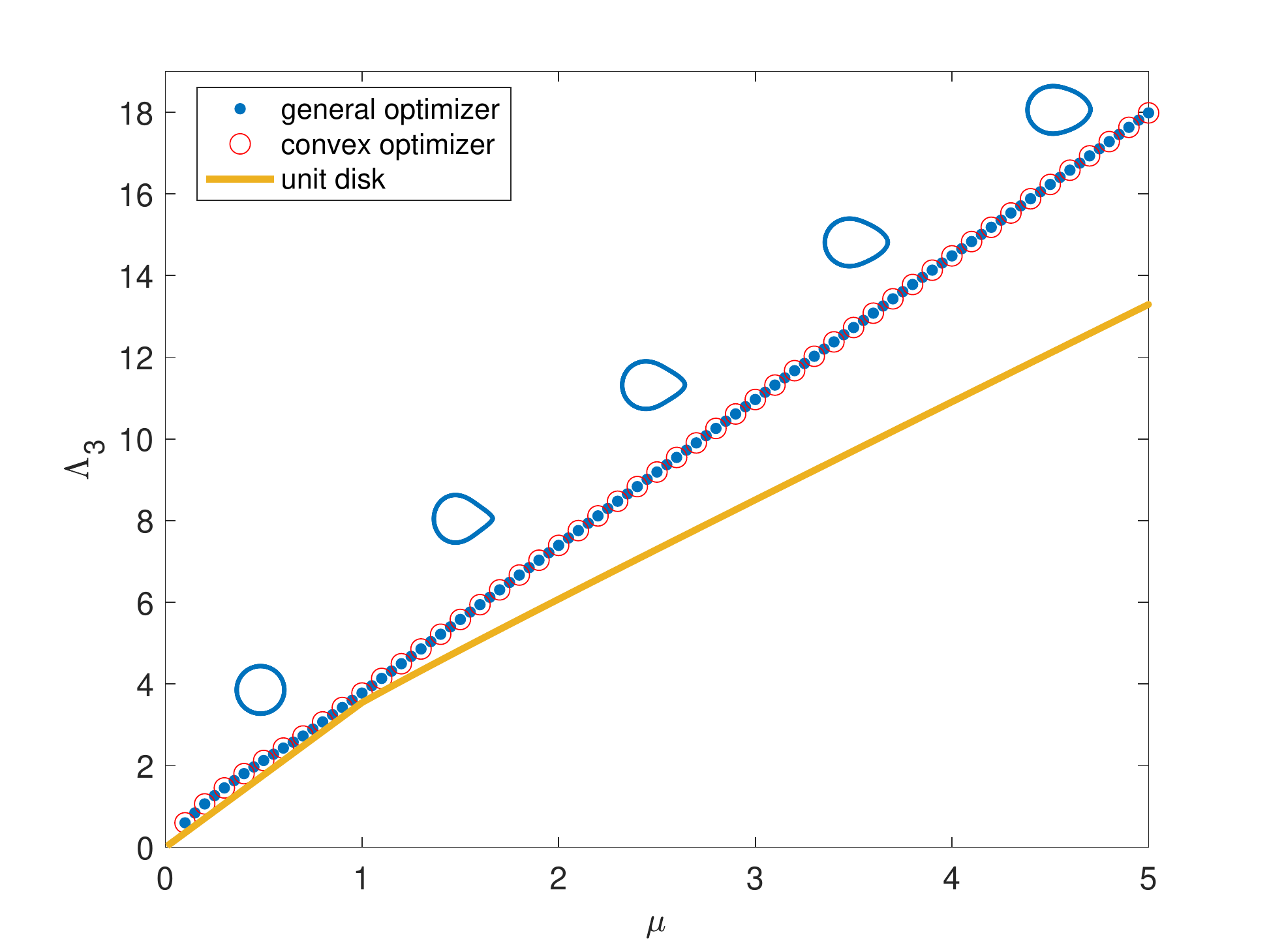}
	\includegraphics[width=0.49\textwidth]{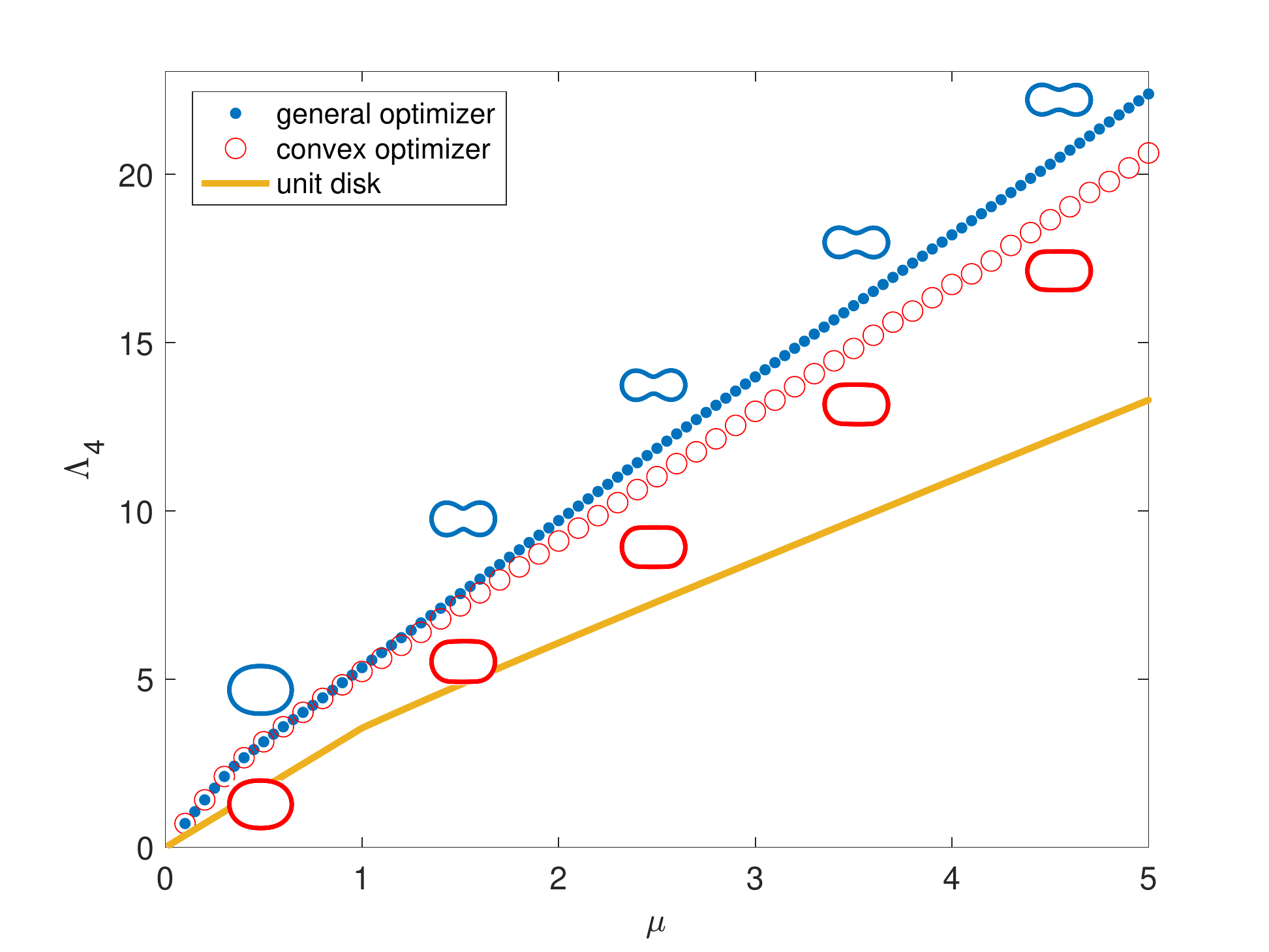}
	\includegraphics[width=0.49\textwidth]{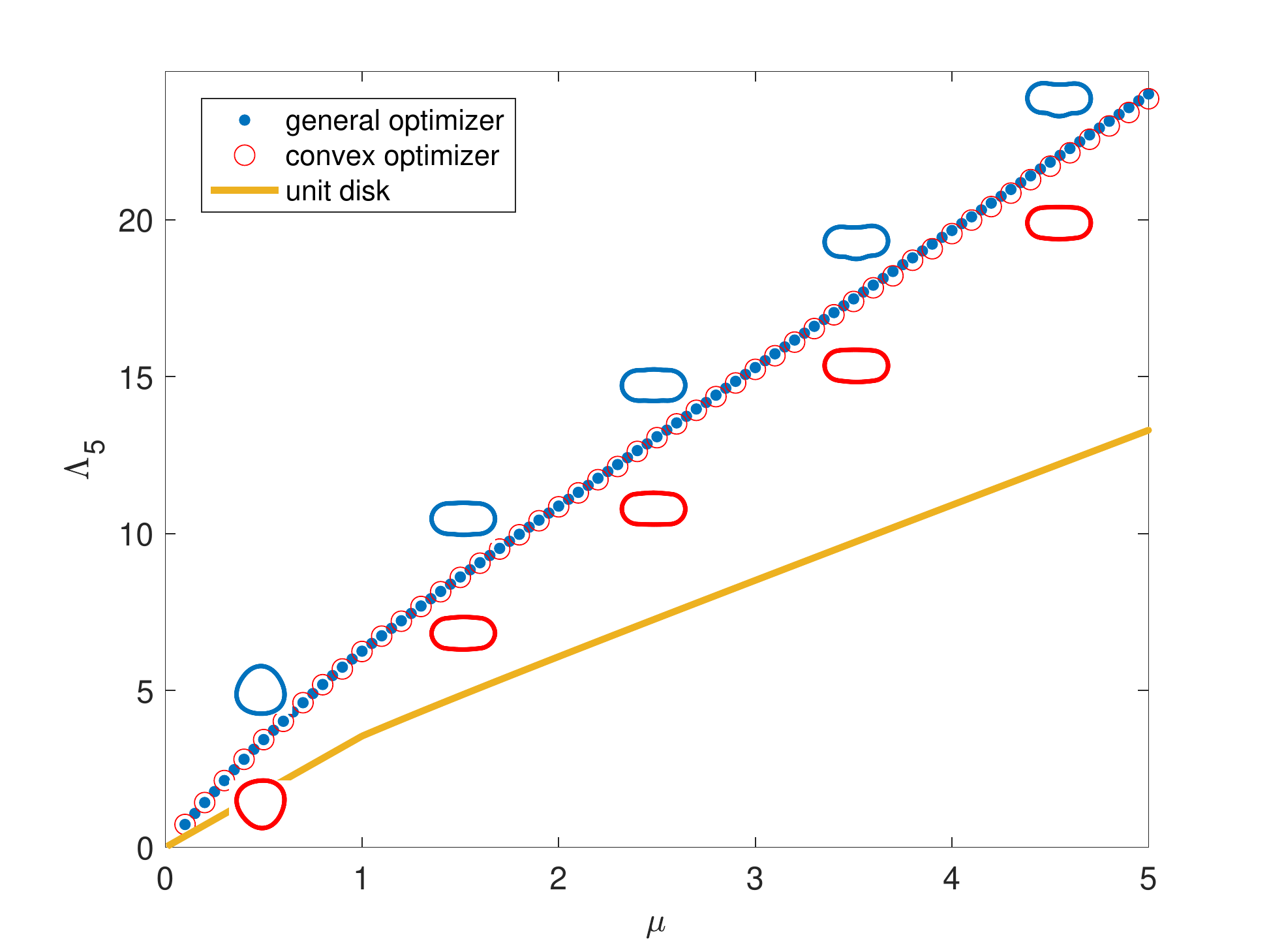}
	\includegraphics[width=0.49\textwidth]{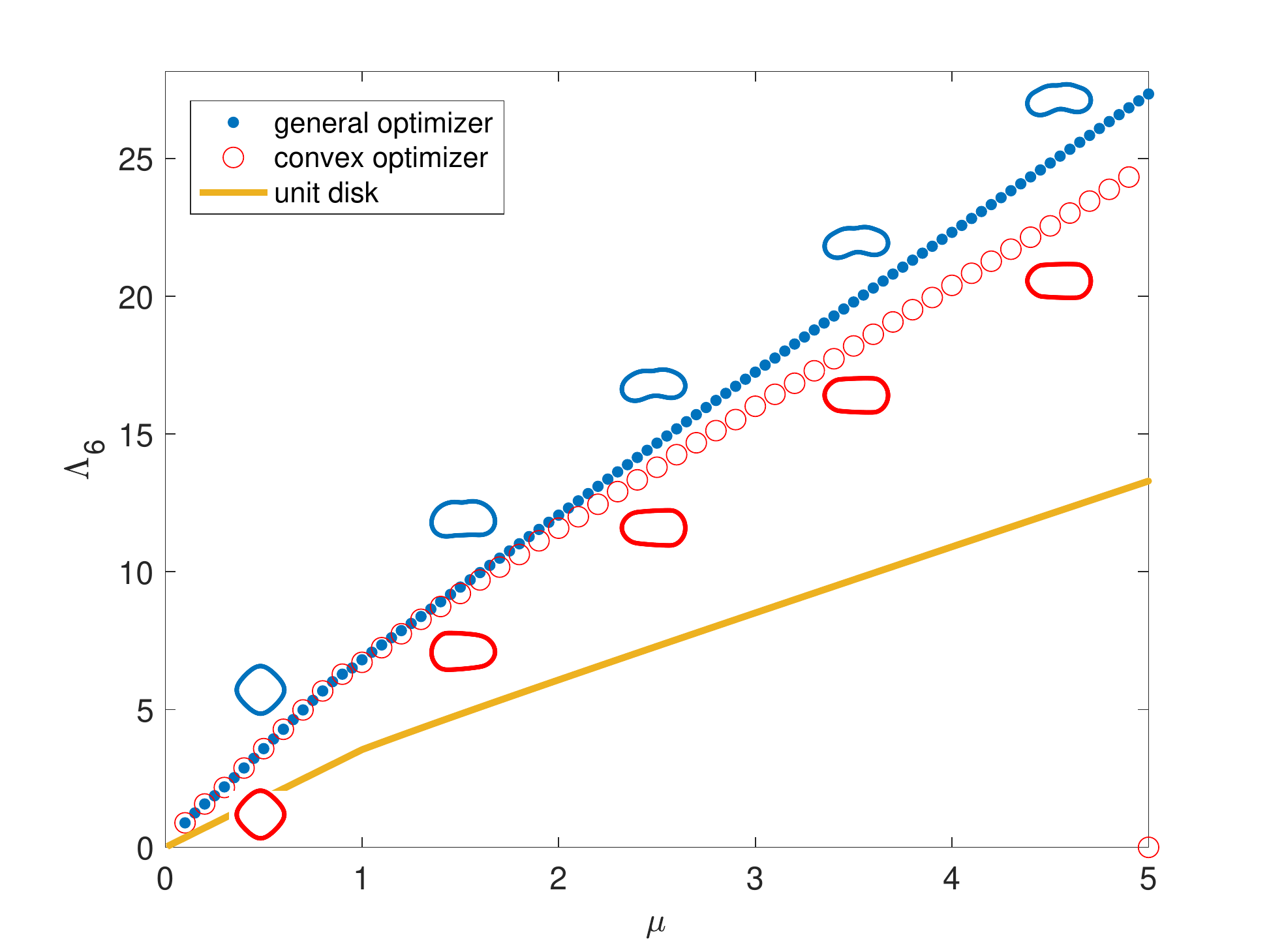}
	\includegraphics[width=0.49\textwidth]{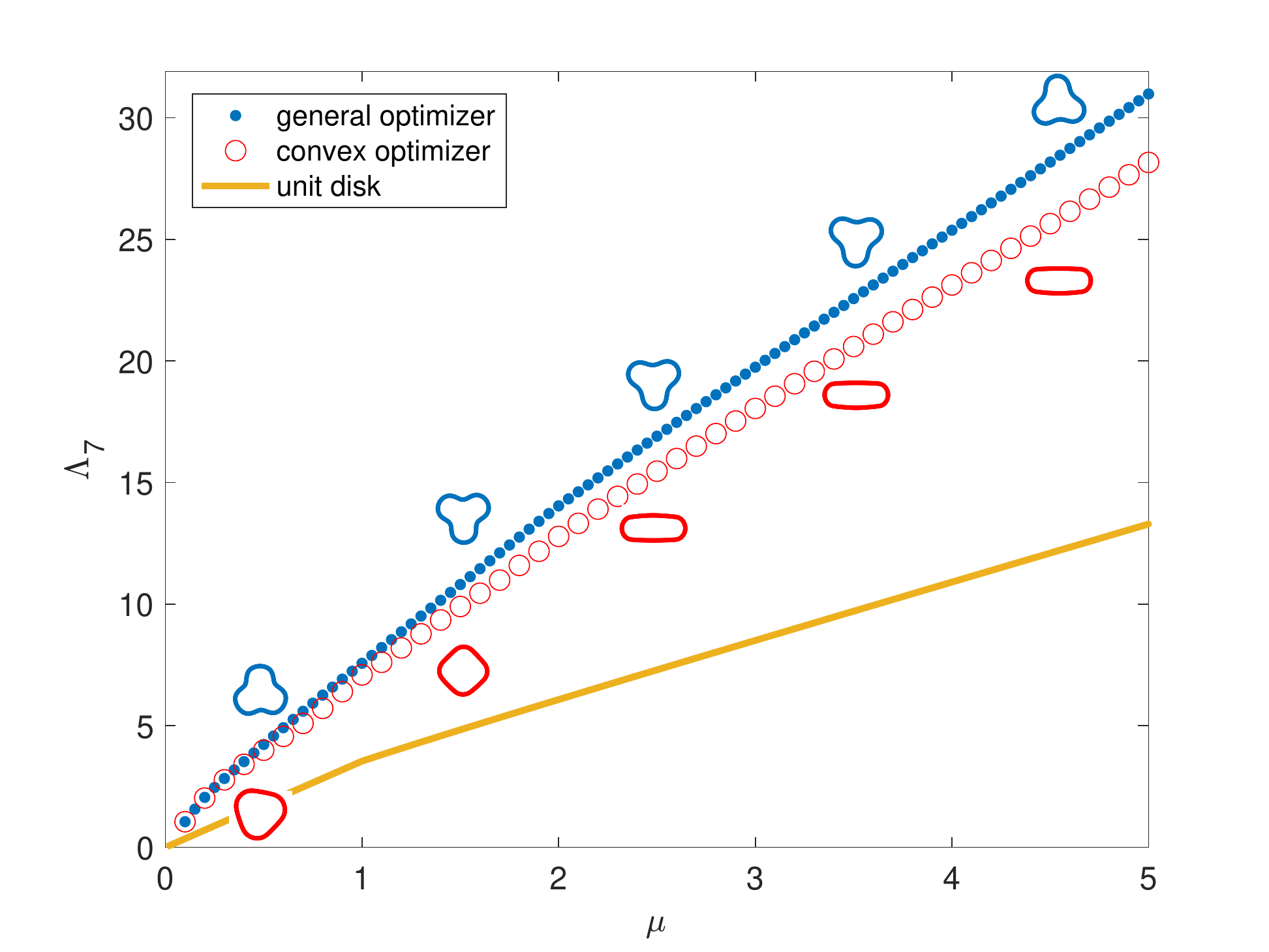}
	\includegraphics[width=0.49\textwidth]{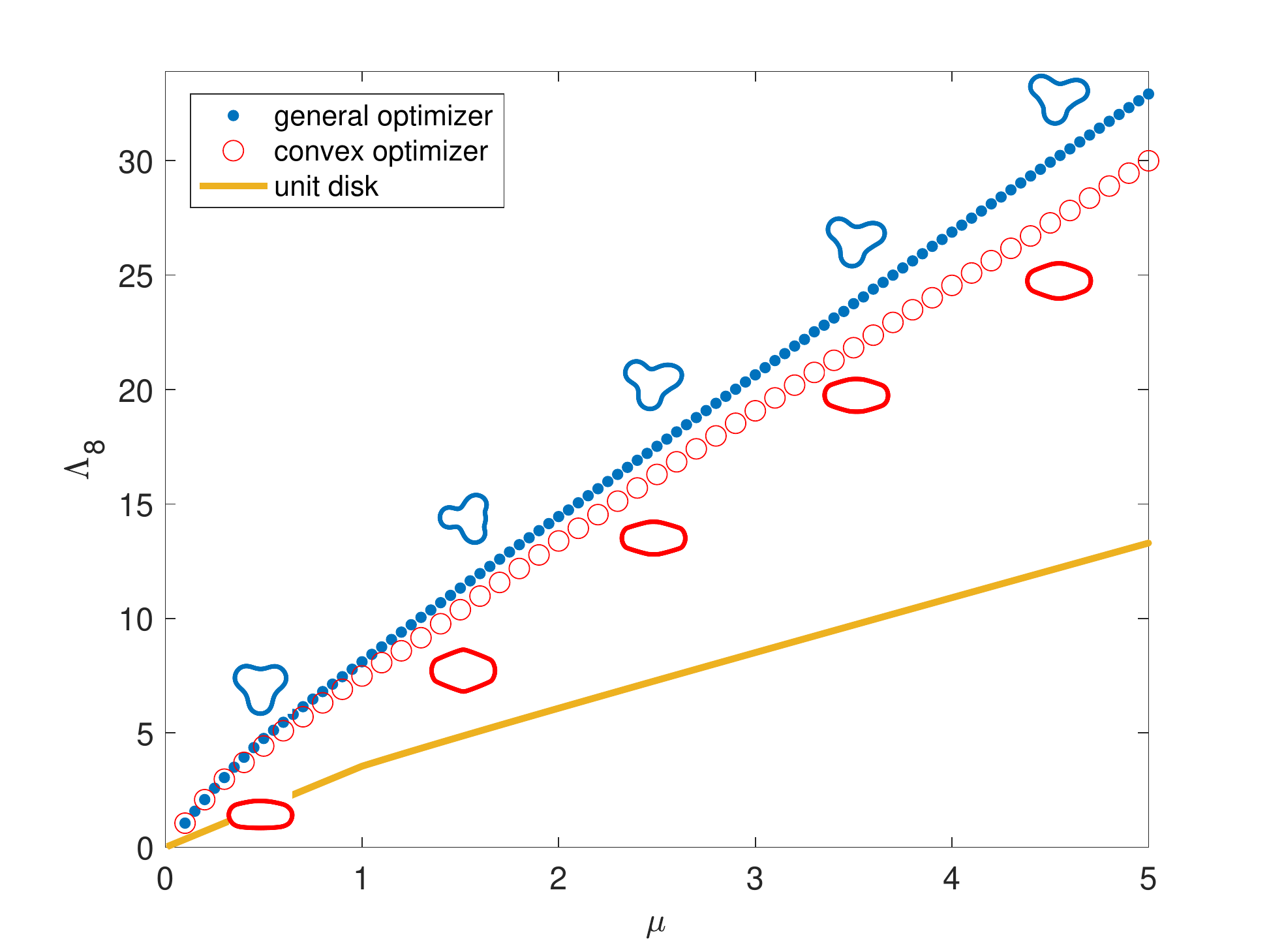}
	\caption{Plots of the optimal eigenvalues $\Lambda_n^\ast$ (marked with $\txtb{\bullet}$) and $\Lambda_n^\star$ (marked with $\txtr{\circ}$), for $n=1,2,...,8$, together with the representation of the optimal domain. In each case we plot also the eigenvalue obtained for the disk with unit area.} \label{fig:opt18}
\end{figure}

\begin{figure}[ht]
	\centering 
	\includegraphics[width=0.49\textwidth]{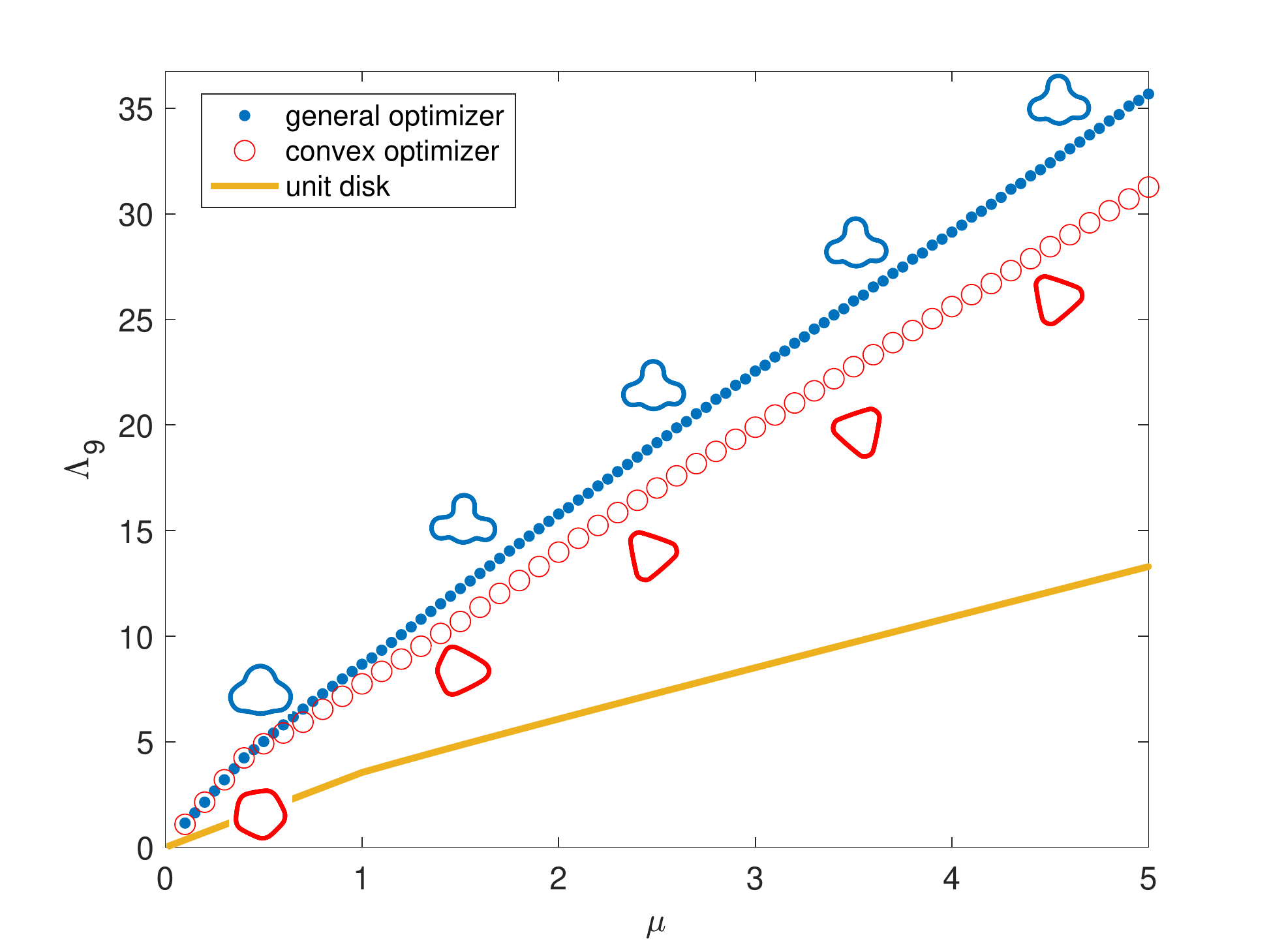}
	\includegraphics[width=0.49\textwidth]{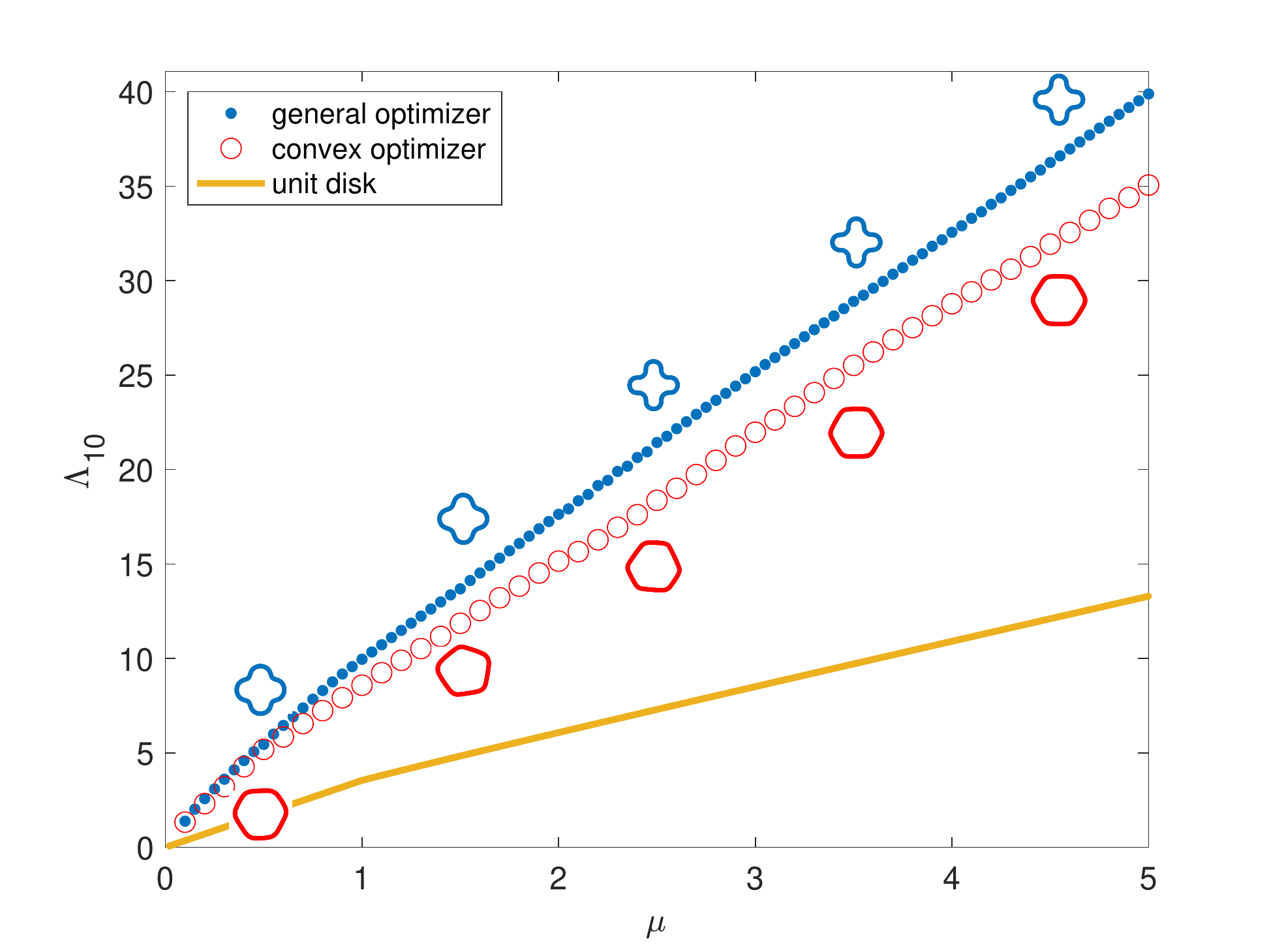}
	\caption{Plots of the optimal eigenvalues $\Lambda_n^\ast$ (marked with $\txtb{\bullet}$) and $\Lambda_n^\star$ (marked with $\txtr{\circ}$), for $n=9,10$, together with the representation of the optimal domain. In each case we plot also the eigenvalue obtained for the disk with unit area.} \label{fig:opt910}
\end{figure}

\section{Conclusions}

In this paper we studied the behavior of the Steklov-Lam\'e eigenvalues on variable domains. The eigenstructure of the disk was determined in Theorem \ref{thm:eigdisk}. This allowed us to partially extend the results of Weinstock \cite{weinstock} and Brock \cite{brock} to the Steklov-Lam\'e eigenvalues in Theorem \ref{thm:optimality-disk}: the disk maximizes the first non-zero eigenvalue when $\lambda>\mu$ under area and perimeter constraints. Numerical observations suggest that this also holds when $\lambda \leq \mu$. 

Upper bounds related to the scalar Steklov eigenvalues generalize to the Steklov-Lam\'e case, as shown in Proposition \ref{prop:upper-bounds}. Theorem \ref{thm:weight-continuity} shows that the eigenvalues are upper-semicontinuous among $\varepsilon$-cone domains converging in the complementary Hausdorff distance. As a direct consequence, there exist maximizers of the Steklov-Lam\'e eigenvalues among convex sets with unit volume. 

A numerical method based on fundamental solutions was proposed to approximate these eigenvalues numerically. This allowed us to study numerically domains maximizing the Steklov-Lam\'e eigenvalues. This work shows that many of the results related to the scalar Steklov eigenproblem \eqref{eq:steklov-eigs} extend to the Steklov-Lam\'e eigenvalues. 

\bmhead{Acknowledgments} The second author has been supported by the ANR SHAPO (ANR-18-CE40-0013) grant.

\section*{Declarations}

The authors have no competing interests to declare that are relevant to the content of this article. The data that support the findings of this paper are available from the corresponding author upon request.

\bibliography{./StekLam.bib}


\begin{thebibliography}{38}
\ifx \bisbn   \undefined \def \bisbn  #1{ISBN #1}\fi
\ifx \binits  \undefined \def \binits#1{#1}\fi
\ifx \bauthor  \undefined \def \bauthor#1{#1}\fi
\ifx \batitle  \undefined \def \batitle#1{#1}\fi
\ifx \bjtitle  \undefined \def \bjtitle#1{#1}\fi
\ifx \bvolume  \undefined \def \bvolume#1{\textbf{#1}}\fi
\ifx \byear  \undefined \def \byear#1{#1}\fi
\ifx \bissue  \undefined \def \bissue#1{#1}\fi
\ifx \bfpage  \undefined \def \bfpage#1{#1}\fi
\ifx \blpage  \undefined \def \blpage #1{#1}\fi
\ifx \burl  \undefined \def \burl#1{\textsf{#1}}\fi
\ifx \doiurl  \undefined \def \doiurl#1{\url{https://doi.org/#1}}\fi
\ifx \betal  \undefined \def \betal{\textit{et al.}}\fi
\ifx \binstitute  \undefined \def \binstitute#1{#1}\fi
\ifx \binstitutionaled  \undefined \def \binstitutionaled#1{#1}\fi
\ifx \bctitle  \undefined \def \bctitle#1{#1}\fi
\ifx \beditor  \undefined \def \beditor#1{#1}\fi
\ifx \bpublisher  \undefined \def \bpublisher#1{#1}\fi
\ifx \bbtitle  \undefined \def \bbtitle#1{#1}\fi
\ifx \bedition  \undefined \def \bedition#1{#1}\fi
\ifx \bseriesno  \undefined \def \bseriesno#1{#1}\fi
\ifx \blocation  \undefined \def \blocation#1{#1}\fi
\ifx \bsertitle  \undefined \def \bsertitle#1{#1}\fi
\ifx \bsnm \undefined \def \bsnm#1{#1}\fi
\ifx \bsuffix \undefined \def \bsuffix#1{#1}\fi
\ifx \bparticle \undefined \def \bparticle#1{#1}\fi
\ifx \barticle \undefined \def \barticle#1{#1}\fi
\bibcommenthead
\ifx \bconfdate \undefined \def \bconfdate #1{#1}\fi
\ifx \botherref \undefined \def \botherref #1{#1}\fi
\ifx \url \undefined \def \url#1{\textsf{#1}}\fi
\ifx \bchapter \undefined \def \bchapter#1{#1}\fi
\ifx \bbook \undefined \def \bbook#1{#1}\fi
\ifx \bcomment \undefined \def \bcomment#1{#1}\fi
\ifx \oauthor \undefined \def \oauthor#1{#1}\fi
\ifx \citeauthoryear \undefined \def \citeauthoryear#1{#1}\fi
\ifx \endbibitem  \undefined \def \endbibitem {}\fi
\ifx \bconflocation  \undefined \def \bconflocation#1{#1}\fi
\ifx \arxivurl  \undefined \def \arxivurl#1{\textsf{#1}}\fi
\csname PreBibitemsHook\endcsname

\bibitem{weinstock}
\begin{barticle}
\bauthor{\bsnm{Weinstock}, \binits{R.}}:
\batitle{Inequalities for a classical eigenvalue problem}.
\bjtitle{J. Rational Mech. Anal.}
\bvolume{3},
\bfpage{745}--\blpage{753}
(\byear{1954})
\end{barticle}
\endbibitem

\bibitem{hersch-payne-schiffer}
\begin{barticle}
\bauthor{\bsnm{Hersch}, \binits{J.}},
\bauthor{\bsnm{Payne}, \binits{L.E.}},
\bauthor{\bsnm{Schiffer}, \binits{M.M.}}:
\batitle{Some inequalities for {S}tekloff eigenvalues}.
\bjtitle{Arch. Rational Mech. Anal.}
\bvolume{57},
\bfpage{99}--\blpage{114}
(\byear{1975})
\end{barticle}
\endbibitem

\bibitem{survey-girouard-polterowich}
\begin{barticle}
\bauthor{\bsnm{Girouard}, \binits{A.}},
\bauthor{\bsnm{Polterovich}, \binits{I.}}:
\batitle{Spectral geometry of the {S}teklov problem (survey article)}.
\bjtitle{J. Spectr. Theory}
\bvolume{7}(\bissue{2}),
\bfpage{321}--\blpage{359}
(\byear{2017})
\end{barticle}
\endbibitem

\bibitem{sloshing}
\begin{barticle}
\bauthor{\bsnm{Mayer}, \binits{H.C.}},
\bauthor{\bsnm{Krechetnikov}, \binits{R.}}:
\batitle{Walking with coffee: Why does it spill?}
\bjtitle{Phys. Rev. E}
\bvolume{85},
\bfpage{046117}
(\byear{2012})
\end{barticle}
\endbibitem

\bibitem{ammari-nigam}
\begin{barticle}
\bauthor{\bsnm{Ammari}, \binits{H.}},
\bauthor{\bsnm{Imeri}, \binits{K.}},
\bauthor{\bsnm{Nigam}, \binits{N.}}:
\batitle{Optimization of {S}teklov-{N}eumann eigenvalues}.
\bjtitle{J. Comput. Phys.}
\bvolume{406},
\bfpage{109211}--\blpage{15}
(\byear{2020})
\end{barticle}
\endbibitem

\bibitem{Bogosel2}
\begin{barticle}
\bauthor{\bsnm{Bogosel}, \binits{B.}}:
\batitle{The method of fundamental solutions applied to boundary eigenvalue
  problems}.
\bjtitle{J. Comput. Appl. Math.}
\bvolume{306},
\bfpage{265}--\blpage{285}
(\byear{2016})
\end{barticle}
\endbibitem

\bibitem{brock}
\begin{barticle}
\bauthor{\bsnm{Brock}, \binits{F.}}:
\batitle{An isoperimetric inequality for eigenvalues of the {S}tekloff
  problem}.
\bjtitle{ZAMM Z. Angew. Math. Mech.}
\bvolume{81}(\bissue{1}),
\bfpage{69}--\blpage{71}
(\byear{2001})
\end{barticle}
\endbibitem

\bibitem{Bogosel}
\begin{barticle}
\bauthor{\bsnm{Bogosel}, \binits{B.}}:
\batitle{The {S}teklov spectrum on moving domains}.
\bjtitle{Appl. Math. Optim.}
\bvolume{75}(\bissue{1}),
\bfpage{1}--\blpage{25}
(\byear{2017})
\end{barticle}
\endbibitem

\bibitem{bogosel-bucur-giacomini}
\begin{barticle}
\bauthor{\bsnm{Bogosel}, \binits{B.}},
\bauthor{\bsnm{Bucur}, \binits{D.}},
\bauthor{\bsnm{Giacomini}, \binits{A.}}:
\batitle{Optimal shapes maximizing the {S}teklov eigenvalues}.
\bjtitle{SIAM J. Math. Anal.}
\bvolume{49}(\bissue{2}),
\bfpage{1645}--\blpage{1680}
(\byear{2017})
\end{barticle}
\endbibitem

\bibitem{osting-steklov}
\begin{barticle}
\bauthor{\bsnm{Akhmetgaliyev}, \binits{E.}},
\bauthor{\bsnm{Kao}, \binits{C.-Y.}},
\bauthor{\bsnm{Osting}, \binits{B.}}:
\batitle{Computational methods for extremal {S}teklov problems}.
\bjtitle{SIAM J. Control Optim.}
\bvolume{55}(\bissue{2}),
\bfpage{1226}--\blpage{1240}
(\byear{2017})
\end{barticle}
\endbibitem

\bibitem{Sebastian}
\begin{barticle}
\bauthor{\bsnm{Dom\'{\i}nguez}, \binits{S.}}:
\batitle{Steklov eigenvalues for the {L}am\'{e} operator in linear elasticity}.
\bjtitle{J. Comput. Appl. Math.}
\bvolume{394},
\bfpage{113558}--\blpage{17}
(\byear{2021})
\end{barticle}
\endbibitem

\bibitem{Alves}
\begin{barticle}
\bauthor{\bsnm{Alves}, \binits{C.J.S.}}:
\batitle{On the choice of source points in the method of fundamental
  solutions}.
\bjtitle{Eng. Anal. Bound. Elem.}
\bvolume{33}(\bissue{12}),
\bfpage{1348}--\blpage{1361}
(\byear{2009})
\end{barticle}
\endbibitem

\bibitem{Alves-Antunes_2013}
\begin{barticle}
\bauthor{\bsnm{Alves}, \binits{C.J.S.}},
\bauthor{\bsnm{Antunes}, \binits{P.R.S.}}:
\batitle{The method of fundamental solutions applied to some inverse
  eigenproblems}.
\bjtitle{SIAM J. Sci. Comput.}
\bvolume{35}(\bissue{3}),
\bfpage{1689}--\blpage{1708}
(\byear{2013})
\end{barticle}
\endbibitem

\bibitem{Bogomolny}
\begin{barticle}
\bauthor{\bsnm{Bogomolny}, \binits{A.}}:
\batitle{Fundamental solutions method for elliptic boundary value problems}.
\bjtitle{SIAM J. Numer. Anal.}
\bvolume{22}(\bissue{4}),
\bfpage{644}--\blpage{669}
(\byear{1985})
\end{barticle}
\endbibitem

\bibitem{FK}
\begin{botherref}
\oauthor{\bsnm{Fairweather}, \binits{G.}},
\oauthor{\bsnm{Karageorghis}, \binits{A.}}:
The method of fundamental solutions for elliptic boundary value problems.
vol. 9,
pp. 69--95
(1998).
Numerical treatment of boundary integral equations
\end{botherref}
\endbibitem

\bibitem{moler-payne}
\begin{barticle}
\bauthor{\bsnm{Moler}, \binits{C.B.}},
\bauthor{\bsnm{Payne}, \binits{L.E.}}:
\batitle{Bounds for eigenvalues and eigenvectors of symmetric operators}.
\bjtitle{SIAM J. Numer. Anal.}
\bvolume{5},
\bfpage{64}--\blpage{70}
(\byear{1968})
\end{barticle}
\endbibitem

\bibitem{VMFG}
\begin{barticle}
\bauthor{\bsnm{Bonnaillie-No\"{e}l}, \binits{V.}},
\bauthor{\bsnm{Dambrine}, \binits{M.}},
\bauthor{\bsnm{H\'{e}rau}, \binits{F.}},
\bauthor{\bsnm{Vial}, \binits{G.}}:
\batitle{Artificial conditions for the linear elasticity equations}.
\bjtitle{Math. Comp.}
\bvolume{84}(\bissue{294}),
\bfpage{1599}--\blpage{1632}
(\byear{2015})
\end{barticle}
\endbibitem

\bibitem{Bucur-Nahon}
\begin{barticle}
\bauthor{\bsnm{Bucur}, \binits{D.}},
\bauthor{\bsnm{Nahon}, \binits{M.}}:
\batitle{Stability and instability issues of the {W}einstock inequality}.
\bjtitle{Trans. Amer. Math. Soc.}
\bvolume{374}(\bissue{3}),
\bfpage{2201}--\blpage{2223}
(\byear{2021})
\end{barticle}
\endbibitem

\bibitem{girouard-polterovich}
\begin{barticle}
\bauthor{\bsnm{Girouard}, \binits{A.}},
\bauthor{\bsnm{Polterovich}, \binits{I.}}:
\batitle{Shape optimization for low {N}eumann and {S}teklov eigenvalues}.
\bjtitle{Math. Methods Appl. Sci.}
\bvolume{33}(\bissue{4}),
\bfpage{501}--\blpage{516}
(\byear{2010})
\end{barticle}
\endbibitem

\bibitem{colbois-elsoufi-girouard}
\begin{barticle}
\bauthor{\bsnm{Colbois}, \binits{B.}},
\bauthor{\bsnm{El~Soufi}, \binits{A.}},
\bauthor{\bsnm{Girouard}, \binits{A.}}:
\batitle{Isoperimetric control of the {S}teklov spectrum}.
\bjtitle{J. Funct. Anal.}
\bvolume{261}(\bissue{5}),
\bfpage{1384}--\blpage{1399}
(\byear{2011})
\end{barticle}
\endbibitem

\bibitem{alsayed-bogosel-henrot-nacry}
\begin{barticle}
\bauthor{\bsnm{Al~Sayed}, \binits{A.}},
\bauthor{\bsnm{Bogosel}, \binits{B.}},
\bauthor{\bsnm{Henrot}, \binits{A.}},
\bauthor{\bsnm{Nacry}, \binits{F.}}:
\batitle{Maximization of the {S}teklov eigenvalues with a diameter constraint}.
\bjtitle{SIAM J. Math. Anal.}
\bvolume{53}(\bissue{1}),
\bfpage{710}--\blpage{729}
(\byear{2021})
\end{barticle}
\endbibitem

\bibitem{stability-steklov}
\begin{botherref}
\oauthor{\bsnm{Ferrero}, \binits{A.}},
\oauthor{\bsnm{Lamberti}, \binits{P.D.}}:
Spectral stability of the {S}teklov problem
(2021)
\end{botherref}
\endbibitem

\bibitem{henrot-pierre-english}
\begin{bbook}
\bauthor{\bsnm{Henrot}, \binits{A.}},
\bauthor{\bsnm{Pierre}, \binits{M.}}:
\bbtitle{Shape Variation and Optimization}.
\bpublisher{European Mathematical Society (EMS)},
\blocation{Z\"{u}rich}
(\byear{2018})
\end{bbook}
\endbibitem

\bibitem{korn-book}
\begin{bbook}
\bauthor{\bsnm{Ole\u{\i}nik}, \binits{O.A.}},
\bauthor{\bsnm{Shamaev}, \binits{A.S.}},
\bauthor{\bsnm{Yosifian}, \binits{G.A.}}:
\bbtitle{Mathematical Problems in Elasticity and Homogenization}.
\bsertitle{Studies in Mathematics and its Applications},
vol. \bseriesno{26},
p. \bfpage{398}.
\bpublisher{North-Holland Publishing Co.},
\blocation{Amsterdam}
(\byear{1992})
\end{bbook}
\endbibitem

\bibitem{kondratiev-korn}
\begin{barticle}
\bauthor{\bsnm{Kondratiev}, \binits{V.A.}},
\bauthor{\bsnm{Oleinik}, \binits{O.A.}}:
\batitle{On {K}orn's inequalities}.
\bjtitle{C. R. Acad. Sci. Paris S\'{e}r. I Math.}
\bvolume{308}(\bissue{16}),
\bfpage{483}--\blpage{487}
(\byear{1989})
\end{barticle}
\endbibitem

\bibitem{robin-bucur-giacomini}
\begin{barticle}
\bauthor{\bsnm{Bucur}, \binits{D.}},
\bauthor{\bsnm{Giacomini}, \binits{A.}}:
\batitle{Faber-{K}rahn inequalities for the {R}obin-{L}aplacian: a free
  discontinuity approach}.
\bjtitle{Arch. Ration. Mech. Anal.}
\bvolume{218}(\bissue{2}),
\bfpage{757}--\blpage{824}
(\byear{2015})
\end{barticle}
\endbibitem

\bibitem{schneider}
\begin{bbook}
\bauthor{\bsnm{Schneider}, \binits{R.}}:
\bbtitle{Convex Bodies: the {B}runn-{M}inkowski Theory}.
\bpublisher{Cambridge University Press},
\blocation{Cambridge}
(\byear{2014})
\end{bbook}
\endbibitem

\bibitem{horgan-payne}
\begin{barticle}
\bauthor{\bsnm{Horgan}, \binits{C.O.}},
\bauthor{\bsnm{Payne}, \binits{L.E.}}:
\batitle{On inequalities of {K}orn, {F}riedrichs and {B}abu\v{s}ka-{A}ziz}.
\bjtitle{Arch. Rational Mech. Anal.}
\bvolume{82}(\bissue{2}),
\bfpage{165}--\blpage{179}
(\byear{1983})
\end{barticle}
\endbibitem

\bibitem{ciarlet-korn}
\begin{barticle}
\bauthor{\bsnm{Ciarlet}, \binits{P.G.}}:
\batitle{On {K}orn's inequality}.
\bjtitle{Chinese Ann. Math. Ser. B}
\bvolume{31}(\bissue{5}),
\bfpage{607}--\blpage{618}
(\byear{2010})
\end{barticle}
\endbibitem

\bibitem{evans-gariepy}
\begin{bbook}
\bauthor{\bsnm{Evans}, \binits{L.C.}},
\bauthor{\bsnm{Gariepy}, \binits{R.F.}}:
\bbtitle{Measure Theory and Fine Properties of Functions}.
\bsertitle{Textbooks in Mathematics}.
\bpublisher{CRC Press},
\blocation{Boca Raton, FL}
(\byear{2015})
\end{bbook}
\endbibitem

\bibitem{Alves-Martins}
\begin{barticle}
\bauthor{\bsnm{Alves}, \binits{C.J.S.}},
\bauthor{\bsnm{Martins}, \binits{N.F.M.}}:
\batitle{The direct method of fundamental solutions and the inverse
  {K}irsch-{K}ress method for the reconstruction of elastic inclusions or
  cavities}.
\bjtitle{J. Integral Equations Appl.}
\bvolume{21}(\bissue{2}),
\bfpage{153}--\blpage{178}
(\byear{2009})
\end{barticle}
\endbibitem

\bibitem{Chen-Zhou}
\begin{bbook}
\bauthor{\bsnm{Chen}, \binits{G.}},
\bauthor{\bsnm{Zhou}, \binits{J.}}:
\bbtitle{Boundary Element Methods}.
\bsertitle{Computational Mathematics and Applications}.
\bpublisher{Academic Press, Ltd.},
\blocation{London}
(\byear{1992})
\end{bbook}
\endbibitem

\bibitem{mclean}
\begin{bbook}
\bauthor{\bsnm{McLean}, \binits{W.}}:
\bbtitle{Strongly Elliptic Systems and Boundary Integral Equations},
p. \bfpage{357}.
\bpublisher{Cambridge University Press},
\blocation{Cambridge}
(\byear{2000})
\end{bbook}
\endbibitem

\bibitem{Antunes-Bogosel}
\begin{barticle}
\bauthor{\bsnm{Antunes}, \binits{P.R.S.}},
\bauthor{\bsnm{Bogosel}, \binits{B.}}:
\batitle{Parametric shape optimization using the support function}.
\bjtitle{Comput. Optim. Appl.}
\bvolume{82}(\bissue{1}),
\bfpage{107}--\blpage{138}
(\byear{2022})
\end{barticle}
\endbibitem

\bibitem{ak-kao-osting}
\begin{barticle}
\bauthor{\bsnm{Akhmetgaliyev}, \binits{E.}},
\bauthor{\bsnm{Kao}, \binits{C.-Y.}},
\bauthor{\bsnm{Osting}, \binits{B.}}:
\batitle{Computational methods for extremal {S}teklov problems}.
\bjtitle{SIAM J. Control Optim.}
\bvolume{55}(\bissue{2}),
\bfpage{1226}--\blpage{1240}
(\byear{2017})
\end{barticle}
\endbibitem

\bibitem{kao-osting-oudet}
\begin{barticle}
\bauthor{\bsnm{Oudet}, \binits{E.}},
\bauthor{\bsnm{Kao}, \binits{C.-Y.}},
\bauthor{\bsnm{Osting}, \binits{B.}}:
\batitle{Computation of free boundary minimal surfaces {\it via} extremal
  {S}teklov eigenvalue problems}.
\bjtitle{ESAIM Control Optim. Calc. Var.}
\bvolume{27},
\bfpage{34}--\blpage{30}
(\byear{2021})
\end{barticle}
\endbibitem

\bibitem{CDM}
\begin{barticle}
\bauthor{\bsnm{Caubet}, \binits{F.}},
\bauthor{\bsnm{Dambrine}, \binits{M.}},
\bauthor{\bsnm{Mahadevan}, \binits{R.}}:
\batitle{Shape derivative for some eigenvalue functionals in elasticity
  theory}.
\bjtitle{SIAM J. Control Optim.}
\bvolume{59}(\bissue{2}),
\bfpage{1218}--\blpage{1245}
(\byear{2021})
\end{barticle}
\endbibitem

\bibitem{Antunes_illcond}
\begin{botherref}
\oauthor{\bsnm{Antunes}, \binits{P.R.S.}}:
A well conditioned Method of Fundamental Solutions
(2022)
\end{botherref}
\endbibitem

\end{thebibliography}
\bibliographystyle{./sn-mathphys}

\end{document}